\documentclass[a4paper]{amsart}

\usepackage{amsmath, amscd}
\usepackage{amsfonts}
\usepackage{amssymb}
\usepackage{graphicx}
\usepackage{hyperref}
\usepackage{color}%erase
\usepackage{soul,xcolor} %erase
\usepackage{cite}
\usepackage{comment}
\usepackage{tikz-cd}
\usepackage{tikz}
\usetikzlibrary{positioning}

\usepackage{mathtools}

\usepackage[alphabetic,initials]{amsrefs}

\usepackage[
    margin=2.5cm,
    top=2cm,
    bottom=2cm,
    includefoot,
    footskip=20pt,
]{geometry}

\usepackage{amsthm}
\usepackage[all,cmtip]{xy}

\newcommand{\HHom}{\mathcal{H}\kern -.5pt om}

\newcommand{\CC}{\#(Supp(\mathcal{C}) \cap F) }

\newcommand{\QQ}{\mathbb Q}

\newcommand{\OO}{\mathcal O}

\DeclareMathOperator{\N}{N}

\DeclareMathOperator{\codim}{codim}
\DeclareMathOperator{\Hom}{Hom}

\DeclareMathOperator{\Pic}{Pic}
\DeclareMathOperator{\Ext}{Ext}
\DeclareMathOperator{\NE}{\overline{NE}}

\DeclareMathOperator{\rk}{rk}

\DeclareMathOperator{\Nef}{Nef}
\DeclareMathOperator{\Mov}{\overline{Mov}}
\DeclareMathOperator{\Eff}{\overline{Eff}}
\DeclareMathOperator{\Amp}{Amp}

\DeclareMathOperator{\length}{length}

\theoremstyle{plain}
\newtheorem{theorem}{Theorem}[section]
\newtheorem{proposition}[theorem]{Proposition}
\newtheorem{lemma}[theorem]{Lemma}
\newtheorem{corollary}[theorem]{Corollary}

\newtheorem{claim}[theorem]{Claim}

\theoremstyle{definition}

\newtheorem{questions}[theorem]{Questions}

\newtheorem{cor}[theorem]{Corollary}
\newtheorem{defn}[theorem]{Definition}
\newtheorem{thm}[theorem]{Theorem}

\newtheorem{prop}[theorem]{Proposition}

\newtheorem{notn}[theorem]{Notation}
\newtheorem{assump}[theorem]{Assumptions}
\newtheorem{rmk}[theorem]{Remark}

\newcommand{\bi}{\begin{itemize}}  
\newcommand{\ei}{\end{itemize}}
\newcommand{\bp}{\begin{proof}}
\newcommand{\ep}{\end{proof}}

\theoremstyle{remark}

\def\CC{\mathbb{C}}

\def\PP{\mathbb{P}}
\def\QQ{\mathbb{Q}}
\def\RR{\mathbb{R}}
\def\ZZ{\mathbb{Z}}

\def\ov{\overline}

\def\al{\alpha}
\def\be{\beta}
\def\de{\delta}

\def\ga{\gamma}

\def\la{\lambda}

\def\cC{\mathcal{C}}
\def\cD{\mathcal{D}}
\def\cE{\mathcal{E}}

\def\cI{\mathcal{I}}

\def\cL{\mathcal{L}}

\def\cN{\mathcal{N}}
\def\cO{\mathcal{O}}

\def\cU{\mathcal{U}}
\def\cW{\mathcal{W}}

\def\ch{\text{ch}}
\def\codim{\text{codim}}
\def\dim{\text{\rm dim}}

\def\Bl{\text{\rm Bl}}

\def\Ext{\text{\rm Ext}}

\def\Hom{\text{\rm Hom}}
\def\h{\text{h}}
\def\HH{\text{H}}

\def\Pic{\text{\rm Pic}}

\def\dra{\dashrightarrow}

\def\ra{\rightarrow}

\title[Blowups, Gale duality, and moduli spaces]{Blowups, Gale duality, and moduli spaces}

\author[C. Araujo]{Carolina Araujo}
\address{Carolina Araujo, IMPA, Estrada Dona Castorina 110, 22460-320 Rio de Janeiro, Brazil}
\email{caraujo@impa.br}
\author[A.-M. Castravet]{Ana-Maria Castravet}
\address{Ana-Maria Castravet, 
Universit\'e Paris-Saclay, UVSQ, Laboratoire de Math\'ematiques de Versailles, 
45 Avenue des \'Etats Unis, 78035 Versailles, France }
\email{ana-maria.castravet@uvsq.fr}
\author[I. Kaur]{Inder Kaur}
 \address{Inder Kaur, School of Mathematics \& Statistics, University of Glasgow, Glasgow G12 8QQ, U.K.}
 \email{inder.kaur@glasgow.ac.uk}
\author[D. Martinelli]{Diletta Martinelli}
\address{Diletta Martinelli, KdV Institute for Mathematics, University of Amsterdam, P.O.Box 94248, 1090 GE Amsterdam, Netherlands
}
\email{d.martinelli@uva.nl}

\begin{document}

\begin{abstract}
The goal of this paper is to describe the birational geometry of the blowup of $\PP^n$ at $n+4$ points in very general position.
To achieve this, we follow an idea of Mukai and explore a special instance of Gale duality, namely, a correspondence between configurations of $n+4$ points in the projective spaces $\mathbb{P}^n$ and $\mathbb{P}^2$.
We first prove that the blowup $X$ of $\mathbb{P}^n$ at $n+4$ very general points is isomorphic to a certain Gieseker moduli space of rank $2$ vector bundles on the surface $S$ obtained by blowing up $\mathbb{P}^2$ at the $n+4$ Gale dual points. 
We then study the variation of these moduli spaces as we vary the polarization $L$ on $S$, and translate this variation into a partial Mori chamber decomposition of $\Eff(X)$, describing  to some extent the birational geometry of $X$.
\end{abstract}

\subjclass[2020]{}
\keywords{}

\maketitle

\setcounter{tocdepth}{1}
\tableofcontents

\section{Introduction}\label{introduction}

Blowups of projective spaces at finite sets of points arise frequently in algebraic geometry and exhibit rich and interesting geometry.  When only a few points in general linear position are blown up, the birational geometry of the resulting varieties can be completely described by finite combinatorial data. However, the geometry becomes increasingly intricate as more points are blown up.
The goal of this paper is to describe the birational geometry of the blowup of $\PP^n$ at $n+4$ points in very general position.
By this, we mean describing its rational contractions, of which there may be infinitely many.
We note that, for $n\geq 5$, the minimal number of points for which the blowup of $\PP^n$ admits infinitely many rational contractions is precisely $n+4$.

The increasing complexity of the birational geometry as more points are blown up is nicely illustrated by the classical case of blowups of the projective plane.
The blowup of $\PP^2$ at up to three points in general linear position is a toric surface. If one blows up $k\geq 4$ points in the plane, the resulting surface is no longer toric. However, if $k\leq 8$ and the points are in general position, the resulting surface $S$ is a del Pezzo surface, and its birational geometry can still be completely described by finite data. In particular, $S$ contains finitely many $(-1)$-curves, each of which corresponds to a distinct blow-down morphism from $S$. On the other hand, as soon as $k\geq 9$ and the points are in sufficiently general position, the resulting surface has infinitely many $(-1)$-curves (\cite{Nagata}). 
To be precise, we say that a set of points in the plane is \emph{Cremona-general} if the points are in general linear position, i.e., no three points are collinear, and remain so after any finite sequence of Cremona transformations centered at three of the points (see Definition~\ref{gen_position}). 
It is important to note that the generality assumption on the position of the points is essential in Nagata's theorem.
For example, if an arbitrary number of points in the plane are blown up, but all lie on an irreducible conic, then the resulting surface contains only finitely many  $(-1)$-curves, and its geometry can be described by finite data (\cite[Theorem 1.2]{CT}). In this case, a Cremona transformation centered at three of the points maps the remaining points onto points lying on a single line.

The finiteness of $(-1)$-curves on blowups of $\PP^2$ at Cremona-general points is more naturally generalized to higher dimensions in the language of Mori dream spaces, which were introduced by Hu and Keel in \cite{HK}.
Before introducing them, let us review some important convex cones associated to a normal projective variety $X$. 
Recall that the N\'eron-Severi space $\N^1(X)$ of $X$ is the finite-dimensional real vector space of Cartier divisors modulo numerical equivalence,
$$
\N^1(X) = \big( \Pic(X)/\equiv\big)\otimes_{\ZZ}\RR.
$$
The \emph{nef cone} $\Nef(X)$, the \emph{movable cone} $\Mov(X)$, and the \emph{pseudoeffective cone} $\Eff(X)$ of $X$ are the closed convex cones in $\N^1(X)$ generated by classes of nef divisors, movable divisors (i.e., nonzero effective divisors without a divisorial component in their base loci), and effective divisors, respectively. The nef cone is the closure of the \emph{ample cone} $\Amp(X)$, which is generated by classes of ample divisors, and there are inclusions
$$
\overline{\Amp}(X) \ = \ \Nef(X) \ \subseteq \ \Mov(X) \ \subseteq \ \Eff(X)  \ \subseteq \ \N^1(X).
$$
We also recall that a \emph{small modification} is a birational map that is an isomorphism in codimension $1$.

A normal $\QQ$-factorial projective variety $X$ is a \emph{Mori dream space} if $\Pic(X)$ is finitely generated and the following conditions hold (see \cite[Definition 1.10 and Proposition 1.11]{HK}): 
\begin{enumerate}
\item The pseudoeffective cone $\Eff(X)$ of $X$ is generated by the classes of finitely many prime divisors.
\item The nef cone $\Nef(X)$ of $X$ is generated by the classes of finitely many semi-ample divisors.
\item There is a finite collection of small $\QQ$-factorial modifications
 $f_i: X \dasharrow X_i$, such that each $X_i$ satisfies (2) above, and the movable cone $\Mov(X)$ of $X$ satisfies
 $$
 \Mov(X) \ = \ \bigcup_i \  f_i^*(\Nef(X_i)).
 $$
\end{enumerate}
Mori dream spaces are characterized by the algebraic condition of finite generation of their Cox ring, and their birational geometry is encoded in a finite polyhedral structure: the pseudoeffective cone  $\Eff(X)$,
together with a decomposition into subcones, each corresponding to a rational contraction from $X$. 
This decomposition is called the \emph{Mori chamber decomposition} of $\Eff(X)$, and its restriction to $\Mov(X)$ is the decomposition described in (3) above.
By \cite{Batyrev-Popov}, the blowup $S$ of $\PP^2$ at $k$ Cremona-general points is a Mori dream space if and only if $k \leq 8$.

We now define the notion of \emph{Cremona-general} points in $\PP^n$ for $n\geq 2$, and introduce the action of the Weyl group, which plays a key role in our study of the birational geometry of blowups of projective spaces.

\begin{notn}\label{notation_Pic}
Let $P_1, \dots, P_{k}\in \PP^n$, with $k\geq n+1$, be distinct points and denote by $\pi:X\to \PP^n$ the blowup at this set of points. 
We denote by $E_i\subset X$ the exceptional divisor over $P_i$, as well as its class in $\Pic(X)$. 
We denote by $H\in \Pic(X)$ the pullback of the hyperplane class of $\PP^n$.
Then 
$$
\Pic(X) \ = \ \ZZ H \oplus \bigoplus_{i\in\{1, \dots,k\}}\ZZ E_i.
$$
For each subset $I\subset \{1, \dots,k\}$ of cardinality $n+1$, we denote by $X_I$ the blowup of $\PP^n$ at the points $\{P_{i}\}_{i\in I}$.
There is a natural birational morphism $X\to X_I$, and a natural isomorphism 
$$
\Pic(X)\ \cong \ \Pic(X_I) \oplus \bigoplus_{j\in\{1, \dots,k\}\setminus I}\ZZ E_j.
$$
Assume that the points are in general linear position, i.e., no $n+1$ of them are contained in a hyperplane.
For each subset $I\subset \{1, \dots,k\}$ of cardinality $n+1$, we denote by 
$$
\varphi_I:\PP^n\dashrightarrow\PP^n
$$
a standard Cremona transformation centered at the set points $\{P_i\}_{i\in I}$. 
This is a birational involution of $\PP^n$ that lifts to a small modification 
$\tilde \varphi_I:X_I\dashrightarrow X_I$, and thus
induces an involution  $\tilde\varphi_I^*:\Pic(X_I)\to\Pic(X_I)$.
\end{notn}

\begin{defn}[{Cremona-general points}]\label{gen_position}
Let $P=\big(P_1, \dots, P_{k}\big)$ be a set of marked points in general linear position in $\PP^n$, with $k\geq n+1$, up to projective transformation. 
For each subset $I\subset \{1, \dots,k\}$ of cardinality $n+1$, let $\varphi_I:\PP^n\dashrightarrow\PP^n$ be as in Notation~\ref{notation_Pic}. 
We define a new set of marked points $I_*P=\big(P_1^I, \dots, P_{k}^I\big)$ in $\PP^n$, up to projective transformation, as follows: 
$$
P_j^I \ = \  
	       \begin{cases}
P_j  & \text{ if } \ j\in I, \\
\varphi_I(P_j) & \text{ if } \ j\not\in I. 
\end{cases}
$$
We say that  the points $P_1, \dots, P_{k}$ are \emph{Cremona-general} if they are in general linear position and, for any finite sequence 
$I_1, \dots ,I_m$ of subsets  of cardinality $n+1$ of $\{1, \dots,k\}$, the new set of marked points $(I_m)_* \dots (I_1)_*P$ is still in general linear position.
\end{defn}

Throughout this paper, we denote by $\Bl_k\PP^n$ the blowup of $\PP^n$ at $k$  Cremona-general points.

\begin{defn}[{The Weyl group action}]\label{Weyl}
Let $\Bl_k\PP^n$ the blowup of $\PP^n$ at $k\geq n+1$ Cremona-general points $P_1, \dots, P_{k}$. 
For each subset $I\subset \{1, \dots,k\}$ of cardinality $n+1$, let $\varphi_I:\PP^n\dashrightarrow\PP^n$ be as in Notation~\ref{notation_Pic}.
We extend the involution $\tilde\varphi_I^*:\Pic(X_I)\to\Pic(X_I)$ to an involution 
$$
\varphi_I^*:\Pic(\Bl_k\PP^n)\to\Pic(\Bl_k\PP^n)
$$
by setting $\varphi_I^*(E_j)=E_j$ for $j\in\{1, \dots,k\}\setminus I$.
Clearly, $\varphi_I^*\big(K_{\Bl_k\PP^n}\big)=K_{\Bl_k\PP^n}$.
We denote by the same symbol the corresponding linear involution $\varphi_I^*:\N^1(\Bl_k\PP^n)\to \N^1(\Bl_k\PP^n)$, and  
define the group:
$$
\cW_{n,k} := \big\langle \ \varphi_I^* \ | \ I\subset \{1, \dots,k\} \text{ with } |I|=n+1 \ \big\rangle \subset GL\big(\N^1(\Bl_k\PP^n)\big).
$$
\end{defn}

When $k> n+1$, the group  $\cW_{n,k}$ admits an alternative set of generators: 
$$
\cW_{n,k} = \langle s_0, s_1, \dots, s_{k-1}\rangle,
$$
where $s_0=\varphi_{\{1,\dots, n+1\}}^*$ and,  for $1\leq i\leq k-1$, $s_i$ swaps $E_i$ and $E_{i+1}$ and fixes the other basis elements.
We refer to Section~\ref{subsection_blowups} for details.
As a consequence, we see that $\cW_{n,k}$ is isomorphic to the Coxeter group associated to the Dynkin diagram $T_{n+1, k-n-1,2}$:
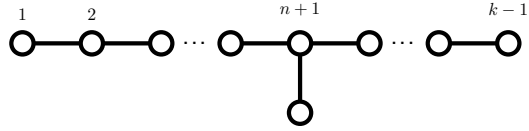
\begin{center}
\begin{tikzpicture}[
    transform shape,
    scale=0.6,
    node font=\large\bfseries\sffamily,
    every node/.append style={circle,minimum size=7mm,line width=1.6pt},
    every edge/.append style={ultra thick},
    d/.style={draw},
    ]
    \node[d,label={$1$},scale=0.7] (1) {};
    \node[d,right=of 1, label={$2$},scale=0.7] (2) {} edge (1);
    \node[d,right=of 2,scale=0.7] (3) {} edge (2);
    \node[d,right=of 3,scale=0.7] (n) {};
    \path (3) -- node[auto=false]{\ldots} (n);
    \node[d,right=of n,label={[label distance=-2mm]${n+1}$},scale=0.7] (n+1) {} edge (n);
    \node[d,right=of n+1,scale=0.7] (n+2) {} edge (n+1);
    \node[d,right=of n+2,scale=0.7] (k-2) {};
    \path (n+2) -- node[auto=false]{\ldots} (k-2);
    \node[d,right=of k-2, label={[label distance=-2mm]${k-1}$},scale=0.7] (k-1) {} edge (k-2);
    \node[d,below=of n+1,scale=0.7] (0) {} edge (n+1);
\end{tikzpicture}
\end{center}
Note that when $k=n+1$, the group $\cW_{n,n+1}$ defined as above (a cyclic group of order $2$) differs from the Coxeter group $\langle s_0, s_1, \dots, s_{n}\rangle$ (with the $s_i$ defined as above) associated to the diagram $A_n+A_1$. 

The action of the Weyl group on $\Pic(\Bl_k\PP^n)$ and $\N^1(\Bl_k\PP^n)$ has been well explored as a tool for describing the birational geometry of  $\Bl_k\PP^n$. 
The study of this action can be traced back to the work of Coble (\cite{Coble1}, \cite{Coble2}, \cite{Coble3}).
We refer to \cite{Dolgachev} and \cite{DO} for a historical discussion together with a more modern account of this theory.
In \cite{Mukai01}, Mukai used the action of the Weyl group to show that $\Bl_k\PP^n$ is not a Mori dream space when $\cW_{n,k}$ is infinite.
Indeed, from its definition, we see that the action of $\cW_{n,k}$ on $\N^1(\Bl_k\PP^n)$ preserves the pseudoeffective cone $\Eff(\Bl_k\PP^n)$ (see Remark~\ref{rmk:cremona_action}).
Therefore, when $\cW_{n,k}$ is infinite, the pseudoeffective cone $\Eff(\Bl_k\PP^n)$ has infinitely many extremal rays, namely, one for each element $D$ in the $\cW_{n,k}$-orbit of $E_1$.
In fact, the works \cite{Mukai01} and \cite{CT} together establish that $\Bl_k\PP^n$ is a Mori dream space exactly when the Weyl group $\cW_{n,k}$ is finite:
$$
\Bl_k\PP^n \ \text{  is a Mori dream space } \ \iff \ 
	       \begin{cases}
n=2, \ k\le 8; \text{  or } \\
n=3, \ k\le 7; \text{  or } \\
n=4, \ k\le 8; \text{  or } \\
n\ge 5, \ k\le n+3. 
\end{cases}
$$
The action of the Weyl group $\cW_{3,8}$ was explored in \cite{LO16} and \cite{SX23} to describe the geometry of the blowup of $\PP^3$ at $8$ very general points. 

From the action of the Weyl group $\cW_{n,k}$ on $\Pic(\Bl_k\PP^n)$, one can extract a bilinear form on $\Pic(\Bl_{n+4}\PP^n)$, which coincides with the intersection product on surfaces when $n=2$. 
For $n>2$, this pairing appears explicitly in \cite[Section 3]{Mukai01}.
We discuss it in more detail and characterize it with respect to the action of the Weyl group in Section~\ref{subsection_blowups}.
In reference to Coble’s work, and following a suggestion of Igor Dolgachev, we call it the \emph{Coble pairing}.

\begin{defn}\label{Coble}
Let $\{H, E_1, \dots, E_{k}\}$ be the basis of $\Pic(\Bl_k\PP^n)$ introduced in Notation~\ref{notation_Pic} above. 
We define the Coble pairing on $\Pic(\Bl_k\PP^n)$ by:
\begin{itemize} 
	       \item  $(H,H)=n-1$;
	       \item  $(H, E_i) = 0$; and 
           \item  $(E_i, E_j) = -\delta_{ij}$.
\end{itemize} 
\end{defn}

\begin{notn}
For any divisor $D\in \Pic(\Bl_k\PP^n)$, the condition $(D> 0)$ with respect to the Coble pairing defines a half-space in $\N^1(\Bl_k\PP^n)$. 
Given any cone $\N\subseteq\N^1(\Bl_k\PP^n)$, 
we write $\N^{D> 0}$ for the intersection of $\N$ with this half-space, and similarly for $\N^{D< 0}$, $\N^{D\leq 0}$ and $\N^{D\geq 0}$.
\end{notn}

As we mentioned above, each element $D$ in the $\cW_{n,k}$-orbit of $E_1$ generates an extremal ray of the pseudoeffective cone $\Eff(\Bl_k\PP^n)$.
Moreover, any such $D$ satisfies $\big(K_{\Bl_k\PP^n},D\big)<0$. 
Conversely, it is natural to ask whether these constitute all the extremal rays of $\Eff(\Bl_k\PP^n)$ lying in the half-space $\big(K_{\Bl_k\PP^n}<0\big)$, as is the case for surfaces.

\begin{questions}[{\cite[Question 7.9]{BDPS}}] \label{questions}
Suppose that $n\geq 3$ and $k\geq n+4$, and let $X=\Bl_k\PP^n$.
\begin{enumerate}
\item Are all the extremal rays $R\subset \Eff(X)^{K<0}$ of the form $R=\RR_{\geq0}E$ for some $E\in \cW_{n,k}\cdot E_1$?
\item Is there a version of the Mori chamber decomposition on $\Eff(X)^{K<0}$ ?
\end{enumerate}
\end{questions}

When $n=3$ and $k=8$, $-K_{\Bl_8\PP^3}$ is nef and positive answers to Questions~\ref{questions} were given in \cite{SX23}. 

In this paper, we provide positive answers to Questions~\ref{questions} when $k=n+4$, and give a description of the birational geometry of $\Bl_{n+4}\PP^n$. Our main result is the following.

\begin{theorem}\label{thm:main}
	Let $X=\Bl_{n+4}\PP^n$ be the blowup of $\PP^n$ at $n+4$ Cremona-general points. 
    Then the following holds.
\begin{enumerate}
\item $\Eff(X)^{K\leq 0} \ = \ \overline{\displaystyle\sum_{E\in \cW_{n,k}\cdot E_1}\RR_{\geq0}E}$. 
\item There are countably many smooth small modifications $g_i \colon X \dashrightarrow Y_i$  such that
	$$\Mov(X)^{K\leq0}\  \ = \  \ \bigcup_i \  g_i^*(\Nef(Y_i)).$$
Moreover, if $g_i^*(\Nef(Y_i))\setminus\{0\}\subseteq (K_{X}<0)$, then the nef cone $\Nef(Y_i)$ is polyhedral. 
 \end{enumerate}    
 \end{theorem}

In order to prove Theorem~\ref{thm:main}, we follow an idea of Mukai and explore a special instance of \emph{Gale duality}.
Gale duality is a correspondence between sets of $k=n+s+2$ points in the projective spaces $\mathbb{P}^n$ and $\mathbb{P}^s$, considered up to projective equivalence.
This duality is defined purely in terms of linear algebra, but it often presents remarkable geometric manifestations. 
We refer to Section~\ref{subsection:gale} for the precise definition of Gale duality and its properties. 

In the simplest case, when $s=1$, Gale duality goes as follows: given $n+3$ marked points in $\mathbb{P}^n$, there is a unique rational normal curve  $C\subset \mathbb{P}^n$ of degree $n$ passing through these points. This determines a configuration of $n+3$ marked points in $\mathbb{P}^1\cong C$. 
Conversely, given $n+3$ marked points in $\mathbb{P}^1$, 
one embeds $\mathbb{P}^1$ as a rational normal curve of degree $n$ in $\mathbb{P}^n$, producing a corresponding configuration of 
$n+3$ marked points in $\mathbb{P}^n$. This correspondence was used in \cite{Mukai} to describe the Mori chamber decomposition of $\Eff(\Bl_{n+3}\PP^n)$. Mukai's idea goes as follows. 
In \cite{Bauer}, Bauer described the moduli spaces of semistable rank $2$ parabolic vector bundles on $\mathbb{P}^1$ with $n+3$ marked points in general position. 
He showed that, for a particular choice of weights, the moduli space is isomorphic to the blowup of $\PP^n$ at the $n+3$ Gale dual points. 
Moreover, he described explicitly how the stability conditions and the moduli spaces vary as one varies the weights. 
Mukai then observed that the weight space can be seen as a cross-section of the effective cone $\Eff(\Bl_{n+3}\PP^n)$, and that the stability chamber decomposition of the weight space coincides with the one induced by the Mori chamber decomposition of $\Eff(\Bl_{n+3}\PP^n)$.
This explicit description of the Mori chamber decomposition has been effectively used for studying the geometry of $\Bl_{n+3}\PP^n$ and its small modifications (see \cite{Mukai}, \cite{AM16}, \cite{BDP16}, \cite{AC17}, \cite{AFKM21}). 

We will exploit the next case, $s=2$, which gives a correspondence between sets of $n+4$ points in $\PP^2$, and $n+4$ points in $\mathbb{P}^n$, up to projective equivalence. In the special case when $n=4$, the blowup $X$ of $\PP^4$ at $8$ points in general position is a Mori dream space, and the blowup $S$ of $\PP^2$ at the $8$ Gale dual points is a del Pezzo surface. As outlined in \cite{Mukai} and worked out in detail in \cite{CCF}, $X$ can be realized as the moduli space of Gieseker semistable torsion free sheaves on $S$, for a particular choice of Chern classes, and a suitable choice of polarization $L\in \Amp(S)$. The variation of these moduli spaces under changes of polarization on a del Pezzo surface was extensively investigated in the 1990s, and it is governed by a wall-and-chamber decomposition of its ample cone.
The stability chamber decomposition of $\Amp(S)$ then translates to the Mori chamber decomposition of $\Eff(X)$.
We show that this strategy can be carried out for any value of $n\geq 2$. 
Our starting point to prove Theorem~\ref{thm:main} is the following (see Theorem~\ref{X as moduli}). 

\begin{theorem} \label{thm:moduli}
	Let $X=\Bl_{n+4}\PP^n$ be the blowup of $\PP^n$ at $n+4$ Cremona-general points, $n\geq 2$, and denote by $S$ the blowup of $\PP^2$ at the $n+4$ Gale dual points.
	Then there exists an ample line bundle $L$ on $S$, explicitly determined, such that $X$ is isomorphic to the moduli space of Gieseker $L$-semistable torsion-free sheaves $F$ on $S$ with $\rk(F)=2$, $\text{c}_1(F) = -K_S$, $\text{c}_2(F) = 2$.
\end{theorem} 
Having realized $X$ as a moduli space of torsion-free sheaves on $S$, we can then exploit the classical theory of moduli spaces of sheaves on surfaces to investigate the birational geometry of X and prove Theorem~\ref{thm:main}. Specifically, there is a wall-and-chamber decomposition of the ample cone $\Amp(S)$ associated with different notions of stability. Unlike the case of del Pezzo surfaces, however, one cannot in general fully describe the variation in the corresponding moduli spaces under changes of polarization from this decomposition. This can nevertheless be done effectively on the subcone $\Amp(S)^{K\leq 0}$. We show that a special linear map, called the \emph{determinant map},
$$
\rho\colon \N^1(S)\to \N^1(X),
$$
maps the stability chamber decomposition of the subcone $\Amp(S)^{K\leq 0}$ to the Mori chamber decomposition of the subcone $\Eff(X)^{K\leq 0}$.
We note that, although our primary goal is to describe the birational geometry of blowups $X=\Bl_{n+4}\PP^n$ that are not Mori dream spaces, namely those  with $n \geq 5$, Theorems~\ref{thm:moduli} and \ref{thm:main} provide a unified treatment of all these blowups. When $2 \leq n \leq 4$, $S=\Bl_{n+4}\PP^2$ is a del Pezzo surface and $X=\Bl_{n+4}\PP^n$ is a Mori dream space. In this case, the determinant map yields a complete description of the Mori chamber decomposition of $\Eff(X)$ in terms of the stability chamber decomposition of $\Amp(S)$.

We expect that some manifestation of Gale duality similar to that described in Theorem~\ref{thm:moduli} holds for all values of $s,n\geq 2$.  Remark~\ref{rem:expectation} provides evidence for this expectation. What is less clear is how to use the realization of $\Bl_{n+s+2}\PP^n$ as a moduli space of torsion-free sheaves on $\Bl_{n+s+2}\PP^s$ to determine $\Eff(\Bl_{n+s+2}\PP^n)$ and a possible Mori chamber decomposition of $\Eff(\Bl_{n+s+2}\PP^n)^{K\leq 0}$.

This paper is organized as follows. 
In Section~\ref{section:Blowups}, we discuss in detail blowups of projective spaces, the action of the Weyl group, and Gale duality. 
In Section~\ref{section:moduli}, we review the classical theory of moduli spaces of Gieseker semistable torsion-free sheaves on a surface $S$, and the stability chamber decomposition of the ample cone $\Amp(S)$, focusing on the special case in which $S$ is the blowup of $\PP^2$ at a finite set of points. 
In Section~\ref{section:blowups2}, we revisit blowups of the projective plane at Cremona-general points in light of the stabilty chamber decomposition and prove Theorem~\ref{thm:moduli}. 
In Section~\ref{section:determinant}, we introduce the determinant map and describe it explicitly in our setting. 
In Section~\ref{section:proof_main_thm}, we put all these results together to prove Theorem~\ref{thm:main}.

We always work over $\CC$. 

\bigskip

\noindent {\bf Acknowledgements.} 
This work was partially supported by the ``CAPES-COFECUB programme'' (project number: Ma 1017/24), funded by the French Ministry for Europe and Foreign Affairs, the French Ministry for Higher Education and Coordena\c c\~ao de Aperfei\c coamento de Pessoal de N\'ivel Superior – Brazil (CAPES), and the ``Brazilian-French Network in Mathematics''.
Carolina Araujo was partially supported by grants from CNPq and Faperj. Ana-Maria Castravet was partially supported by ANR grant FanoHK and the Institut Universitaire de France. During the writing of this project, Inder Kaur was partially supported by DFG, TRR 326 Geometry and Arithmetic of Uniformized Structures, project number 444845124, an EPSRC grant EP/W026554/1 and a CNRS ``poste rouge" at IRMAR Rennes. We thank Cinzia Casagrande, Igor Dolgachev, Olivia Dumitrescu, Andrea Fanelli, and Elisa Postinghel for the many enlightening discussions on the topics treated in this paper. We also thank the audiences of the various talks related to this work for their insightful questions and valuable comments.

%%%%%%%%%%%%%%%%%%%%%%%%%%%%%%%%%%%%%%%%%%%%%%%%%%%%%%%%%%%%%%%%%%%%%%%%%

\section{Blowups and Gale duality}\label{section:Blowups}
In this section, we discuss blowups of projective spaces, the action of the Weyl group, and Gale duality, and fix the notation used throughout the paper. In Subsection~\ref{subsection_blowups}, we study the action of the Weyl group on the Picard group of the blowup of $\PP^n$ at Cremona-general points, and characterize the Coble pairing with respect to this action.
In Subsection~\ref{subsection:gale}, we introduce Gale duality, explain how it can be used to realize blowups of projective spaces as moduli spaces of torsion-free sheaves, and relate it to the actions of the Weyl group.
In Subsection~\ref{subsection:blowups1}, we focus on blowups of the projective plane and review what is known about their nef and pseudoeffective cones.

\subsection{Blowups and the action of the Weyl group} \label{subsection_blowups}

Let $\Bl_k\PP^n$ be the blowup of $\PP^n$ at $k$ Cremona-general points $P_1,\dots, P_{k}$, and follow  Notation~\ref{notation_Pic}:
$$\Pic(\Bl_k\PP^n) \ = \ \ZZ H \oplus \bigoplus_{i\in\{1, \dots,k\}}\ZZ E_i \ , \quad K_{\Bl_k\PP^n} \ = \ -(n+1)H + \sum_{i=1}^{k} (n-1) E_i.$$
Recall the action of the Weyl group introduced in Definition~\ref{Weyl}:
$$
\cW_{n,k} := \big\langle \ \varphi_I^* \ | \ I\subset \{1, \dots,k\} \text{ with } |I|=n+1 \ \big\rangle \subset GL\big(\N^1(\Bl_k\PP^n)\big).
$$
When $k> n+1$, the group $\cW_{n,k}$ contains the permutation group on $k$ letters, corresponding to the permutations of the $E_i$'s.
To see this, for each subset $J\subset \{1, \dots,k\}$ of cardinality $n$, and $i\in \{1, \dots,k\}\setminus J$, set 
$\varphi_{Ji}:=\varphi_{J\cup\{i\}}^*:\N^1(\Bl_k\PP^n)\to \N^1(\Bl_k\PP^n)$, and for each pair of indices $i, j\in \{1, \dots,k\}\setminus J$, $i\neq j$, let $\tau_{i,j}:\N^1(\Bl_k\PP^n)\to \N^1(\Bl_k\PP^n)$ be the isomorphism that swaps $E_i$ and $E_j$ and fixes the other basis elements.
One can easily check that $\varphi_{Ji}\cdot \varphi_{Jj}\cdot \varphi_{Ji}=\tau_{i,j}$. 
So we can take the following alternative set of generators: 
$$
\cW_{n,k} = \langle s_0, s_1, \dots, s_{k-1}\rangle \subset GL\big(\N^1(\Bl_k\PP^n)\big),
$$
where $s_0=\varphi_{\{1,\dots, n+1\}}^*$ and $s_i=\tau_{i,(i+1)}$ for each $i\in \{1,\dots, k-1\}$. 
Up to scaling, the canonical class $K_{\Bl_k\PP^n}$ is the unique nonzero vector fixed by the action of $\cW_{n,k}$.

\begin{rmk} \label{rmk:cremona_action}
While Definition~\ref{Weyl} makes sense as long as the $k$ points of $\PP^n$  are in general linear position, its action on the Picard group is only geometrically meaningful when they are Cremona-general. 
Indeed, suppose that the set of marked points $P^0=P=\big(P_1, \dots, P_{k}\big)$ in $\PP^n$ is Cremona-general. 
For any finite sequence $I_1, \dots ,I_m$ of subsets  of cardinality $n+1$ of $\{1, \dots,k\}$, consider the corresponding sequence of sets of marked points
$P^j=\big(P^j_1, \dots, P^j_{k}\big)=(I_j)_* \dots (I_1)_*P$, for $j=1,\dots,m$, and the standard Cremona transformation $\varphi_j:\PP^n\dashrightarrow\PP^n$ centered at the set of points $\{P^{j-1}_i\}_{i\in I_j}$. 
The composed Cremona transformation $\varphi=\varphi_m\circ\ldots\circ \varphi_1:\PP^n\dashrightarrow\PP^n$ induces a small modification between the blowup $X_P$ of $\PP^n$ at $P$ and the blowup $X_{P'}$ of $\PP^n$ at the new set of marked points $P'=P^m=(I_m)_* \dots (I_1)_*P$, and thus induces an isomorphism 
between $\Pic(X_P)$ and $\Pic(X_{P'})$. 
The element $\varphi_{I_1}^*\circ\ldots\circ \varphi_{I_m}^*\in \cW_{n,k}$ is then the composition of this induced isomorphism with the 
natural identification of $\Pic(X_P)$ and $\Pic(X_{P'})$ obtained by identifying the natural bases $\{H, E_1, \dots, E_k\}$ and $\{H', E'_1, \dots, E'_k\}$.
In particular, this map sends classes of effective divisors to classes of effective divisors.
It follows that the action of $\cW_{n,k}$ on $\N^1(X)$ preserves the pseudoeffective cone $\Eff(\Bl_k\PP^n)$.
Similarly, it also preserves the movable cone $\Mov{(\Bl_k\PP^n)}$.
We refer to \cite{Dolgachev} and \cite{DO} for more details on the Weyl group $\cW_{n,k}$ and its action on $\Pic(\Bl_k\PP^n)$ and $\N^1(\Bl_k\PP^n)$.
\end{rmk}

Recall the Coble pairing on $\Pic(\Bl_k\PP^n)$ introduced in Definition~\ref{Coble}:
\begin{itemize} 
	       \item  $(H,H)=n-1$;
	       \item  $(H, E_i) = 0$; and 
           \item  $(E_i, E_j) = -\delta_{ij}$.
\end{itemize} 
The justification for this pairing, which coincides with the intersection product on surfaces when $n=2$, is that it endows $\Pic(\Bl_k\PP^n)$ with a lattice structure under which $\cW_{n,k}$ is a group generated by simple reflections through $(-2)$-roots.
Indeed, set $\alpha_0=H-E_1-\dots -E_{n+1}$ and $\alpha_i=E_i-E_{i+1}$ for $1\leq i\leq k-1$.
Then, for each $i\in\{0,\dots, k-1\}$,  $\alpha_i\in K_{\Bl_k\PP^n}^\perp$, $\alpha_i^2=-2$ and $s_i$ is precisely the reflection through the simple root $\alpha_i$.
The Dynkin diagram of this root system is depicted below.
$$
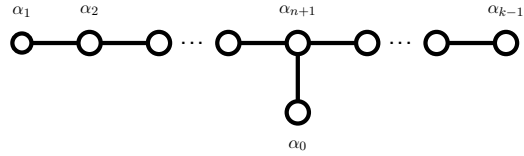
\begin{tikzpicture}[
    transform shape,
    scale=0.6,
    node font=\large\bfseries\sffamily,
    every node/.append style={circle,minimum size=7mm,line width=1.6pt},
    every edge/.append style={ultra thick},
    d/.style={draw},
    ]
    \node[d,label={$\alpha_1$},scale=0.6] (1) {};
    \node[d,right=of 1, label={$\alpha_2$},scale=0.7] (2) {} edge (1);
    \node[d,right=of 2,scale=0.7] (3) {} edge (2);
    \node[d,right=of 3,scale=0.7] (n) {};
    \path (3) -- node[auto=false]{\ldots} (n);
    \node[d,right=of n,label={[label distance=-2mm]$\alpha_{n+1}$},scale=0.7] (n+1) {} edge (n);
    \node[d,right=of n+1,scale=0.7] (n+2) {} edge (n+1);
    \node[d,right=of n+2,scale=0.7] (k-2) {};
    \path (n+2) -- node[auto=false]{\ldots} (k-2);
ar@adblo\node[d,right=of k-2, label={[label distance=-2mm]$\alpha_{k-1}$},scale=0.7] (k-1) {} edge (k-2);
    \node[d,below=of n+1, label=below:{$\alpha_0$},scale=0.7] (0) {} edge (n+1);
\end{tikzpicture}
$$

While the Coble pairing is not the unique one with this property, it is the only one that satisfies the additional property that $(E_i, E_j) = 0$ for $i\neq j$. This is a consequence of the following lemma. 
Although this result is not needed in the sequel, we include it for the sake of completeness.

\begin{lemma}
Suppose that $n\geq 2$ and $k> n+1.$
Let $(\cdot,\cdot):\Pic(\Bl_k\PP^n)\times\Pic(\Bl_k\PP^n)\ra\ZZ$ be a symmetric bilinear form with
\begin{itemize} 
	       \item  $(H,H)=a$ ;
	       \item  $(H, E_i)=b$ ; 
           \item  $(E_i, E_j) = \begin{cases}
c\quad \text{if} \quad i\neq j \ ;  \\
d\quad \text{if} \quad i= j \ ;
\end{cases}$
\end{itemize} 
for some $a,b,c,d\in\ZZ$.  
Suppose that, for each subset $I\subset\{1,\ldots,k\}$ of cardinality $|I|=n+1$, there exists $\al_I\in\Pic(X)$ such that $(\al_I,\al_I)=-2$ and the map $\varphi^*_I\in \cW_{n,k}$ is the reflection across the hyperplane $\al_I^\perp$. Then:
$$a=(n+1)^2c+(n-1),\quad b=(n+1)c\quad \text{and}\quad d=c-1.$$
In particular, if $c=0$ then $(\cdot,\cdot)$ is the Coble pairing from Definition \ref{Coble}. 
\end{lemma}

\bp
For each subset $I\subset\{1,\ldots,k\}$ of cardinality $n+1$, 
let $\al_I\in\Pic(X)$ be such that $(\al_I,\al_I)=-2$ and the map $\varphi^*_I$ is the reflection $r_{\alpha_I}$ across the hyperplane $\al_I^\perp$.
Since $(\al_I,\al_I)=-2$, the reflection $r_{\alpha_I}$ is given by 
$$r_{\alpha_I}(x)=x+(\alpha_I,x)\alpha_I.$$
Therefore, for any $i\in I$, we have: 
$$
H-\sum_{j\in I\setminus\{i\}}E_j \ = \  \varphi_I^*(E_i) \ = \ r_{\alpha_I}(E_i) \ = \ E_i+(\al_I,E_i)\al_I.
$$
It follows that 
$$
\al_I=\frac{1}{(\al_I,E_i)}\big(H-\sum_{j\in I}E_j\big),\ \text{for all}\ i\in I.
$$
Since $\al_I\in Pic(X)$, we must have
$(\al_I,E_i)=\pm1 \ \ \forall i\in I$ and we may assume without loss of generality that 
$$\al_I=H-\sum_{i\in I}E_i \quad \text{and} \quad (\al_I,E_i)=1\ \ \forall  i\in I.$$
It follows that $b-d-n\cdot c=1$. 
Similarly, for any $j\in\{1, \dots, k\}\setminus I$, we have: 
$$
E_j \ = \ \varphi_I^*(E_j) \ = \ r_{\alpha_I}(E_j)  \ = \ E_j+(\al_I,E_j)\al_I.
$$
It follows that $b=(n+1)c$. Together with $b-d-n\cdot c=1$, this implies that $d=c-1$. 
From $(\al_I,\al_I)=-2$ and the above relations, we get that $a=(n+1)^2c+(n-1)$.
\ep

\subsection{Gale duality}\label{subsection:gale}
Let $n$ and $s$ be positive integers, and set $k:=n+s+2$. 
In this subsection we introduce Gale duality, a bijective correspondence between configurations of $k$ points in $\PP^s$ and  configurations of $k$ points  in $\PP^n$, considered up to projective equivalence. 
Our main references are \cite{DO} and \cite{EP}. 
For simplicity, we restrict to configurations of points in linearly general position, as this will suffice for our purposes.
We remark, however, that the correspondence can be defined in a more general setting.

\begin{defn}\label{def:gale}
Let $P=\big(P_1, \dots, P_{k}\big)$ and $Q=\big(Q_1, \dots, Q_{k}\big)$ be configurations of points in linearly general position in $\PP^n$ and $\PP^s$, respectively, considered up to projective equivalence. 
Write $A$ for the $k\times (s+1)$ matrix of rank $s+1$ whose rows correspond to the points $Q_1, \dots, Q_{k}$, 
and write $B$ for the $k\times (n+1)$ matrix of rank $n+1$ whose rows correspond to the points $P_1, \dots, P_{k}$. 
The two configurations $P$ and $Q$ are said to be \emph{Gale dual} if there is an invertible diagonal $k\times k$ matrix $D$ such that 
$B^t D A=0$.
\end{defn}

\begin{rmk}\label{rmk:gale}
If we are given $k$ points $Q_1, \dots, Q_{k}$ in linearly general position in $\PP^s=\PP(\CC^{s+1})$, here is a concrete way to construct their Gale dual points.
Write $A$ for the $k\times (s+1)$ matrix whose rows correspond to a choice of coordinates for the points $Q_1, \dots, Q_{k}$. 
Let $G\subset \CC^k$ be the subspace spanned by the columns of $A$, and consider the quotient $\CC^k\twoheadrightarrow \CC^k/G$. 
Then the Gale dual points $P_1,\ldots, P_k$ in $\PP(\CC^k/G)=\PP^n$ are given by the images of the vectors in the standard basis of $\CC^k$.    
\end{rmk}

Next, we explain a manifestation of Gale duality observed by Mukai, which serves as the starting point of our study of blowups of projective spaces as moduli spaces of vector bundles. We follow the argument in \cite[Lemma 5.16(a)]{CCF_v1}, which is done for $s=2$, but works for  any $s\geq 2$.

Let $Q_1, \dots, Q_{k}\in \PP^s$ be a collection of points in linearly general position, and consider the blowup $Y=\Bl_{Q_1,\ldots, Q_k}\PP^s$. 
In what follows, denote by $h\in \Pic(Y)$ the pullback of the hyperplane class of $\PP^s$, by $e_i\subset Y$ the exceptional divisor over $Q_i$ and its class in $\Pic(Y)$, and set $e:=\sum_{i=1}^k e_i$.
We will be interested in nontrivial extensions 
\begin{equation}\label{xi_F}
\xi_F: \ \ 0\ra \cO_Y(h-e)\ra F \ra \cO_Y\ra0 \ ,
\end{equation}
which are parametrized by the vector space $\Ext^1\big(\cO_Y, \cO_Y(h-e)\big)\cong H^1\big(Y,\cO_Y(h-e)\big)$.

\begin{prop}\label{prop:Mukai_Gale}
Let $P=\big(P_1, \dots, P_{k}\big)$ and $Q=\big(Q_1, \dots, Q_{k}\big)$ be Gale dual configurations of points in $\PP^n$ and $\PP^s$, respectively.
Set $Y=\Bl_{Q_1,\ldots, Q_k}\PP^s$, and let $h$ and $e$ be as above. Then the following statements hold. 
\begin{enumerate}
    \item $H^1\big(Y,\cO_Y(h-e)\big)\cong \CC^{n+1}$.
    \item Under the identification $\PP^n\cong \PP\big(H^1(\cO_Y(h-e))\big)$, the point $P_i\in \PP^n$ corresponds to the unique (up to scaling) non-trivial extension $\xi_{F_i}$ as in (\ref{xi_F}) for which $F_i$ contains a subsheaf isomorphic to $\cO_Y(-e_i)$ (with locally free cokernel), $i\in\{1,\dots, k\}$. 
    
    \item Under the identification $\PP^n\cong \PP\big(H^1(\cO_Y(h-e))\big)$, the points in the line $\ell_{ij}\subset \PP^n$ through distinct points $P_i$ and $P_j$ correspond to the extensions $\xi_{F_{ij}}$ as in (\ref{xi_F}) for which $F_{ij}$ contains a subsheaf isomorphic to $\cO_Y(-e_i-e_j)$  (with locally free cokernel except for the extensions corresponding to $P_i$ and $P_j$).
\end{enumerate}
\end{prop}

\begin{proof}
Denote by $\cI_{Q}\subset\cO_{\PP^s}$ the ideal sheaf of the set of points $\{Q_1,\ldots, Q_k\}\subset \PP^s$. There is a canonical identification 
$H^1\big(\PP^s,\cI_{Q}(1)\big)\cong H^1\big(Y,\cO_Y(h-e)\big)$. Statement (1) follows from considering the long exact sequence associated to the canonical sequence
$$
0\ra \cI_{Q}(1) \ra \cO_{\PP^s}(1) \ra \bigoplus_{i=1}^k \cO_{\PP^s}(1)_{|Q_i}\ra0,
$$
namely:
$$
0\ra H^0\big(\PP^s,\cO_{\PP^s}(1)\big)\xrightarrow{\ \iota\ } \bigoplus_{i=1}^k H^0\big(\PP^s,\cO_{\PP^s}(1)_{|Q_i}\big)\xrightarrow{\ \pi \ } 
H^1\big(\PP^s,\cI_{Q}(1)\big)\ra0.
$$

The map $\iota$ 
is given by $f\mapsto \big(f(Q_1),\ldots, f(Q_k)\big)$ for some choice of coordinates for the points $Q_1, \dots, Q_{k}$. 
It follows from Remark~\ref{rmk:gale} that the Gale dual points $P_1, \dots, P_{k}\in \PP^n$ are given by the images under $\pi$ 
of  the basis vectors $z_1, \dots , z_k\in \bigoplus_{i=1}^k H^0\big(\PP^s,\cO_{\PP^s}(1)_{|Q_i}\big)$, where 
$\CC z_i= H^0\big(\PP^s,\cO_{\PP^s}(1)_{|Q_i}\big)$ for each $i\in\{1,\dots, k\}$.    

We now determine the extensions $\xi_{F_i}$ as in (\ref{xi_F}) that are associated to the corresponding vectors in 
$H^1\big(Y,\cO_Y(h-e)\big)$. To this end, we identify $\CC z_i$ with  $H^0\big(Y,\cO_Y(h)_{|e_i}\big)$
for each $i\in\{1,\dots, k\}$, and regard the sequences above as 
\begin{equation}\label{can_seq}
0\ra\cO_Y(h-e)\ra\cO_Y(h)\ra\bigoplus_{i=1}^k\cO_Y(h)_{|e_i}\ra0\, , \text{ and }
\end{equation}
\begin{equation}\label{pi2}
0\ra H^0\big(Y,\cO_Y(h)\big)\xrightarrow{\ \iota\ } \bigoplus_{i=1}^k H^0\big(Y,\cO_Y(h)_{|e_i}\big)= \bigoplus_{i=1}^k\CC{z_i}\xrightarrow{\ \pi\ } H^1\big(Y,\cO_Y(h-e)\big)\ra 0 \, .
\end{equation}
In what follows, we will show that the vector $\pi(z_i)\in H^1\big(Y,\cO_Y(h-e)\big)$ corresponds to the unique (up to scaling) extension $\xi_{F_i}$ as in (\ref{xi_F}) for which $F$ contains a subsheaf isomorphic to $\cO_Y(-e_i)$, $i\in\{1,\dots, k\}$.

The short exact sequence 
$$
0\ra\cO_Y(h-e)\ra\cO_Y(h-e+e_i)\ra\cO_Y(h-e+e_i)_{|e_i}\cong \cO_Y(h)_{|e_i} \ra0 \, 
$$
fits into a commutative diagram with sequence (\ref{can_seq}), inducing compatible long exact sequences in cohomology:
\begin{equation*}
\begin{CD}
0 @>>> \CC z_i  @>>>  H^1\big(Y,\cO_Y(h-e)\big) @>\delta_i>> H^1\big(Y,\cO_Y(h-e+e_i)\big) @>>> 0 \\
@.           @VVV @V{\cong}VV @VVV @. \\
\dots @>\iota>>  \bigoplus_{i=1}^k\CC{z_i}  @>\pi>> H^1\big(Y,\cO_Y(h-e)\big)    @>>>  0\, .
\end{CD}
\end{equation*}
So the vector $\pi(z_i)\in H^1\big(Y,\cO_Y(h-e)\big)$ generates the kernel of $\delta_i$.
We now analyse this kernel in terms of the extensions parametrized by these cohomology groups, namely:
$$
\Ext^1\big(\cO_Y, \cO_Y(h-e)\big)\cong H^1\big(Y,\cO_Y(h-e)\big) \xrightarrow{\ \delta_i\ } H^1\big(Y,\cO_Y(h-e+e_i)\big)\cong \Ext^1\big(\cO_Y(-e_i), \cO_Y(h-e)\big).
$$
Given a nontrivial extension $\xi_{F}\in \Ext^1\big(\cO_Y, \cO_Y(h-e)\big)$ as in (\ref{xi_F}), 
we apply the functors $\Hom(\cO_Y,-)$ and  $\Hom(\cO_Y(-e_i),-)$ to it, and obtain a commutative diagram with horizontal exact sequences:
\begin{equation*}
\begin{CD}
0 @>>> \Hom(\cO_Y, F)=0  @>>>  \Hom(\cO_Y, \cO_Y)=\CC\cdot Id @>\al>> \Ext^1\big(\cO_Y, \cO_Y(h-e)\big) \\
@.           @VVV @V{\gamma_i}VV @V{\delta_i}VV \\
0 @>>>  \Hom\big(\cO_Y(-e_i), F\big)  @>>> \Hom\big(\cO_Y(-e_i), \cO_Y\big)=\CC\cdot\tau_i    @>\beta_i>>  \Ext^1\big(\cO_Y(-e_i), \cO_Y(h-e)\big)\, ,
\end{CD}
\end{equation*}
where $\tau_i\in \Hom(\cO_Y(-e_i), \cO_Y)$ corresponds to the canonical inclusion $\cO_Y(-e_i)\hookrightarrow \cO_Y$.
We have that $\gamma_i(Id)=\tau_i$ and $\al(Id)=\xi_F$.
It follows that 
$$
\xi_F\in \ker(\delta_i) \ \iff \ \tau_i\in \ker(\beta_i)\, .
$$
Since there are no nonzero maps $\cO_Y(-e_i)\ra\cO_Y(h-e)$, we have that $\tau_i\in \ker(\beta_i)$ 
if and only if the vector bundle $F$ has a subsheaf isomorphic to $\cO_Y(-e_i)$. 
In this case, a straightforward diagram chase shows that the cokernel is locally free whenever the extension $\xi_F$ is nontrivial. 
This proves statement (2).

Statement (3) is proved in the same way as statement (2).
\end{proof}
    
\begin{rmk}\label{rem:expectation}
Proposition~\ref{prop:Mukai_Gale} suggests that $\PP^n=\PP\big(H^1(\cO_Y(h-e))\big)$ may be viewed as the moduli space of 
vector bundles $F$ on $Y=\Bl_{Q_1,\ldots, Q_k}\PP^s$  satisfying 
$$
\rk(F)=2,\quad \text{c}_1(F)=h-e,\quad \text{c}_2(F)=0,
$$
and which are Gieseker semistable for an appropriate choice of polarization on $Y$.
By suitably varying the polarization, one may hope that the extensions
$$
 0\ra \cO_Y(h-e)\ra F_i \ra \cO_Y\ra0 
$$
corresponding to the points $P_i\in \PP^n$ for $i\in\{1,\dots, k\}$ become unstable. These are precisely  
the extensions containing a subsheaf isomorphic to $\cO_Y(-e_i)$ for some $i\in\{1,\dots, k\}$.
The moduli space for this new polarization could then be $X=\Bl_{P_1,\ldots, P_k}\PP^n$.
In the next section, we show that this expectation is indeed correct when $s=2$, and we expect it to hold for any $s\geq 2$.
\end{rmk}

We conclude this subsection by examining Gale duality in connection with the action of the Weyl group discussed in Subsection~\ref{subsection_blowups}.
Let $P=\big(P_1, \dots, P_{k}\big)$ and $Q=\big(Q_1, \dots, Q_{k}\big)$ be Gale dual configurations of points in $\PP^n$ and $\PP^s$, respectively, let $I\subset \{1, \dots,k\}$ be a subset of cardinality $n+1$, and denote by  $I^c=\{1, \dots,k\}\setminus I$ its complement, which has cardinality $s+1$.
One can check that the configurations $(I)_*P$ and $(I^c)_*Q$, as defined in Definition~\ref{gen_position}, are again Gale dual configurations of points in $\PP^n$ and $\PP^s$, respectively. 
In particular, Gale duality preserves the property of being Cremona-general. Henceforth, we assume that the points 
$P_1,\dots, P_{k}\in \PP^n$ and $Q_1,\ldots, Q_k\in \PP^s$ are Cremona-general.

\begin{notn}\label{notation:isometry}
We follow the same notation as above: let $n$ and $s$ be integers $\geq 2$, and set $k=n+s+2$. 
We consider the blowups $X=\Bl_k\PP^n$ and $Y=\Bl_k\PP^s$, the usual bases for their Picard groups, 
$\Pic(X)=\ZZ\{H,E_1,\ldots,E_k\}$  and $\Pic(Y)=\ZZ\{h,e_1,\ldots,e_k\}$, and the Coble pairings on $\Pic(X)$ and $\Pic(Y)$ introduced in Definition \ref{Coble}.
Set 
$$
\al_0=H-E_1-\ldots-E_{n+1},\quad  \be_0=h-e_{n+2}-\ldots-e_{k}, \quad \text{ and }
$$
$$
\al_i=E_i-E_{i+1},\quad \be_i=e_{k+1-i}-e_{k-i}, \quad \text{ for  } \ i\in\{1,\ldots,k-1\}.
$$
Note that $\be_0$ differs from the definition given in \cite{Mukai}.
Recall from Subsection~\ref{subsection_blowups} that
the Weyl group $\cW_X=\cW_{n,k}$ is generated by the reflections $r_{\alpha_i}$ for $i=0,\ldots,k-1$, while  
$\cW_Y=\cW_{s,k}$ is generated by the reflections $r_{\beta_i}$ for $i=0,\ldots,k-1$.

There is an obvious isomorphism between the Dynkin diagrams associated to 
the Weyl groups $\cW_X$ and $\cW_Y$. When $s=2$, this is illustrated below by  
the Dynkin diagrams associated to $\cW_{n,n+4}$ and $\cW_{2,n+4}$.
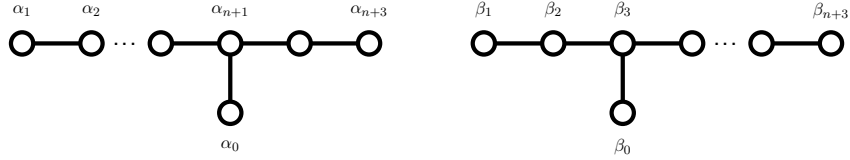
\begin{figure}[h]\notag
    \centering
    \begin{tikzpicture}[
	transform shape,
	%%% Tamanho da figura em porcentagem (1=100%):
	scale=0.6,
	%%%
	node font=\large\bfseries\sffamily,
	%%% Configura√ß√£o dos v√©rtices:
	every node/.append style={circle,minimum size=7mm,line width=1.6pt},
	%%% Configura√ß√£o das arestas:
	every edge/.append style={ultra thick},
	%%% "ultra thick" √© equivalente a "line width=1.6pt". Pode ser substitu√≠do tamb√©m por "ultra thin", "very thin", "thin", "semithick", "thick" e "very thick", para os valores padr√µes de 0.1pt, 0.2pt, 0.4pt, 0.6pt, 0.8pt e 1.2pt, respectivamente. Se quiser customizar mais livremente, use o "line width=xpt", substituindo o "x" por um n√∫mero.
	d/.style={draw},
	]
	\node[d,label={$\alpha_1$},scale=0.7] (1) {};
	\node[d,right=of 1, label={$\alpha_2$},scale=0.7] (2) {} edge (1);
	\node[d,right=of 2,scale=0.7] (n) {};
	%%% Retic√™ncias
	\path (2) -- node[auto=false]{\ldots} (n);
	%%%
	\node[d,right=of n,label={[label distance=-2mm]$\alpha_{n+1}$},scale=0.7] (n+1) {} edge (n);
	\node[d,right=of n+1,scale=0.7] (n+2) {} edge (n+1);
	\node[d,right=of n+2, label={[label distance=-2mm]$\alpha_{n+3}$},scale=0.7] (n+3) {} edge (n+2);
	\node[d,below=of n+1, label=below:{$\alpha_0$},scale=0.7] (0) {} edge (n+1);
	%%%
	%%% Beta
	\node[d,right=of n+3,xshift=10mm, label={$\beta_1$},scale=0.7] (b1) {};
	\node[d,right=of b1, label={$\beta_2$},scale=0.7] (b2) {} edge (b1);
	\node[d,right=of b2, label={$\beta_3$},scale=0.7] (b3) {} edge (b2);
	\node[d,right=of b3,scale=0.7] (b4) {} edge (b3);
	\node[d,right=of b4,scale=0.7] (bn+2) {};
	%%% Retic√™ncias
	\path (b4) -- node[auto=false]{\ldots} (bn+2);
	%%%
	\node[d,right=of bn+2,label={[label distance=-2mm]$\beta_{n+3}$},scale=0.7] (bn+3) {} edge (bn+2);
	\node[d,below=of b3, label=below:{$\beta_0$},scale=0.7] (b0) {} edge (b3);
\end{tikzpicture}
\end{figure}

This induces a group isomorphism $\cW_Y\ra\cW_X$ given by $r_{\be_0}\mapsto r_{\al_0}$ and $r_{\be_i}\mapsto r_{\al_{k-i}}$ for $i=1,\ldots, k-1$. We denote by $\phi_0$ the isometry between the two lattices $K_Y^\perp\subset\Pic(Y)$ and $K_X^\perp\subset\Pic(X)$ given by  
$$
\phi_0:\quad \beta_0\mapsto \alpha_0 \quad \text{ and } \quad \beta_i\mapsto \alpha_{k-i}\quad \text{ for } \ i\in\{1,\ldots,k-1\}.
$$
This isometry is compatible with the identification of the two Weyl groups $\cW_Y$ and $\cW_X$ described above.
\end{notn}

The groups $\Pic(Y)$ and $\Pic(X)$ are clearly isomorphic for rank reasons. 
The following proposition characterizes the isomorphisms $\Pic(Y)\ra\Pic(X)$ that are equivariant with respect to the action of the  Weyl group $\cW=\cW_Y=\cW_X$, identified as in Notation~\ref{notation:isometry}. 
The notion of $\cW$-equivariance means that 
$$
\rho(w\cdot L)=w\cdot \rho(L) \quad \text{for all} \ w\in\cW, \  L\in\Pic(Y).
$$
Although this result is not strictly necessary for what follows, it helps to interpret the determinant map that will be constructed in Section~\ref{section:determinant}.

\begin{proposition}\label{W and rho}
Let the notation be as in Notation~\ref{notation:isometry}.
Let $\rho:\Pic(Y)\otimes\QQ\ra\Pic(X)\otimes\QQ$ be a linear map.

\bi 
\item[(1)] The restriction $\rho_{|K_Y^\perp}$ is $\cW$-equivariant  if and only if $\exists \lambda\in\QQ$ such that 
$\rho_{|K_Y^\perp}=\lambda\cdot\phi_0$. Furthermore, $\rho$ induces an isometry $K_Y^\perp\ra K_X^\perp$ if and only if $\lambda=\pm 1$.

\item[(2)] The map $\rho$ is $\cW$-equivariant  if and only if $\exists a, b\in\QQ$ such that
$$\rho(e_i)=aH-b\sum_{j=1}^k E_j-\lambda E_{i}\quad  \text{ for }  i=1,\ldots,k,$$
$$\rho(h)=\big(a(s+1)+\lambda\big)H-\big(b(s+1)+\lambda\big)\sum_{j=1}^k E_j,$$
where $\lambda=a(n-1)-b(n+1)$. In this case, we have 
$$
\rho(K_Y)=\big(2a-b(s+1)\big)K_X,
$$ 
and $\rho_{|K_Y^\perp}=\lambda\cdot \phi_0$ as in (1).
For any $\alpha\in\QQ$, $\alpha\cdot\rho_{|K_Y^\perp}\colon K_Y^\perp\ra K_X^\perp$ is an isometry if and only if $\alpha\cdot\lambda=\pm 1$. 
\ei
\end{proposition}

\bp
It is immediate to check that $\lambda\cdot\phi_0$ is $\cW$-equivariant for any $\la\in\QQ$.
Conversely, assume that $\rho_{|K_Y^\perp}$ is $\cW$-equivariant, i.e., $\rho(w\cdot L)=w\cdot \rho(L)$, for any $w\in\cW$,  $L\in K_Y^\perp$. 
Recall that $K_Y^\perp=\langle \be_0, \dots, \be_{k-1}\rangle$.
$\cW$-invariance implies that $\rho(r_{\be_j}(\be_i))=r_{\al_{k-j}}(\rho(\be_i))$ if $j\neq0$ and 
$\rho(r_{\be_0}(\be_i))=r_{\al_0}(\rho(\be_i))$. Using the definition of the reflections $r_{\be_i}$, this implies
that  $\rho(\be_i)=\la\cdot\al_{k-i}$ for some $\la\in\QQ$
and all $i\in\{1,\ldots, k-1\}$.
From $r_{\be_0}(\be_{s+1})=\be_{s+1}+\be_0$, we obtain that 
$\rho(\be_0)=\la\al_0$ and hence $\rho_{|K_Y^\perp}=\lambda\cdot\phi_0$.
The fact that $\rho$ induces an isometry $K_Y^\perp\ra K_X^\perp$ if and only if $\lambda=\pm 1$ is immediate.
This proves (1). 

Now we prove (2). A direct computation shows that the linear map $\rho:\Pic(Y)\otimes\QQ\ra\Pic(X)\otimes\QQ$ defined by the formulae above satisfies $\rho(K_Y)=\big(2a-b(s+1)\big)K_X$ and $\rho_{|K_Y^\perp}=\lambda\cdot \phi_0$.
In particular, $\rho$ is $\cW$-equivariant. 
Conversely, assume now that $\rho:\Pic(Y)\otimes\QQ\ra\Pic(X)\otimes\QQ$ is $\cW$-equivariant. 
This is equivalent to the restriction $\rho_{|K_Y^\perp}$ being $\cW$-equivariant, and moreover 
$\rho(w\cdot e_i)=w\cdot \rho(e_i) \quad \forall w\in\cW$ for some $i\in\{1,\ldots, k-1\}$.
Note that $r_{\be_j}(e_1)=e_1$ for all $j\neq k-1$, and recall that 
the group $\cW$ contains the permutation group on $k$ letters, corresponding to the permutations of the $e_i$'s.
It follows that there exist $a, b, c\in\QQ$ such that
$$\rho(e_i)=aH-b\sum_{j=1}^k E_j-cE_i,$$
and from $r_{\be_0}(e_1)=e_1$ we deduce that $c=a(n-1)-b(n+1)$.
By statement (1), $\rho(\be_i)=\la\cdot\al_{k-i}$ for some $\la\in\QQ$, and so we must have $\la=c$. To compute $\rho(h)$, use $\cW$-equivariance and the condition $\rho(\be_0)=\la\al_0$. This proves (2). 
\ep

\subsection{Blowups of the projective plane}\label{subsection:blowups1}
As suggested in Remark~\ref{rem:expectation}, we will investigate the blowups $\Bl_{n+4}\PP^n$ by viewing them as moduli spaces of semistable torsion-free sheaves on $\Bl_{n+4}\PP^2$. As we shall review in Section~\ref{section:moduli}, the notion of semistability depends on the choice of a polarization on the surface, and as the polarization varies, so does the corresponding moduli space. To control this variation, we need to understand the ample cone of $\Bl_{n+4}\PP^2$, as well as its closure, the nef cone, and its dual cone $\NE(\Bl_{n+4}\PP^2)$, the cone of pseudoeffective curves. In this section, we review what is known about these cones. While their precise shape is not fully known in general, the portions of these cones contained in the half-space $(K < 0)$ are well understood.

\begin{notn}\label{notation_Pic(S)}
Throughout this subsection, we denote by $S=\Bl_k\PP^2$ the blowup of $\PP^2$ at $k\geq 0$ distinct points, $Q_1, \dots, Q_{k}$, and by $\pi:S\to \PP^2$ the blowup morphism. When $k\geq 3$, we assume that the points are Cremona-general.
We denote by $e_i\subset S$ the exceptional curve over $Q_i$, as well as its class in $\Pic(S)$, and by $h\in \Pic(S)$ the pullback of the hyperplane class of $\PP^2$.
We denote simply by $K$ the canonical class $K_S$ of $S$, and by $\cW=\cW_{2,k}$ the Weyl group from Definition~\ref{Weyl}.
\end{notn}

We denote by $\mathcal R\subset \NE(S)$ the set of classes of $-1$-curves on $S$.
By \cite[Corollary 1]{Dolgachev}, when $k\geq 4$ the set $\mathcal R$ coincides with the orbit of a single $-1$-curve under the action of the Weyl group $\cW$
(see also Proposition~\ref{partial classif}):
$$
\mathcal R \ = \ \big\{ \ \text{classes of $-1$-curves} \ \big\} \ = \ \cW\cdot e_1.
$$
When $k\geq2$, by the Cone Theorem (see for instance \cite[Theorem~1.28]{KM}), we have:
$$
\NE(S)=\NE(S)^{K\geq0}+\sum_{e\in \mathcal R} \RR_{\geq0} e.
$$
While this description of $\NE(S)$ holds for any surface, we can say more in the special case $S=\Bl_k\PP^2$.
Part of this statement appears in \cite[Lemma 4.1]{deF} under a stronger generality assumption. 

\begin{lemma}\label{E positive surface case}
Let $S=\Bl_k\PP^2$, with $k\geq2$. 
\begin{enumerate}
    \item If $D\in \Pic(S)$ is such that $D\cdot e\geq0$ for every $e\in \mathcal R$, then
    $
    D\in \ZZ_{\geq0}(-K)+\sum_{e\in \mathcal R} \ZZ_{\geq0} e.
    $
    \item If $D\in \Pic(S)$ is effective, then
    $
    D\in \ZZ_{\geq0}(-K)+\sum_{e\in \mathcal R} \ZZ_{\geq0} e.
    $
    \item $\NE(S)\subseteq\overline{\RR_{\geq0}(-K)+\sum_{e\in \mathcal R} \RR_{\geq0} e}$.
    \item When $k=9$, $\NE(S)=\RR_{\geq0}(-K)+\sum_{e\in \mathcal R} \RR_{\geq0} e$.
\end{enumerate}
\end{lemma}

\bp
If $k\leq8$, then $S$ is a del Pezzo surface,  
$$
\Nef(S) \subset \NE(S) = \sum_{e\in \mathcal R} \RR_{\geq0} e. % \quad \text{and} \quad 
%\Pic(S)\cap \NE(S) = \sum_{e\in \mathcal R} \ZZ_{\geq0} e.
$$
and $\Pic(S)\cap \NE(S) = \sum_{e\in \mathcal R} \ZZ_{\geq0} e$ if $k\leq7$, and 
$\Pic(S)\cap \NE(S) = \ZZ_{\geq0} (-K) +\sum_{e\in \mathcal R} \ZZ_{\geq0} e$ if $k=8$.

Suppose from now on that $k\geq9$. Let $D\in \Pic(S)$ be such that $D\cdot e\geq0$ for every $e\in \mathcal R=\cW\cdot e_1$.
We will prove that $D\in \ZZ_{\geq0}(-K)+\sum_{e\in \mathcal R} \ZZ_{\geq0} e$ by induction on $k$. 
Suppose first that $D\cdot e=0$ for some $e\in\mathcal R$. Up to the action of $\cW$, we may assume that $e=e_{k}$. Then $D$ is the pullback of a divisor on $\Bl_{k-1}\PP^2$ satisfying the same assumption as $D$, and  the result follows by induction on $k$.
If $D\cdot e>0$ for every $e\in\mathcal R$, then the divisor $D'=D+K$ has the property that $D'\cdot e=D\cdot e-1\geq 0$ for every $e\in\mathcal R$, and (1) is proved by induction on $\min\big\{D\cdot e\,|\,e\in \mathcal R\big\}\geq0$.

Now suppose that $D$ is effective. If $D\cdot e\geq0$ for every $e\in \mathcal R$, then we conclude by (1). 
Suppose $m:=-D\cdot e>0$ for some $e\in\mathcal R$.
Up to the action of $\cW$, we may assume that $e=e_k$ and $D=D'+me_k$, where $D'$ is the pullback of an effective divisor on $\Bl_{k-1}\PP^2$, and (2) is proved by induction on $k$.

Statement (3) follows from (2) by passing to the closure. 

When $k=9$, $-K$ is effective, and so we have $\NE(S)=\overline{\RR_{\geq0}(-K)+\sum_{e\in \mathcal R} \RR_{\geq0} e}$. Note that $-K$ is in fact nef, and it follows from statement (2) that in this case $\NE(S)\cap K^{\perp}=\RR_{\geq0}(-K)$. This proves (4).
\ep

The next lemma describes the cone $\Nef(S)^{K\leq0}$ and its boundary points. 

\begin{lemma}\label{K negative nef}\label{ample K-negative}
Let $S=\Bl_k\PP^2$, with $k\geq2$. 
Then 
\bi
\item[(1) ] $\Nef(S)^{K\leq0}=\big(K\leq0\big)\cap\displaystyle\bigcap_{e\in \mathcal R}\big(e\geq0\big)$. 
\ei
Let $D\in \Pic(S)\cap\Nef(S)^{K\leq0}$.
\bi
\item[(2) ] If $D\cdot e>0$ for every $e\in \mathcal R$, then either $D$ is ample, or $k=9$ and $D=-mK$ for some $m\in\ZZ_{>0}$. 
\item[(3) ] If $D$ is not ample and $D^2>0$, then $D$ is the pullback of an ample divisor on a surface $S'$ via a smooth blowdown $S\ra S'$.
\item[(4) ] If $D\neq 0$ but $D^2=0$, then $D=mD_0$, where $m\in\ZZ_{>0}$ and $D_0$ satisfies the following conditions. 
\bi 
\item[(i) ] If $D\cdot (-K)>0$, then $D_0=\pi^*(h'-e')$, where $\pi\colon S\ra \Bl_1\PP^2$ is a smooth blowdown, $h'\in \Pic(\Bl_1\PP^2)$ is the pullback of a line in $\PP^2$ and $e'\in \Pic(\Bl_1\PP^2)$ is the class of the unique $-1$-curve on $\Bl_1\PP^2$. 
\item[(ii) ] If $D\cdot K=0$, then $k\geq9$ and $D_0\in\cW\cdot (3h-e_1-\ldots-e_9)$. 
\ei
\ei
\end{lemma}

\bp
The inclusion $\Nef(S)^{K\leq0}\subseteq \big(K\leq0\big)\cap\bigcap_{e\in \mathcal R}\big(e\geq0\big)$ is obvious. To prove the reverse inclusion and establish (1), observe that $\NE(S)\subseteq\overline{\RR_{\geq0}(-K)+\sum_{e\in \mathcal R} \RR_{\geq0} e}$ by Lemma~\ref{E positive surface case} and consider the dual cones. This proves (1). 

Let $D\in \Pic(S)\cap\Nef(S)^{K\leq0}$.
If $k\leq8$, then $S$ is a del Pezzo surface and statements (2)--(4) are well known. 
So from now on we assume that $k\geq9$.

Suppose that $D\cdot e>0$ for every $e\in \mathcal R$ and let us prove (2). 
By assumption, $D^2\geq0$ and $D\cdot (-K)\geq0$. 
Since $h^2(D)=h^0(K-D)=0$, it follows from Riemann-Roch that $D$ is an effective divisor. 
By Lemma \ref{E positive surface case}(2), we can write
$D=l(-K)+\sum m_\al e_\al$, for some $e_\al\in \mathcal R$ and integers $l\geq0$, $m_\al\geq0$. 
If all $m_\al=0$, then $D=l(-K)$, and since $D^2\geq0$ and we assume $k\geq9$, we must have $k=9$. 
Thus, we may assume that at least one $e_\al$ appears with coefficient $m_\al>0$. 
Our assumptions $D\cdot(-K)\geq0$ and $D\cdot e_\al>0$ then imply that 
$$
D^2=l(D\cdot(-K))+\sum m_\al (D\cdot e_\al)>0.
$$
Suppose that $D$ is nef but not ample. 
By the Nakai-Moishezon criterion, there exists a curve $C\subset S$ such that $D\cdot C=0$. 
By Lemma \ref{E positive surface case}, $C=a(-K)+\sum b_\be e_\be$ for some $a, b_\be\in\ZZ_{\geq0}$. 
Then the condition that 
$D\cdot C=D\cdot a(-K)+\sum b_\be (D\cdot e_\be)=0$ implies that $C=a(-K)$ and $D\cdot K=0$. 
But this is only possible if $k=9$ and $D=l(-K)$. This proves (2).

We now prove (3) and (4). Assume that $D\neq 0$ and that $D$ is nef but not ample. 
By (2), either $k=9$ and $D=-mK$ for some $m\in\ZZ_{>0}$, or $D\cdot e=0$ for some $e\in \mathcal R$. 
The first case is a particular case of (4.ii).
Now assume that $D\cdot e=0$ for some $e\in \mathcal R$.
By subsequently contracting $-1$-curves with zero intersection with $D$, we obtain a smooth blowdown $\pi: S\ra S'=\Bl_{k'}\PP^2$, with $0\leq k'<k$, 
and a nef divisor $D'$ on $S'$ such that $D'\cdot e'>0$ for all $-1$-curves $e'$ on $S'$, and $D=\pi^*D'$. 
In particular, $(D')^2=D^2$.
If $k'\geq 2$, we can apply induction, and (2) implies that either $D'$ is ample or $k'=9$ and $D'=-mK_{S'}$ for some $m\in\ZZ_{>0}$.
In the first case we have (3), and in the latter case we have (4.ii).
Assume now that $k'\leq 1$. 
The condition that $D'$ is a nef divisor on $S'$ and $D'\cdot e'>0$ for all $-1$-curves $e'$ on $S'$ implies that 
either $D'$ is ample, or $k'=1$ and $D'=m(h'-e')$, where $m\in\ZZ_{>0}$, $h'$ is the pullback of a line in $\PP^2$ and $e'$ is the unique $-1$-curve on $S'=\Bl_1\PP^2$. 
In the first case we have (3), and in the latter case we have (4.i).
\ep

\begin{corollary}\label{FDomain}
Let $S=\Bl_k\PP^2$, with $k\geq4$. 
Then the extremal rays of the nef cone $\Nef(S)$ in the half-space $(K<0)$ are precisely the following:
\begin{itemize}
    \item $\RR_{\geq0}C$ for some $C\in \cW\cdot h$;
    \item $\RR_{\geq0}C$ for some $C\in\cW\cdot (h-e_1)$.
\end{itemize}
\end{corollary}

\bp
By Lemma~\ref{K negative nef}(1), 
\[
\Nef(S)^{K\leq0}=(K\leq0)\cap\bigcap_{e\in \mathcal R}(e\geq0).
\]
Note that the set $R$ is locally discrete in $(K<0)$. 
Therefore, any extremal ray $\Gamma\subset \Nef(S)$ contained in the half-space $(K<0)$ is the intersection of $k$ hyperplanes of the form $(e=0)$ for some $e\in \mathcal R$. In particular, $\Gamma$ is rational, i.e., it is generated by a class $C\in \Pic(S)$.
Applying Lemma~\ref{K negative nef}(3) and (4), we conclude that $C\in \cW\cdot h$ or $C\in \cW\cdot (h-e_1)$. Here we use the fact that, when $k\geq 4$, the condition on $D_0$ in (4.i) is equivalent to $D_0\in \cW\cdot (h-e_1)$.
\ep

\begin{rmk}\label{fundamental_domain}
For $k\geq 9$, \cite[Theorem 6.2]{DUrso} explicitly exhibits a rational polyhedral fundamental domain $\Gamma$ for $\Nef(S)^{K\leq0}$ with respect to the action of the Weyl group $\cW$. The polyhedral cone $\Gamma$ has exactly two extremal rays $\Gamma$ satisfying $\Gamma^2=0$, namely $\RR_{\geq 0}\cdot (h-e_1)$ and $\RR_{\geq 0}\cdot (3h-e_1-\ldots-e_9)$.
\end{rmk}

\begin{corollary}\label{La}
Let $S=\Bl_k\PP^2$, with $k\geq0$, and consider the $\QQ$-divisor
\[
L_a:=-K+ah=(a+3)h-\sum_{i=1}^{k} e_i, \quad \text{ with } \ a\in\QQ_{\geq0}.
\]
Then $L_a$ is ample and $L_a\cdot (-K)>0$ if one of the following holds: 
\bi
\item $k<9$ and $a\geq 0$,
\item $k\geq9$ and $a> \frac{k}{3}-3$.
\ei
\end{corollary}

\bp
If $k\leq 9$, then $-K$ is nef and the result is clear.
So we assume $k>9$.
We have $L_a\cdot (-K)>0$ if and only if $a> \frac{k}{3}-3$.
By Lemma~\ref{ample K-negative}, $L_a$ is ample if $L_a\cdot e>0$ for every $e\in \mathcal R$. But this is clear, as $(-K)\cdot e=1$ for every $e\in \mathcal R$ and $h$ is nef. 
\ep

%%%%%%%%%%%%%%%%%%%%%%%%%%%%%%%%%%%%%%%%%%%%%%%%%%%%%%%%%%%%%%%%%%%%%%%%%

\section{Moduli spaces of semistable sheaves on blowups of $\PP^2$}\label{section:moduli}

In this section, we introduce and describe moduli spaces of semistable torsion-free sheaves on blowups of $\PP^2$ at a finite set of points.
The variation of these moduli spaces under changes of polarization on the surface was extensively investigated in the 1990s
(\cite{Ellingsrud_Gottsche}, \cite{FQ}, \cite{G}, \cite{MW}).
As we shall review, this variation is governed by a wall-and-chamber decomposition of the ample cone of the surface.
Throughout this section, we denote by $S$ the blowup of $\PP^2$ at $k=n+4\geq 6$ distinct points, $Q_1, \dots, Q_{k}$,
not necessarily Cremona-general,
by $e_i\subset S$ the exceptional curve over $Q_i$, as well as its class in $\Pic(S)$, by $h\in \Pic(S)$ the pullback of the hyperplane class of $\PP^2$, and by $K$ the canonical class $K_S$ of $S$.

\subsection{Gieseker stability and slope stability}
We start by reviewing the notions of Gieseker stability and slope stability and
introduce the relevant moduli spaces of torsion-free sheaves on $S$. 
We refer to \cite{HL} for a general background.

Let $F$ be a coherent sheaf on $S$. 
We will often make use of the following formula, which follows from Hirzebruch Riemann-Roch:
\begin{equation}\label{HRR1}
\chi(F)=\ch_2(F)+\frac{1}{2}c_1(F)\cdot (-K)+\rk F,
\end{equation}
where $\ch_2(F)$ denotes the second Chern character of $F$.

\begin{defn}\label{def:stability}
Let $L$ be an ample line bundle on $S$. 
Let $F$ be a rank $2$ torsion free sheaf  on $S$.
We say that $F$ is \emph{(Gieseker) $L$-stable} (respectively, \emph{$L$-semistable}) if, for any subsheaf 
$F_0\subset F$ of rank $1$ and $m\gg 0$, one has 
$$\chi(F_0\otimes L^{\otimes m}) <\frac{1}{2}\chi(F\otimes L^{\otimes m})\quad (\text{respectively,}\ \leq).$$
We say that $F$ is \emph{$\mu_L$-stable} (respectively, \emph{$\mu_L$-semistable}) if, for any subsheaf 
$F_0\subset F$ of rank $1$, one has 
$$\mu_L(F_0):=\text{c}_1(F_0)\cdot L< \mu_L(F):=\frac{1}{2}(\text{c}_1(F)\cdot L)\quad (\text{respectively,}\ \leq).$$
\end{defn}

\begin{rmk}\label{rmk:stability}
Let $L$ and $F$ be as above. 
By (\ref{HRR1}), 
$$\chi(F\otimes L^{\otimes m})=\chi(F)+m\cdot \rk(F)\left(\mu_L(F)+\frac{L^2-K\cdot L}{2}\right).$$
Therefore $F$ is $L$-stable (respectively, $L$-semistable)  if and only if $F$ is $\mu_L$-semistable, and, for any subsheaf 
$F_0\subset F$ of rank $1$ with $\mu_L(F_0)= \mu_L(F)$, one has $\chi(F_0)< \frac{\chi(F)}{2}$ (respectively, $\leq$). 
In particular, one has the following chain of implications
$$\mu_L\text{-stable}\quad\Rightarrow\quad L\text{-stable}\quad\Rightarrow\quad L\text{-semistable}\quad\Rightarrow\quad \mu_L\text{-semistable}.$$
\end{rmk}

\subsection{Moduli spaces of semistable sheaves}

In order to consider moduli spaces of torsion-free sheaves $F$ on $S$, we must fix some discrete invariants of $F$, namely,
$\rk(F)$, $\text{c}_1(F)$ and $\text{c}_2(F)$. Remark~\ref{rem:expectation} suggests taking
$$
\rk(F)=2,\quad \text{c}_1(F)=h-\sum_{i=1}^k e_i,\quad \text{c}_2(F)=0.
$$
However, the correspondence $F\mapsto F\otimes\OO_S(h)$ induces isomorphisms between the moduli spaces for these invariants and those for
$$
\rk(F)=2,\quad \text{c}_1(F)=-K,\quad \text{c}_2(F)=2.
$$
We choose to use the latter in order to follow \cite{Mukai} and \cite{CCF}. 

\begin{assump} 
We consider torsion-free sheaves $F$ on $S$ with Chern classes
\begin{equation} \label{data} \tag{\ddag}
\rk(F)=\ch_0=2,\quad \text{c}_1(F)=-K,\quad \text{c}_2(F)=2.
\end{equation}
\end{assump}

\begin{notn}\label{not:M_L}
Let $L$ be an ample line bundle on $S$. 
We denote by $M_L$ the moduli space of \emph{$L$-semistable torsion free sheaves} $F$ with (\ref{data}) \cite[Theorem 4.3.4]{HL}.
We denote by $M^{\mu ss}_L$ the moduli space of  \emph{$\mu_L$-semistable torsion free sheaves} with (\ref{data}), as defined in \cite[Def. 8.2.7]{HL}. 
We denote by $M^{s}_L\subseteq M_L$ the open subset corresponding to \emph{$L$-stable torsion free sheaves}, and by  
$M^{\mu s}_L\subseteq M^{\mu ss}_L$ the open subset corresponding to \emph{$\mu_L$-stable torsion free sheaves}.   
\end{notn}

We will be particularly interested in the moduli spaces $M_L$ and $M^{\mu ss}_L$ when $L$ is a polarization such that $K \cdot L < 0$. 
Lemma~\ref{smoothness} below shows that these moduli spaces are smooth at stable points and computes their dimension, while
Lemma~\ref{sstable vb} identifies necessary conditions for the moduli space $M^{\mu ss}_L$ to be nonempty.

\begin{lemma}\label{smoothness}
Suppose that $K \cdot L < 0$, and 
let $F$ be a $L$-stable torsion-free sheaf with (\ref{data}). Then the moduli space  $M_L$ is smooth at $[F]$ and 
$$\dim_{[F]} M_L=k-4=n.$$ 
\end{lemma}

\bp
We show that  $\Ext^2(F,F) = 0$. 
Using Serre duality, it is enough to prove that $\Hom(F, F\otimes K)=0$. As $F$ is $\mu_L$-semistable, 
$F\otimes K$ is also $\mu_L$-semistable. 
Since $\mu_L(F\otimes K)-\mu_L(F)=L\cdot K<0$, it follows that $\Hom(F, F\otimes K)=0$. It follows from \cite[Thm. 4.5.4, Definition 4.5.6]{HL} that 
$M_L$ is smooth at $[F]$ and that $\dim_{[F]} M_L$ equals the expected dimension 
\begin{align*}
\dim\big(\Ext^1(F,F)\big) &= 2\rk(F)\text{c}_2(F)-(\rk(F)-1)\text{c}_1(F)^2-(\rk(F)^2-1)\chi(\cO_S)\\
&=k-4=n.
\end{align*}
\ep

\begin{lemma}\label{sstable vb}
Let $h'\in \Pic(S)$ be an effective divisor class with $(h')^2=1$ and $h'\cdot (-K)=3$.
Let $L$ be an ample line bundle on $S$ such that $L \cdot K < 0$  and $M^{\mu ss}_L\neq \emptyset$.
Then $$(2h'+K)\cdot L\geq 0.$$ 
\end{lemma}

\bp
Let $F\in M^{\mu ss}_L\neq \emptyset$, and consider its double dual $E:=F^{\vee\vee}$. 
Then $E$ is a $\mu_L$-semistable vector bundle having the same slope as $F$.
Let us compute $\chi(E^\vee(h'))=\chi(F^\vee(h'))$ using formula (\ref{HRR1}).
We have: 

\begin{align*}
\ch_2(F^\vee(h'))&=\ \ch_2(F^\vee)+\text{c}_1(F^\vee)\cdot h'+(h')^2=\\
                &=\left(\frac{K^2}{2}-2\right)+K\cdot h'+(h')^2= \\
                &= \ \frac{K^2}{2}-4, 
\end{align*}
$$\text{c}_1(F^\vee(h'))=\text{c}_1(F^\vee)+2h'=K+2h'.$$
Thus it follows from (\ref{HRR1}) that 
\begin{equation}\label{chiE'(h)}
\chi(E^\vee(h'))=\chi(F^\vee(h'))=1.
\end{equation}
Since $E$ is a vector bundle, we have 
$\HH^i(E^{\vee}(h'))\cong\Ext^i(E,\cO(h'))$ for every $i\geq 0$. By Serre duality, 
$$\Ext^2(E,\cO(h'))\cong\Hom(\cO(h'), E\otimes K)^{\vee}.$$
Since $L \cdot K < 0$ and  $L$ is ample, we have 
$$\mu_L(\cO(h'))= h'\cdot L>0>\frac{K\cdot L}{2} = \mu_L (E\otimes K).$$
Since $E$ is $\mu_L$-semistable, it follows that $\Hom(\cO(h'), E\otimes K)=0$, i.e., $\HH^2(E^\vee(h'))=0$. 
By (\ref{chiE'(h)}), we must have:
$$\HH^0(E^\vee(h'))\cong \Hom(E,\cO(h'))\neq 0.$$
Since  $E$ is  $\mu_L$-semistable, this implies that
$$\mu_L(E)=\frac{-K\cdot L}{2}\leq \mu_L(\cO(h'))=h'\cdot L,$$
and thus 
$$(2h'+K)\cdot L\geq 0.$$ 
\ep

\begin{rmk}\label{rmk:emptyM_L}
We note that there exist polarizations $L\in \Amp(S)$ with $L \cdot K < 0$  for which $M^{\mu ss}_L= \emptyset$.
Indeed, take $L=dh-\sum_{i=1}^{n} e_i$ with $d\gg 0$
and apply Lemma~\ref{sstable vb} with $h'=h$. 
\end{rmk}

\subsection{The variation of the moduli spaces $M_L$} \label{subsection:variation_of_ML}
The variation of the moduli spaces $M_L$ and $M^{\mu ss}_L$ when we vary the polarization $L\in \Amp(S)$
has been well studied, particularly when $K \cdot L < 0$. 
Our presentation follows \cite{Ellingsrud_Gottsche}, but we also refer to \cite{FQ, G, MW, HL}.
As we shall see, the moduli space $M_L$ changes when the polarization $L$ crosses certain hyperplanes in $\N^1(S)$, called \emph{walls}, while it remains constant when $L$ belongs to a fixed \emph{chamber}.
We define the stability walls and chambers in our particular setting (\ref{data}).
This is a special case of \cite[Definition 2.4]{Ellingsrud_Gottsche}. 

\begin{defn}\label{def:wall}
A \emph{(stability) wall} in $\N^1(S)$ is a hyperplane  of the form
$$(2D+K)^{\perp}=\{L\in \N^1(S)\ | \ (2D+K)\cdot L=0\},$$
where $D\in \Pic(S)$ satisfies the following conditions:
\begin{enumerate}
    \item $D\cdot (D+K)\geq -2$, and
    \item there is a polarization $L_0\in (2D+K)^{\perp}$. 
\end{enumerate}
Note that the set of walls is locally finite on $\Amp(S)$. 

A \emph{(stability) chamber} $\cC\subset \Amp(S)$ is a connected component of the complement of the union of all walls in the ample cone $\Amp(S)\subset \N^1(S)$. 
We say that a wall $(2D+K)^{\perp}$ \emph{bounds} a chamber $\cC$ if the closure of $\cC$ contains a nonempty open subset of $(2D+K)^{\perp}$. 
Two distinct chambers are \emph{neighbouring chambers} if the intersection of their closures contains a nonempty open subset of a wall.
\end{defn}

The following is \cite[Theorem 4.C.3]{HL} in our setting (\ref{data}).

\begin{prop}[{}]\label{prop:wall&chamber}
Let $L\in \Amp(S)$ be a polarization not lying on a wall. 
Then the notions of $\mu_L$-stability, $L$-stability, $L$-semistability and $\mu_L$-semistability all coincide. 
Moreover, the moduli space $M_L$ depends only on the chamber $\cC$ containing $L$. 
We write 
$$M_{\cC}:=M_L\quad  (L\in \cC),$$ 
and often refer to $L$-stability as $\cC$-stability. 
\end{prop}

\begin{rmk}\label{universal}
The moduli spaces $M^s_L$ always admit quasi-universal families \cite[Section 4.6]{HL}. When $n$ is odd, or when $L$ is in a chamber, 
the moduli spaces $M^s_L$ admit universal families. This follows for instance from \cite[Corollary 4.6.7]{HL} by taking a suitable polarization in the chamber. 
\end{rmk}

We now want to describe the variation of the $L$-stability condition when the polarization $L$ crosses a wall.  
To fix notation, suppose that $\cC^+$ and $\cC^-$ are neighbouring chambers separated by a wall $(2D+K)^{\perp}$.
We assume that $\cC^-\subset (2D+K)^{<0}$ and $\cC^+\subset (2D+K)^{>0}$. 
The general case is described in \cite[Proposition 2.7(6)]{Ellingsrud_Gottsche}: any torsion free sheaf $E$ with (\ref{data}) that is 
$\cC^-$-stable but not $\cC^+$-stable must fit into an exact sequence as follows
$$
0\ra \cO(D)\otimes\cI_{Z_1}\ra E\ra \cO(-K-D)\otimes\cI_{Z_2}\ra0,
$$
where $Z_1, Z_2\subset S$ are either empty or $0$-dimensional subschemes whose lenghts satisfy
$$
\length(Z_1)+\length(Z_2)=2+D\cdot (D+K)\geq 0. 
$$
This last inequality justifies condition~(1) of Definition~\ref{def:wall}.

We will be interested in walls for which there is a polarization $L\in (2D+K)^{\perp}$ with $K \cdot L < 0$.
We will show in Lemma~\ref{good_walls} below that in this case we have $2+D\cdot (D+K)=0$,
and so $Z_1 =Z_2=\emptyset$ in the exact sequence above. 
In particular, the extension $E$ is locally free. 
This special feature will greatly simplify the analysis of the variation of the moduli spaces $M_L$ when $L \cdot K < 0$.

\begin{defn}\label{K_negative_walls}
A \emph{$K$-negative wall} is a stability wall $(2D+K)^{\perp}$ containing a polarization 
$L\in (2D+K)^{\perp}$ with $K \cdot L < 0$. 

A \emph{$K$-negative chamber} is a stability chamber $\cC$ such that $\cC\subset (K)^{<0}$. 
\end{defn}

\begin{rmk}
We note that a chamber $\cC$ satisfies either $\cC\subset (K)^{<0}$ or $\cC\subset (K)^{>0}$. 
Indeed, the ample cone $\Amp(S)$ intersects the half space $K^{<0}$ (see Remark \ref{rmk:emptyM_L}), and 
$K^{\perp}$ is a wall if and only if $\Amp(S)\nsubseteq K^{\leq 0}$, i.e., 
$-K$ is not pseudoeffective. 
If $-K$ is pseudoeffective, then $\Amp(S)\subseteq K^{<0}$ and all chambers are $K$-negative.
\end{rmk}

For any $K$-negative chamber $\cC$, if the moduli space $M_{\cC}$ is nonempty, then it is
smooth of pure dimension $n$ by Lemma \ref{smoothness}.
We will see in Corollary~\ref{cor:irreducibility} that it is moreover irreducible and rational, and all torsion-free sheaves parametrized by $M_{\cC}$ are in fact locally free.

\begin{lemma} \label{good_walls}
Let $(2D+K)^{\perp}$ be a $K$-negative wall. Then
\begin{enumerate}
    \item $D\cdot (D+K) = -2$, and
    \item $-n\leq D^2 \leq 1$ (equivalently, $-n+2\leq -K\cdot D\leq 3$).
\end{enumerate}
\end{lemma}

\bp
Let $L\in (2D+K)^{\perp}$ be a polarization such that $K \cdot L < 0$. Then 
$$
D\cdot L=(-K-D)\cdot L=\frac{-K\cdot L}{2}>0.
$$
Since $L$ is ample, this implies that 
$$\h^0(-D)=0 \ \text{ and } \ \h^2(-D)=\h^0(K+D)=0,$$
and therefore 
$$\chi(-D)=-\h^1(-D)\leq0.$$
On the other hand, by Riemann-Roch and condition~(1) of Definition~\ref{def:wall}, we must have
$$\chi(-D)=1+ \frac{1}{2}D\cdot(D+K)\geq 0.$$
This shows that $\chi(-D)=0$, which gives (1).

To show (2), we analyze $\chi(2D+K)$ and $\chi(-2D-K)$. 
There are polarizations $L^+$ and $L^-$ such that $(2D+K)\cdot L^+>0$ and $(2D+K)\cdot L^-<0$.
Hence, we have $H^0(2D+K)=0=H^0(-2D-K)$. 
By Serre duality and the fact that $D\cdot L=(-K-D)\cdot L>0$, we have 
$$
\HH^2(2D+K)=\HH^0(-2D)=0\ , \ \text{ and}
$$
$$
\HH^2(-2D-K)=\HH^0(2D+2K)=0.
$$
Finally, by Riemann-Roch, we have: 
\begin{align}\label{dim_PD}
&0\geq -\h^1(2D+K)=\chi(2D+K)=D^2-1 \quad \text{and} \\
&0\geq -\h^1(-2D-K)=\chi(-2D-K)=-D^2-n, \notag
\end{align}
which gives (2).
\ep

The $K$-negative walls belong to the class of \emph{good walls} studied in \cite{Ellingsrud_Gottsche}.
When the polarization $L$ crosses a good wall, the variation of the moduli space $M_L$ can be described precisely. 
The following is \cite[Proposition 2.7(5)-(6)]{Ellingsrud_Gottsche} in our setting (\ref{data}).

\begin{prop}\label{description_of_walls}
Let $\cC^+$ and $\cC^-$ be neighbouring chambers separated by a $K$-negative wall $(2D+K)^{\perp}$, with
$\cC^-\subset (2D+K)^{<0}$ and $\cC^+\subset (2D+K)^{>0}$. 
    \begin{enumerate}
\item The torsion-free sheaves with (\ref{data}) that are $\cC^-$-stable but not $\cC^+$-stable 
are precisely all the locally free sheaves $E$ that fit into a non-split exact sequence of the form 
$$
0\ra \cO(D)\ra E\ra \cO(-K-D)\ra0.
$$    
\item The torsion-free sheaves with (\ref{data}) that are $\cC^+$-stable but not $\cC^-$-stable 
are precisely the locally free sheaves $E'$ that fit into a non-split exact sequence of the form 
$$
0\ra \cO(-K-D)\ra E'\ra \cO(D)\ra0.
$$    
\end{enumerate}
\end{prop}

\begin{rmk}\label{unique_extension}
Suppose that a torsion-free sheaf $E$ with (\ref{data}) fits into the exact sequence described in  Proposition~\ref{description_of_walls}(1). 
By \cite[Proposition 2.7(4)]{Ellingsrud_Gottsche}, this is the Harder-Narasimhan filtration of $E$ with respect to 
any $L\in \cC^+$. In particular, the subsheaf $\cO(D)\subset E$ is the unique subsheaf of $E$ that $\cC^+$-destabilizes it.
\end{rmk}

Proposition~\ref{description_of_walls} leads us to consider the projective spaces parametrizing the non-split extensions described above. The dimensions of these projective spaces can be computed from (\ref{dim_PD}).

\begin{defn} \label{defn:special_loci}\label{some special walls}
Let $(2D+K)^{\perp}$ be a $K$-negative wall.
We set
$$\PP_D:=\PP\Ext^1\big(\cO(-K-D),\cO(D)\big)=\PP\HH^1(2D+K)\cong \PP^{-D^2},$$ %\subset M_{\cC^-}  
$$\PP_{-K-D}:=\PP\Ext^1\big(\cO(D),\cO(-K-D)\big)=\PP\HH^1(-2D-K)\cong \PP^{n+D^2-1}. $$ %\subset M_{\cC^+}
(Here we put $\PP^{-1}=\emptyset$ and $\dim(\emptyset)=-1$ by convention.)
Note that 
$$
\dim(\PP_D)+\dim(\PP_{-K-D})=n-1.
$$
In particular:
\bi
\item  If $D^2=1$, then 
$$\PP_D=\emptyset,\quad \PP_{-K-D}=\PP^{n}.$$
\item  If $D^2=0$, then 
$$\PP_D=pt,\quad \PP_{-K-D}=\PP^{n-1}.$$
\item If $D^2=-1$, then 
$$\PP_D=\PP^1,\quad \PP_{-K-D}=\PP^{n-2}.$$
\ei
\end{defn}

The following theorem describes how the moduli spaces $M_L$ vary when the polarization $L$ crosses a  $K$-negative wall. 
It is a strengthening of Proposition~\ref{description_of_walls} and a special instance of \cite[Proposition 4.4 and Theorem 5.3]{Ellingsrud_Gottsche} in our setting (\ref{data}).

\begin{thm}\label{prop:special_loci}
Let $(2D+K)^{\perp}$ be a $K$-negative wall separating two neighbouring chambers $\cC^-\subset (2D+K)^{<0}$ and $\cC^+\subset (2D+K)^{>0}$. 
\bi
\item If $D^2=1$, then 
$$\PP_D=M_{\cC^-}=\emptyset,\quad \PP_{-K-D}=M_{\cC^+}=\PP^{n}.$$
\item If $D^2=-n$, then 
$$\PP_D=M_{\cC^-}=\PP^{n},\quad \PP_{-K-D}=M_{\cC^+}=\emptyset.$$
\item If $-n< D^2 < 1$, then the following holds.
\begin{enumerate}
    \item There are closed embeddings 
    $$
    \PP_D \hookrightarrow M_{\cC^-},\quad  \PP_{-K-D} \hookrightarrow M_{\cC^+}.
    $$
    \item The normal bundles of these special loci are 
    $$
    N_{\PP_D/M_{\cC^-}}\cong \cO_{\PP_D}(-1)^{\oplus q_D}, \quad  N_{\PP_{-K-D}/M_{\cC^+}}\cong \cO_{\PP_{-K-D}}(-1)^{\oplus q_{-K-D}},
    $$
    where $q_D$ denotes the codimension of $\PP_D$ in $M_{\cC^-}$, and $q_{-K-D}$ denotes the codimension of $\PP_{-K-D}$ in $M_{\cC^+}$.
    \item There is an isomorphism 
    $$
    \Bl_{\PP_D}M_{\cC^-} \ \cong \ \Bl_{\PP_{-K-D}}M_{\cC^+}.
    $$
    In particular, if $D^2=0$, then $M_{\cC^+}$ is the blowup of $M_{\cC^-}$ at the point $\PP_{D}$, with exceptional divisor $\PP_{-K-D}$, and if $D^2=-n+1$, then $M_{\cC^-}$ is the blowup of $M_{\cC^+}$ at the point $\PP_{-K-D}$, with exceptional divisor $\PP_{D}$.
\end{enumerate}
\ei
\end{thm}

\bp
Assume that $D^2=1$. Then $D\cdot K=-3$ and $\chi(D)=3$. As $(2D+K)\cdot L_0=0$, we have that $(K-D)\cdot L_0=\frac{K\cdot L_0}{2}<0$. It follows that $\HH^0(K-D)=0$, hence $\HH^2(D)=0$ by Serre duality. As $\chi(D)>0$ it follows that $D$ is an effective divisor class and one can apply Lemma \ref{sstable vb} with $h'=D$ to obtain that
$M_{\cC^-}=\emptyset$.  Any $\cC_{+}$-stable torsion free sheaf $E$ is therefore not $\cC_{-}$-stable, hence by Proposition \ref{description_of_walls}(2) we have that $[E]\in\PP_{-K-D}$. It follows that $M_{\cC_{+}}=\PP_{-K-D}$. The case $D^2=-n$ is analogous. The rest of the proposition follows from   \cite[Proposition 4.4 and Theorem 5.3]{Ellingsrud_Gottsche}. 
\ep

\begin{rmk}\label{unique_extension2}
Suppose that $(2D+K)^{\perp}$ and $(2D'+K)^{\perp}$ are distinct $K$-negative walls bounding a stability chamber $\cC$.
By replacing $D$ with $-K-D$ if necessary, and similarly for $D'$, we may assume that $\cC\subset (2D+K)^{<0}\cap (2D'+K)^{<0}$, and thus $\PP_D,\PP_{D'}\subset M_{\cC}$. 
By Remark~\ref{unique_extension}, $\PP_D\cap \PP_{D'}=\emptyset$.
\end{rmk}

\begin{cor}\label{cor:irreducibility}
Let $\cC\subset \Amp(S)$ be a $K$-negative stability chamber such that $M_{\cC}\neq \emptyset$.
Then $M_{\cC}$ is an irreducible and rational smooth projective variety of dimension $n$.
Moreover, every torsion-free sheaf parametrized by $M_{\cC}$ is locally free. 
\end{cor}

\bp
The fact that $M_{\cC}$ is smooth of pure dimension $n$ follows from Lemma~\ref{smoothness}.
By Lemma~\ref{sstable vb}, $\cC\subset (2h+K)^{>0}$ while, 
by Remark~\ref{rmk:emptyM_L}, there is a $K$-negative chamber $\cC_0\subset (2h+K)^{<0}$, so that $M_{\cC_0}= \emptyset$. 
We may take $\cC_0$ bounded by the wall $(2h+K)^{\perp}$. 
Let $L_0\in \cC_0$ and $L_1\in \cC$ be general polarizations, so that the line segment $\{L_t\,|\, 0\leq t\leq1\}\subset \Amp(S)$ joining 
$L_0$ and $L_1$ crosses finitely many $K$-negative walls $(2D_i+K)^{\perp}$ at distinct points $L_{t_i}$, $0<t_1<\dots <t_m<1$. 
We denote by $\cC_i$ the chamber containing $L_t$ for $t_i<t<t_{i+1}$, set $\cC_{m+1}=\cC$, and choose $D_i$ so that 
$\cC_i\subset (2D_i+K)^{<0}$ and $\cC_{i+1}\subset (2D_i+K)^{>0}$. 
In particular, we take $D_1=h$, and so $M_{\cC_1}=\PP^{n}$ by Theorem~\ref{prop:special_loci}.
We then apply Theorem~\ref{prop:special_loci} and increasing induction on $i$, to conclude that $M_{\cC_i}$
is irreducible and birational to $M_{\cC_1}=\PP^{n}$ for every $i\in\{1, \dots, m+1\}$. 
Moreover, by Proposition~\ref{description_of_walls} and increasing induction on $i$, we see that 
every torsion-free sheaf parametrized by $M_{\cC_i}$ is locally free. 
\ep

We will also need to understand the moduli spaces $M_{L_0}$ and $M^{\mu ss}_{L_0}$ for polarizations $L_0$ lying on a 
$K$-negative wall $(2D+K)^{\perp}$ separating two neighbouring chambers $\cC^-$ and $\cC^+$, 
and compare these moduli spaces with $M_{\cC^-}$ and $M_{\cC^+}$.
We will be interested in the cases when both $M_{\cC^-}$ and $M_{\cC^+}$ are nonempty. 
This is equivalent to assuming that $-n< D^2 < 1$.

\begin{prop}\label{reduction map}
Let $(2D+K)^{\perp}$ be a $K$-negative wall separating two neighbouring $K$-negative chambers 
$\cC^-\subset (2D+K)^{<0}$ and $\cC^+\subset (2D+K)^{>0}$. 
Let $L_0\in (2D+K)^{\perp}$ be a polarization lying in the closures of the chambers $\cC^-$ and $\cC^+$, and not
lying in any other wall. Then the following holds. 
\begin{enumerate}
    \item Suppose that $-n< D^2 < 1$, i.e., both $M_{\cC_{-}}$ and $M_{\cC_{+}}$ are non-empty. 
    Then there are morphisms 
$$
\phi_-: M_{\cC^-}\ra M^{\mu ss}_{L_0}, \quad \phi_+: M_{\cC^+}\ra M^{\mu ss}_{L_0}
$$
such that 
$$\phi_-(\PP_{D})=\phi_+(\PP_{-K-D})=pt=M_{L_0}^{\mu ss}\setminus M_{L_0}^{\mu s}.$$

    \item 
    \bi
    \item If $D^2<\frac{1-n}{2}$, then $M_{\cC^-}\cong M_{L_0}=M^s_{L_0}$.
    %\item If $D^2=\frac{1-n}{2}$, then $M_{L_0}= M^{\mu ss}_{L_0}$
    \item If $D^2>\frac{1-n}{2}$, then $M_{\cC^+}\cong M_{L_0}=M^s_{L_0}$.
    \item If $n$ is odd and $D^2=\frac{1-n}{2}$, then there are morphisms 
   $$\psi_-: M_{\cC^-}\ra M_{L_0}, \quad \psi_+: M_{\cC^+}\ra M_{L_0}, \quad f: M_{L_0}\ra  M^{\mu ss}_{L_0},$$
such that $\phi_{-}=f\circ\psi_{-}$, $\phi_+=f\circ\psi_+$ and 
$$\psi_-(\PP_{D})=\psi_+(\PP_{-K-D})=pt\in M_{L_0}\setminus M^s_{L_0}.$$
The maps $\psi_{-}$ and $\psi_+$ are isomorphisms over $M^s_{L_0}$. 
\ei
\end{enumerate}
\end{prop}

\bp
We start by observing that $K\cdot L_0<0$ and that any $F$ as in (\ref{data}) that is $\cC^{-}$-stable or $\cC^{+}$-stable is necessarily $\mu_{L_0}$-semistable. 
Recall that the moduli spaces $M_{\cC^-}$ and $M_{\cC^+}$ admit universal families (see Remark \ref{universal}), and that 
$M^{\mu ss}_{L_0}$ has the property that for any flat family of $\mu_{L_0}$-semistable sheaves over a scheme $T$, there is a canonical morphism $T\ra M^{\mu ss}_{L_0}$ (see \cite[Section 8.2]{HL}). 
Therefore, by taking universal families over $M_{\cC^-}$ and $M_{\cC^+}$, we obtain canonical morphisms $\phi_-: M_{\cC^-}\ra M^{\mu ss}_{L_0}$ and $\phi_+: M_{\cC^+}\ra M^{\mu ss}_{L_0}$.

To prove that $\phi_-(\PP_{D})=\phi_+(\PP_{-K-D})=pt=M_{L_0}^{\mu ss}\setminus M_{L_0}^{\mu s}$, 
we first prove the following. 

\begin{claim}\label{claim}
Let $E$ be a strictly  $\mu_{L_0}$-semistable torsion-free sheaf with (\ref{data}).
Then exactly one of the following conditions holds.
\begin{itemize}
    \item Either $E\cong \cO(D)\oplus \cO(-K-D)$ (in particular, $E$ is neither $\cC^-$-stable nor $\cC^+$-stable), or
    \item $E$ corresponds to a point in $\PP_D\subset M_{\cC^-}$ and is not $\cC^+$-stable, or
    \item $E$ corresponds to a point in $\PP_{-K-D}\subset M_{\cC^+}$ and is not $\cC^-$-stable.

\end{itemize}
\end{claim}

\begin{proof}[Proof of the claim]
By assumption, there is a rank $1$ (saturated) subsheaf $F_0\subset E$ such that $\mu_{L_0}(F_0)=\mu_{L_0}(E)$.
We have that $F_0=\cO(D')\otimes \cI_{Z_1}$ and $E/F_0=\cO(-K-D')\otimes \cI_{Z_2}$, for some divisor $D'$ on $S$ and subschemes
$Z_1,Z_2\subset S$ that are either empty or $0$-dimensional. 
In other words, $E$ fits into an exact sequence of the form 
$$
0\ra \cO(D')\otimes\cI_{Z_1}\ra E\ra \cO(-K-D')\otimes\cI_{Z_2}\ra 0.
$$
The condition that $\mu_{L_0}(\cO(D')\otimes \cI_{Z_1})=\mu_{L_0}(E)$ implies that  
$L_0$ belongs to the $K$-negative wall $(2D'+K)^{\perp}$. 
By Lemma~\ref{good_walls}, $Z_1=Z_2=\emptyset$ and $D'\cdot (D'+K)=-2$.
Since by assumption $L_0$ lies on a unique wall, we must have $(2D'+K)=\lambda(2D+K)$ for some $\lambda\in\RR$. 
The condition that $D'\cdot (D'+K)=-2$ then implies that $\lambda=\pm 1$, i.e., $D'=D$ or $D'=-D-K$.
Therefore, either $E\cong \cO(D)\oplus \cO(-K-D)$, or $E$ fits into a non-split exact sequence
described in Proposition~\ref{description_of_walls} (1) or (2).
In case (1), $E$ is $\cC^-$-stable and not $\cC^+$-stable, and it corresponds to a point in $\PP_D\subset M_{\cC^-}$. 
In case (2), $E$ is $\cC^+$-stable and not $\cC^-$-stable, and it corresponds to a point in $\PP_{-K-D}\subset M_{\cC^+}$. 
\end{proof}

The sheaf $\cO(D)\oplus \cO(-K-D)$ and any torsion-free sheaf $E$  in $\PP_D\subset M_{\cC^-}$ or in $\PP_{-K-D}\subset M_{\cC^+}$ correspond to the same point in $M^{\mu ss}_{L_0}$, as their $\mu$-Jordan-H\"older filtrations have factors $\cO(D)$ and $\cO(-K-D)$ (see \cite[Theorem 8.2.11]{HL}). This observation together with Claim~\ref{claim} finishes the proof of (1). 

Now we proceed to prove (2).
Consider first the case when $D^2<\frac{1-n}{2}$.
We will show that a torsion-free sheaf $E$ with (\ref{data}) is $\cC^-$-stable if and only if it is $L_0$-stable.
This implies that $M_{\cC^-}$ and $M_{L_0}$ are coarse moduli spaces for the same moduli functor, and thus $M_{\cC^-}\cong M_{L_0}=M^s_{L_0}$.

Suppose first that $E$ is $\cC^-$-stable, and thus $\mu_{L_0}$-semistable. 
If $E$ is $\mu_{L_0}$-stable, then it is $L_0$-stable.
So we may assume that $E$ is strictly $\mu_{L_0}$-semistable, and Claim~\ref{claim} implies that 
$E$ fits into a non-split exact sequence
$$
    0\ra \cO(D)\ra E\ra \cO(-K-D)\ra0.
$$    
Recall from Remark~\ref{rmk:stability} that $E$ is $L_0$-stable if and only if it is
$\mu_{L_0}$-semistable and $\chi(F)<\frac{\chi(E)}{2}$ for any rank $1$ (saturated) subsheaf $F\subset E$ 
such that  $\mu_{L_0}(F)= \mu_{L_0}(E)$. 
As in the proof of the claim, we see that the only rank $1$ saturated subsheaf $F\subset E$ such that  $\mu_{L_0}(F)= \mu_{L_0}(E)$ is $F\cong\cO(D)$. 
We easily calculate $\chi\big(\cO(D)\big)$ and $\chi(E)$ using Riemann-Roch.
The assumption that $D^2<\frac{1-n}{2}$ then gives 
$$
\chi\big(\cO(D)\big) = D^2 +2 <\frac{5-n}{2}=\frac{K^2}{2}=\frac{\chi(E)}{2},
$$
and therefore $E$ is ${L_0}$-stable.

Conversely, suppose that $E$ is $L_0$-stable, but not $\cC^-$-stable. Then $E$ is strictly $\mu_{L_0}$-semistable,  
and Claim~\ref{claim} implies that 
$E$ fits into an exact sequence
$$
0\ra \cO(-K-D)\ra E\ra \cO(D)\ra0.
$$    
Since $E$ is $L_0$-stable, we must have $\chi\big(\cO(-K-D)\big)<\frac{\chi(E)}{2}$.
Riemann-Roch then gives 
$$
\chi\big(\cO(-K-D)\big) = 3-n -D^2 <\frac{5-n}{2}=\frac{\chi(E)}{2},
$$
contradicting the assumption that $D^2<\frac{1-n}{2}$.

The cases $D^2>\frac{1-n}{2}$ and $D^2=\frac{1-n}{2}$ are treated analogously. Note that if $D^2=\frac{1-n}{2}$ then all torsion-free sheaves  in $\PP_D\subset M_{\cC^-}$ or in $\PP_{-K-D}\subset M_{\cC^+}$ correspond to the same point in $M_{L_0}$, as their Jordan-H\"older filtrations have factors $\cO(D)$ and $\cO(-K-D)$.
 
\ep

\begin{rmk}\label{flips&flops}
Let $(2D+K)^{\perp}$, $\cC^-$, $\cC^+$ and $L_0$ be as in Proposition~\ref{reduction map},
and suppose that $-n+2 \leq D^2 \leq -1$.
Then the composed map $\phi_+\circ \phi_-:M_{\cC^-}\dashrightarrow M_{\cC^+}$ is an isomorphism in codimension one that fits into the following diagram.
$$
    \xymatrix@R=.5cm@C=.5cm{
    \PP_D \subset M_{\cC^-} \  \ar[rd]_{\phi_-} \ar@{-->}[rr]^{\phi_+\circ \phi_-} && 
    \   M_{\cC^+}\supset \PP_{-K-D} \ar[ld]^{\phi_+}\\
    & M^{\mu ss}_{L_0}
    }
$$
From the description of the normal bundles of the exceptional loci $\PP_D \subset M_{\cC^-}$ and $\PP_{-K-D} \subset M_{\cC^+}$ in Theorem~\ref{prop:special_loci}, one can easily compute the intersection numbers $-K_{\cC^-}\cdot \ell_D$ and $-K_{\cC^+}\cdot \ell_{-K-D}$, where $\ell_D$ denotes a line in $\PP_D$ and $\ell_{-K-D}$ denotes a line in $\PP_{-K-D}$.
It follows that 
$$
\phi_+\circ \phi_-:M_{\cC^-}\dashrightarrow M_{\cC^+} \quad \text{is a } \quad  
\left\{
    \begin{array}{l}
    \text{flip}  \\
    \text{flop}  \\
    \text{anti-flip}     
    \end{array}   
    \right.
\quad \text{if}  \quad 
\left\{
    \begin{array}{l}
    D^2<\frac{1-n}{2}  \\
    D^2=\frac{1-n}{2}  \\
    D^2>\frac{1-n}{2}     
    \end{array}   
    \right. \ .
$$
\end{rmk}

\section{Blowups of the projective plane revisited}\label{section:blowups2}

In this section, we revisit the blowups of the projective plane at Cremona-general points discussed in Section~\ref{subsection:blowups1}, in light of the stability wall-and-chamber decomposition of the ample cone described in the previous section.
Throughout this section, we denote by $S=\Bl_k\PP^2$ the blowup of $\PP^2$ at $k=n+4\geq 6$ Cremona-general points $Q_1, \dots, Q_{k}$, and follow Notation~\ref{notation_Pic(S)}.
In Subsection~\ref{subsection:classification_of_walls}, we provide a partial classification of $K$-negative walls of $\N^1(S)$. 
In Subsection~\ref{subsection:C_0}, we focus on a particular stability chamber $\cC_0\subset \Nef(S)^{K\leq0}$ and show that the corresponding moduli space $M_{\cC_0}$ is isomorphic to $\Bl_k\PP^n$. In particular, this proves Theorem~\ref{thm:moduli}.
In Subsection~\ref{subsection:special_cones}, we describe two special subcones $\Pi\subset \cE\subset \Nef(S)^{K\leq0}$ containing the chamber $\cC_0$. The relevance of these cones will become clear in the following sections: in Section~\ref{section:determinant}, we construct a special linear isomorphism $\rho:\N^1(S)\to \N^1(\Bl_{n+4}\PP^n)$, and in Section~\ref{section:proof_main_thm}, we show $\rho$ maps the cones 
$$
   \cC_0\subseteq \Pi \subseteq \cE\subseteq \N^1(S)^{K\leq0}
$$ 
bijectively onto the cones 
$$
\Nef(\Bl_{n+4}\PP^n)\subseteq \Mov(\Bl_{n+4}\PP^n)^{K\leq0}\subseteq  \Eff(\Bl_{n+4}\PP^n)^{K\leq0}\subseteq \N^1(\Bl_{n+4}\PP^n)^{K\leq0}.
$$

\subsection{A partial classification of $K$-negative walls}\label{subsection:classification_of_walls}

Recall from Definition~\ref{def:wall} and Lemma~\ref{good_walls} 
that $K$-negative (stability) walls are hyperplanes of the form $(2D+K)^{\perp}\subset \N^1(S)$,
with $D\in \Pic(S)$ satisfying $D\cdot (D+K)= -2$ and $-n\leq D^2 \leq 1$.
This motivates the following definition. 

\begin{defn}
A \emph{numerical rational class} is a divisor class $D\in \Pic(S)$ such that 
$$D^2+D\cdot K=-2 \quad \text{and} \quad -n\leq D^2 \leq 1.$$ 
A numerical rational class $D$ is called 
\bi
\item a \emph{numerical cubic} if $D\cdot (-K)=3$ ($\iff D^2=1$),
\item a \emph{numerical conic} if $D\cdot (-K)=2$ ($\iff D^2=0$),
\item a \emph{numerical line} if $D\cdot (-K)=1$ ($\iff D^2=-1$),
\item a \emph{$(-m)$-numerical rational class} if $D^2=-m\leq -2$.
\ei
\end{defn}

\begin{rmk}
Note that $D\in \Pic(S)$ is a numerical rational class if and only if so is $-K-D$.
A numerical rational class $D\in \Pic(S)$ defines a $K$-negative wall if and only if 
there is a polarization $L$ such that
$$(2D+K)\cdot L=0\quad \text{and} \quad -K\cdot L>0.$$
As we shall see below, not every numerical rational class defines a $K$-negative wall. 
If a numerical rational class $D\in \Pic(S)$ defines a $K$-negative wall,
then we must have 
$$H^0(2D+K)=0 \quad \text{ and }\ \quad H^0(K-D)=0.$$
\end{rmk}

The following result gives a partial classification of $K$-negative walls.

\begin{prop}\label{partial classif}
Let $D$ be a numerical rational class with $D^2\in\{1, 0, -1, -2\}$.
\begin{enumerate}
\item If $k\leq 9$, then $H^0(2D+K)=0$ and $H^0(K-D)=0$.
\item Suppose that $H^0(2D+K)=0$ and $H^0(K-D)=0$. 
(This condition holds whenever $(2D+K)^\perp$ is a $K$-negative wall.)
\begin{itemize}
\item If $D^2=1$, then $D\in \cW\cdot h$. 
\item If $D^2=0$, then $D\in \cW\cdot (h-e_i)$. 
\item If $D^2=-1$, then $D\in \cW\cdot (h-e_i-e_j)$ (or, equivalently $D\in \cW\cdot e_i$). 
\item If $D^2=-2$ and $k\leq 9$, then  
$D\in \cW\cdot (h-e_i-e_j-e_l)$ (or, equivalently, $D\in \cW\cdot (e_i-e_j)$).   
\end{itemize}
\end{enumerate}
In particular, for $k\leq9$, numerical lines, conics and cubics are completely classified. 
\end{prop}

\begin{rmk}
When $k\geq10$, there are numerical rational classes $D$ that do not satisfy $H^0(2D+K)=0$, in particular, they do not define $K$-negative walls. Consider for instance the (effective) numerical line $D=3h-e_1-\ldots-e_9+e_{10}$. Note also that $D\notin \cW\cdot (h-e_i-e_j)$.  
\end{rmk}

\bp[{Proof of Proposition~\ref{partial classif}}]
To prove (1), note that, if $k\leq 9$, then $-K$ is nef and $K^2\geq0$.
Since $D$ is a numerical rational class with $D^2\in\{1, 0, -1, -2\}$, we must have $D\cdot K\leq0$.
Therefore we get that 
$$(2D+K)\cdot (-K)=-K^2+2D\cdot K\leq0,$$
$$(K-D)\cdot (-K)=-K^2+D\cdot K\leq0,$$
with both inequalities strict if $D\cdot K<0$ or $k\leq8$.
In either of these cases, we conclude that $H^0(2D+K)=0$ and $H^0(K-D)=0$.
So we are left with the case when $k=9$, $K^2=0$ and $D\cdot K=0$, i.e., $D^2=-2$.
Since $-K$ is effective, if $H^0(2D+K)\neq 0$, then $H^0(2D)\neq 0$. 
By Lemma~\ref{E positive surface case}(4), $D$ must be a multiple of $-K$, but in this case $D$ cannot be a numerical rational class. 
Similarly, if $H^0(K-D)\neq 0$, then $H^0(-D)\neq 0$.
Again, this would imply that $-D$ is a multiple of $-K$, which is not possible. 

Next we prove (2).
Since $D$ is a numerical rational class, we have $\chi(D)=D^2+2$. 
By assumption, we have $h^2(D)=h^0(K-D)=0$ and $h^0(2D+K)=0$.
If $D^2\in\{1,0,-1\}$, then Riemann-Roch gives that $h^0(D)>0$. %, and so $d\geq0$. 
Write 
$$D=D'+\sum_{j\in T}l_j e_j  \quad \text{and} \quad \ D'=dh-\sum_{i\in R}m_i e_i  \quad \text{with} \quad R\cap T=\emptyset, \quad l_j, m_i>0.$$
Note that $h^0(D)=h^0(D')$.
Moreover,
\begin{equation}\label{eqns}
\sum_{i\in R}m_i^2+\sum_{j\in T}l_j^2=d^2-D^2 \quad \text{and} \quad \sum_{i\in R}m_i-\sum_{j\in T}l_j=3d-D^2-2.
\end{equation}

Assume $d=0$. Then the cases $D^2=1,0$ are not possible. If $D^2=-1$, then $D=e_j$, while if $D^2=-2$, then $D=e_i-e_j$ ($i\neq j$) and the result follows. 

Assume $d=1$. The following statements follow from \eqref{eqns}.
\bi
\item If $D^2=1$, then $D=h$.
\item If $D^2=0$, then $D=h-e_i$.
\item If $D^2=-1$, then $D=h-e_i-e_j$.
\item If $D^2=-2$, then $D=h-e_i-e_j-e_l$.
%\item If $D^2=-3$, then $D=h-e_i-e_j-e_k-e_l$.
\ei

Assume now that $d\geq2$. 
We have 
$$
D'^2=D^2+\sum_{j\in T}l_j^2 \, , \quad D'\cdot K=D\cdot K+\sum_{j\in T}l_j  \quad \text{and}  \quad \chi(2D'+K)=-1+D^2+\sum_{j\in T}(2l_j^2+l_j) \, .
$$
Moreover, $h^0(2D'+K)=h^0(2D+K)=0$ and $h^2(2D'+K)=h^0(-2D')=0$ since $d>0$. 
It follows that  $\chi(2D'+K)\leq 0$, i.e., 
$$\sum_{j\in T}(2l_j^2+l_j)\leq 1-D^2.$$
Since $2l_j^2+l_j\geq3$, it follows that $T=\emptyset$ if $D^2\in\{1,0,-1\}$, and $|T|\leq1$ if $D^2=-2$. The case
$D^2=-2$, $|T|=1$ may occur if $k=n+4\geq11$ (see Remark \ref{special -2 classes} below), but we claim that it does not occur if $k=n+4\leq 10$. Indeed, if $D^2=-2$ and $|T|=1$, we may assume $T=\{k\}$ and we have $l_k=1$.
By (\ref{eqns}) and Cauchy-Schwarz inequality, we have 
$$(3d+1)^2=\big(\sum_{i=1}^{k-1}m_i\big)^2\leq (k-1)\big(\sum_{i=1}^{k-1}m_i^2\big)=(k-1)(d^2+1),$$
which leads to a contradiction if $k\leq 10$ and $d\geq2$. %The case $D^2=-3$, $|T|=1$ is similar. 

Assume from now on that $T=\emptyset$. 
We may assume without loss of generality 
$$m_1\geq m_2\geq\ldots \geq m_k\geq 0.$$ 
We observe that $2d\geq m_1+m_2+m_3$. This is clear if $h^0(D)>0$ (for example, if $D^2\in\{1,0,-1\}$) as $2h-e_1-e_2-e_3$ is a movable curve class. Assume now 
$D^2=-2$ and $2d<m_1+m_2+m_3$. We have
$$(2d+1)^2\leq (m_1+m_2+m_3)^2\leq 3(m_1^2+m_2^2+m_3^2)\leq 3\sum_{i=1}^k m_i^2\leq 3(d^2+2),$$
which is impossible if $d\geq 2$. 

We also observe that $m_3>0$. Indeed, if $m_3=0$, by (\ref{eqns}) we have
$$(3d-D^2-2)^2=(m_1+m_2)^2\leq 2(m_1^2+m_2^2)=2(d^2-D^2),$$
which cannot happen if $D^2\in\{1,0,-1, -2\}$ and $d\geq2$.

If $D^2\in\{1,0,-1, -2\}$, then Noether's inequality (see Lemma \ref{Nagata} below) gives that
$m_1+m_2+m_3>d$. We are now done by induction on $d$.
Indeed, let $w\in \cW$ correspond to the Cremona transformation centered at the points $Q_1, Q_2, Q_3$.
It lowers the $h$-degree of $D$ while keeping it non-negative: 
$$0\leq 2d-m_1-m_2-m_3=(w\cdot D)\cdot h<D\cdot h=d.$$

Assume now that $k=n+4=9$, $D^2=-2$ and $d=D\cdot h<0$. In this special case $-K-D$ is also a $(-2)$-numerical rational class, and $(-K-D)\cdot h>0$. It follows from the previous argument that, up to the action of $\cW$, we may assume that 
$D=-K-(h-e_1-e_2-e_3)=2h-e_4-\ldots-e_9$. If $w\in \cW$ corresponds to the Cremona transformation centered at the points $Q_4,Q_5,Q_6$, then $w\cdot D=h-e_7-e_8-e_9$, and hence $D\in\cW\cdot (h-e_1-e_2-e_3)$. 
The cases $k=6, 7, 8$ can be treated similarly. 
\ep

The following result is known as Noether's inequality. 
 
\begin{lemma}[{\cite[Section V.3, Proposition 5]{DO}}]\label{Nagata}
Let $D=dh-\sum_{i=1}^{k} m_ie_i\in \Pic(S)$, with $d\geq 2$, $m_1\geq m_2\geq\ldots\geq m_k\geq0$. %, and $m_3>0$. 
Suppose that $D$ is a numerical rational class with $D^2\in\{1,0,-1,-2\}$.
Then $m_1+m_2+m_3>d$.  
\end{lemma}

\begin{rmk}\label{special -2 classes}
The argument in the proof of Lemma~\ref{Nagata} given in \cite[Section V.3, Proposition 5]{DO} can be used to treat $(-2)$-numerical rational classes of the form 
$$D=dh-\sum_{i=1}^l m_ie_i +e_{l+1},$$
with $m_1\geq m_2\geq\ldots\geq m_l>0$, $l\geq3$. 
But in this case it only shows that $m_1+m_2+m_3\geq d$, and in fact the equality $m_1+m_2+m_3=d$ can occur.

For example, if $k\geq11$ then $D=3h-e_1-\ldots-e_{10}+e_{11}$ is a $(-2)$-numerical rational class for which no $w\in \cW$ decreases its $h$-degree. In particular, $D\notin\cW\cdot (h-e_1-e_2-e_3)$. Note, however, that $(2D+K)^\perp$ is not a $K$-negative wall. Indeed, if $(2D+K)\cdot L=0$ and $K\cdot L<0$, then $(D+K)\cdot L<0$. But this is not possible for an ample $L$ since $D+K=2e_{11}+e_{12}+\ldots +e_k$ is effective.
\end{rmk}

As we explained in Section~\ref{section:moduli}, the set of stability walls is locally finite on $\Amp(S)$ (\cite[Lemma 4.C.2]{HL}). 
We end this subsection by discussing the locally finiteness of $K$-negative stability walls through boundary points of $\Nef(S)$.

\begin{proposition}\label{locally finite}
Let $L\in \Nef(S)$ be such that $L^2>0$. Then there is an open neighborhood of $L$ intersecting finitely many $K$-negative stability walls.
\end{proposition}

\bp
Write $L=xh-\sum y_ie_i$, with $x, y_i\in\RR$, and suppose that it lies on a $K$-negative stability wall $(2D+K)^\perp$ defined by a numerical rational class 
\begin{equation}\label{rewritten D}
D=(1-d)h+\sum_{i=1}^k m_ie_i, \  d, m_i\in\ZZ.
\end{equation}
After replacing $D$ with $-K-D$ if necessary, we may assume that $d\geq0$.
The conditions that $D$ is a numerical rational class and $L\in (2D+K)^\perp$ are equivalent to: 
\begin{equation}\label{L on wall}
d^2+d=\sum(m_i^2+m_i) \quad \text{ and } \quad \sum(2m_i+1)y_i=(2d+1)x \, .
\end{equation}
Using the Cauchy-Schwartz inequality, we have that
$$(2d+1)^2x^2=\big(\sum(2m_i+1)y_i \big)^2\leq 
\sum(2m_i+1)^2\cdot \sum y_i^2= (4d^2+4d+k)\cdot \sum y_i^2.
$$

Suppose that $L^2=x^2-\sum y_i^2>0$. Then the inequality $(2d+1)^2x^2\leq (4d^2+4d+k)\cdot \sum y_i^2$ above gives: 
\begin{equation}\label{bound on d}
d^2+d\leq \frac{k\cdot\sum y_i^2-x^2}{4L^2}.
\end{equation}
Hence, for a fixed $L$ with $L^2>0$, there is an open neighborhood $\cU$ of $L$ satisfying the following condition. 
Any $K$-negative stability wall $(2D+K)^\perp$ intersecting $\cU$ is defined by a numerical rational class as in \eqref{rewritten D}
with the integer $d$ is bounded. Hence the integers $m_i$ are also bounded. 
This shows that there are only finitely many $K$-negative stability walls intersecting $\cU$.
\ep

When $L\in \Nef(S)$ is a boundary point with $L^2=0$, there may be infinitely many $K$-negative stability walls through $L$. 
The next result addresses this case.

\begin{proposition}\label{locally finite 2}
Let $L\in \Pic(S)$ be a primitive class such that $L\in\Nef(S)^{K\leq0}$ and $L^2=0$.
Then, 
either $L\in\cW\cdot (h-e_1)$ and $L$ belongs to finitely many stability walls, or  $L\in\cW\cdot(3h-e_1-\ldots-e_9)$ and $L$ lies on infinitely many stability walls. 

\end{proposition}

\bp
By Lemma~\ref{ample K-negative}(4) and Remark~\ref{FDomain}, either $L\in\cW\cdot (h-e_1)$ or  $L\in\cW\cdot(3h-e_1-\ldots-e_9)$.
Up to the action of $\cW$, we may assume that either $L=h-e_1$ or $L=3h-e_1-\ldots-e_9$.

Suppose $L=h-e_1$. 
Let $(2D+K)^\perp$ be a $K$-negative stability wall through $L$, defined by a numerical rational class as in \eqref{rewritten D}.
Then \eqref{L on wall} implies that $d=m_1$ and there is a decomposition $S\sqcup S^c=\{2,\ldots,k\}$ such that 
$m_i=-1$ for $i\in S$ and $m_j=0$ for $j\in S^c$.
We have:
$$2D+K=-(2d+1)h+(2d+1)e_1-\sum_{i\in S}e_i+\sum_{j\in S^c}e_j.$$  
Since $(2D+K)^\perp$ is a stability wall, $A\in(2D+K)^\perp$ for some ample line bundle $A=ah-\sum b_i e_i$, so we have
$$(a-b_1)(2d+1)=\sum_{j\in S^c}b_i-
\sum_{i\in S}b_i\leq\sum_{i=2}^k b_i.$$
Since $A$ is ample, we have $A\cdot (h-e_1-e_i)>0$ for every $i\geq2$.
This gives $a-b_1>b_i$ for every $i\geq2$, and hence $(a-b_1)(k-1)>\sum_{i=2}^kb_i$. 
It follows that $(2d+1)<k-1$ and, as before, the integers $m_i$ are also bounded. 
This implies that $L=h-e_1$ lies on finitely many stability walls.  

Assume now that $L=3h-e_1-\ldots-e_9$ (in particular, $k\geq 9$).
Then $L\cdot K=0$, and hence $L\in(2D+K)^\perp$ for any $(-2)$-numerical rational class $D\in\cW_9\cdot (h-e_1-e_2-e_3)$, where $\cW_9\subset \cW$ is the Weyl group associated with $\Bl_9\PP^2$. 
\ep

\subsection{A special chamber $\cC_0\subset \Nef(S)^{K\leq0}$}\label{subsection:C_0}

In this subsection, we examine a special stability chamber $\cC_0\subset \Nef(S)^{K\leq0}$, namely, the chamber containing the polarization 
$L=(n+1)h-\sum e_i$.
We will show that the corresponding moduli space $M_{\cC_0}$ is isomorphic to $\Bl_k\PP^n$.
More precisely, let $S$ be the blowup of $\PP^2$ at $k=n+4\geq 6$ Cremona-general points $Q_1, \dots, Q_{k}$, and let $P_1, \dots, P_{k}\in \PP^n$ be the Gale dual points. We will show that $M_{\cC_0}$ is isomorphic to the blowup of $\PP^n$ at the $k$ Cremona-general points $P_1, \dots, P_{k}$, thereby establishing Theorem~\ref{thm:moduli}.

\begin{notn}\label{L_a}
Consider the following family of $\QQ$-divisor classes:
\[
L_a:=-K+ah=(a+3)h-\sum_{i=1}^{k} e_i \quad \text{ with } \quad  a\in\QQ_{\geq0}.
\]
\end{notn}

Note that $L_a\in \Amp(S)^{K<0}$  for $a>\max\{\frac{k}{3}-3,0\}$ by Corollary~\ref{La}.
By Lemma \ref{sstable vb}, $M_{L_a}=\emptyset$ if $a>n+1$.
We will determine all the stability walls crossed by the segment $\big\{L_a\,|\, \max\{\frac{k}{3}-3,0\}< a \leq n+1\big\}$ as $a$ decreases from $n+1$ to $\max\{\frac{k}{3}-3,0\}$. 

\begin{lemma}\label{line segment}
The line segment $\big\{L_a\,|\, \max\{\frac{k}{3}-3,0\}<a \leq n+1=k-3\big\}\subset \Nef(S)^{K\leq 0}$ intersects the stability wall $(2D+K)^\perp$ for the following numerical rational classes $D$, and at the following values of $a$:
\bi
\item $D=h-\sum_{i\in T}e_i$, for each subset $T\subset \{1, \dots, k\}$ with $|T|<\min\{\frac{k}{3},\frac{k-3}{2}\}$, at $a=k-3-2|T|$.
\item $D=e_i$, for $i\in \{1, \dots, k\}$, at $a=\frac{k-7}{3}$ ($k\geq8$).
\ei
\end{lemma}

\bp
In order to determine all the stability walls $(2D+K)^\perp$ that contain some polarization $L_a$ as above, 
we proceed as in the proof of Proposition~\ref{locally finite}. Write
$$
D=(1-d)h+\sum m_i e_i \quad \text{ with } \quad d\geq0 \, . \quad 
$$
The conditions that $D$ is a numerical rational class and $L_a\in (2D+K)^\perp$ are equivalent to: 
\begin{equation}\label{L on wall - special case}
d^2+d=\sum(m_i^2+m_i) \quad \text{ and } \quad (2d+1)(a+3)=\sum(2m_i+1) \, .
\end{equation}
In the present case, inequality (\ref{bound on d}) becomes:
\begin{equation}\label{bound on d - special case}
d^2+d\leq \frac{k^2-(a+3)^2}{4\big((a+3)^2-k\big)}.
\end{equation}

The idea of the proof is to use (\ref{bound on d - special case}) to bound $d$ from above in terms of $a$. We make this bound explicit in our range of interest, that is, for $\max\{\frac{k}{3}-3,0\}< a\leq k-3$. Subsequently, for a given $d$, we determine all the possibilities for the integers $m_i$ using the relation $d^2+d=\sum (m_i^2+m_i)$. 

Letting $x=a+3$ in the right hand side of (\ref{bound on d - special case}) gives the function 
$$f(x)=\frac{k^2-x^2}{4\big(x^2-k\big)}.$$ 
The case $k=9$ will be analized separately and we will prove that $d\leq5$. If $k\neq 9$, the function $f(x)$
is decreasing on the interval $\max\{\frac{k}{3},3\}\leq x\leq k$. If $k\geq10$ the maximum is obtained at $x=\frac{k}{3}$, so we have $d^2+d< f\big(\frac{k}{3}\big)=2+\frac{18}{k-9}$. Similarly, if $k\in\{6,7,8\}$ the maximum is obtained at $x=3$. This gives the following bounds on $d$: 
\bi
\item $d\in\{0, 1\}$ if $k\geq14$, or $k\in\{6, 7\}$,
\item $d\in\{0, 1, 2\}$ if $k\in\{11, 12, 13\}$, 
\item $d\in\{0, 1, 2, 3\}$ if $k\in\{8,10\}$, 
\item $d\in\{0, 1, 2, 3, 4, 5\}$ if $k=9$. %and $a>0$.
\ei

Consider the case $k=9$. We are interested in the stability walls $(2D+K)^\perp$ that the line $L_a$ intersects for $0< a\leq 6$. 
Set $l_i:=2m_i+1$. By (\ref{L on wall - special case}) we have 
$$\sum l_i=(2d+1)(a+3),\quad \sum l_i^2=4d^2+4d+9.$$
By the Cauchy-Schwartz inequality $(\sum l_i)^2\leq 9\sum l_i^2$ we obtain:
\begin{equation}\label{k=9 bound a}
(a+3)^2\leq \frac{9(4d^2+4d+9)}{4d^2+4d+1}=9+\frac{72}{4d^2+4d+1}.
\end{equation}
By (\ref{L on wall - special case}) we have
$a\in\frac{1}{(2d+1)}\ZZ$, so we may write $a=\frac{u}{2d+1}$, for some $u\in\ZZ$, $u>0$. Substituting $a=\frac{u}{2d+1}$ in (\ref{k=9 bound a}) we obtain $u^2+6(2d+1)u\leq72$. As $u>0$ is an integer, it follows that $6(2d+1)<72$, so $d\leq5$. 

Suppose now $k\geq6$ is arbitrary. For $d\in\{0, 1, 2, 3, 4, 5\}$ we determine the possible values of $a$ and $D$ from (\ref{L on wall - special case}). In what follows we denote 
$T:=\{\,i\, |\, m_i=-1\}\subseteq \{1,\ldots,k\}$, $R:=\{\,i\, |\, m_i=0\}\subseteq \{1,\ldots,k\}$. 

\underline{Case $d=0$}. Then $m_i\in\{0,-1\}$ for all $i$ and
$$D=h-\sum_{i\in T}e_i,\quad a=k-3-2|T|
.$$
Since $\max\{\frac{k}{3}-3,0\}< a\leq k-3$ we must have 
$|T|<\min\{\frac{k}{3},\frac{k-3}{2}\}$. 

\underline{Case $d=1$}. Up to relabeling, it follows that $m_1\in\{1,-2\}$ and $m_i\in\{0,-1\}$ for all $i\neq1$, so 
$\{2,\ldots,k\}=T\sqcup R$.
Consider first the case $m_1=1$. Then 
$D=e_1-\sum_{i\in T}e_i$ and by (\ref{L on wall - special case}) we have that $3(a+3)=3+|R|-|T|=k+2-2|T|$.
Since $a>\frac{k}{3}-3$, it follows that $T=\emptyset$. The condition $a>0$ implies that $k\geq8$. Assume now that $m_1=-2$. Similarly, we obtain that 
$3(a+3)=-3+|R|-|T|=k-4-2|T|$, but this is a contradiction with $a>\frac{k}{3}-3$. 

\underline{Case $d=2$ ($k\in\{8,\ldots, 13\}$)}. Up to relabeling, it follows that either (a) $m_1\in\{2,-3\}$ and $m_i\in\{0,-1\}$ for all $i\neq1$, or (b) $m_1, m_2, m_3\in\{1,-2\}$ and $m_i\in\{0,-1\}$ for all $i\neq1,2,3$. 
Consider the case (a). 
By (\ref{L on wall - special case}) we have $5(a+3)=(2m_1+1)+|R|-|T|=2m_1+k-2|T|\leq k+4-2|T|$. 
But this is impossible when $\max\{\frac{k}{3}-3,0\}< a$. 
Now consider the case (b). By (\ref{L on wall - special case}) we have $5(a+3)=\sum_{i=1,2,3}(2m_i+1)+|R|-|T|=2\sum_{i=1,2,3}m_i+k-2|T|\leq k+6-2|T|$. Again, this is impossible when $\max\{\frac{k}{3}-3,0\}< a$.
%It follows that $k=9$, $a=0$, $D=-h+e_1+e_2+e_3$, a numerical rational class with $D^2=-2$. 

\underline{Case $d=3$ ($k\in\{8,9,10\}$)}. Up to relabeling, one of the following holds:
\bi
\item[(a) ] $m_1\in\{3,-4\}$ and $m_i\in\{0,-1\}$ for all $i>1$,
\item[(b) ] $m_1, m_2\in\{2,-3\}$ and $m_i\in\{0,-1\}$ for all $i>2$,
\item[(c) ] $m_1\in\{2,-3\}$, $m_2, m_3, m_4\in\{1,-2\}$ and $m_i\in\{0,-1\}$ for all $i>4$.
\item[(d) ] $m_1,\ldots, m_6\in\{1,-2\}$ and $m_i\in\{0,-1\}$ for all $i>6$.
\ei

By (\ref{L on wall - special case}) we have $7(a+3)=2(\sum m_i)+k$. As $a>0$, $k\leq10$, it follows that $2\sum m_i>21-k\geq 11$, i.e., $\sum m_i\geq6$. This can only happen in case (d) when $k=10$, $m_1=\ldots=m_6=1$, $m_i=0$ for all $i>6$. But in this case $7(a+3)=2(\sum m_i)+k=22$ and 
the condition $\frac{1}{3}=\frac{k}{3}-3< a$ leads to a contradiction. 

\begin{comment}
Consider case (a). By (\ref{L on wall - special case}) we have $7(a+3)=(2m_1+1)+|R|-|T|=2m_1+k-2|T|\leq k+6-2|T|$. Similarly, in case (b) we have
$7(a+3)=\sum_{i=1,2}(2m_i+1)+|R|-|T|=2\sum_{i=1,2}m_i+k-2|T|\leq k+8-2|T|$, while in case (c) we obtain 
$7(a+3)=\sum_{i=1}^4(2m_i+1)+|R|-|T|=2\sum_{i=1,2}m_i+k-2|T|\leq k+10-2|T|$. All these cases are impossible when  $\frac{k}{3}-3< a$. In case (d), we have
$7(a+3)=\sum_{i=1}^6(2m_i+1)+|R|-|T|=2\sum_{i=1}^6m_i+k-2|T|\leq k+12-2|T|$. The condition $\frac{k}{3}-3< a$ implies that $k<9$, while $a>0$ implies $k>9$, so we have a contradiction. 
\end{comment}

From the bounds on $d$ above, we may assume in the remaining cases that $k=9$. 

\underline{Case $d=4$ ($k=9$)}. 
Up to relabeling, one of the following holds:
\bi
\item[(a) ] $m_1\in\{4,-5\}$ and $m_i\in\{0,-1\}$ for all $i>1$,
\item[(b) ] $m_1\in\{3,-4\}$, $m_2\in\{2,-3\}$, 
$m_3\in\{1,-2\}$ and $m_i\in\{0,-1\}$ for all $i>3$,
\item[(c) ] $m_1\in\{3,-4\}$, $m_2,\ldots, m_5\in\{1,-2\}$ and $m_i\in\{0,-1\}$ for all $i>5$,
\item[(d) ] $m_1, m_2, m_3\in\{2,-3\}$, 
$m_4\in\{1,-2\}$ and $m_i\in\{0,-1\}$ for all $i>4$,
\item[(e) ] $m_1, m_2\in\{2,-3\}$, 
$m_3,\ldots, m_6\in\{1,-2\}$ and $m_i\in\{0,-1\}$ for all $i>6$,
\item[(f) ] $m_1\in\{2,-3\}$, 
$m_2,\ldots, m_8 \in\{1,-2\}$ and 
$m_9\in\{0,-1\}$.
%$m_i\in\{0,-1\}$ for all $i>8$,
%\item[(f) ] $m_1,\ldots, m_{10} \in\{1,-2\}$ ($k=10$). 
\ei
Using similar arguments, the condition $0<a$ leads to a contradiction in all cases.

\underline{Case $d=5$ ($k=9$)}. 
Up to relabeling, one of the following holds:
\bi
\item[(a) ] $m_1\in\{5,-6\}$ and $m_i\in\{0,-1\}$ for all $i>1$,
\item[(b) ] $m_1\in\{4,-5\}$, $m_2\in\{2,-3\}$, 
$m_3, m_4\in\{1,-2\}$ and $m_i\in\{0,-1\}$ for all $i>4$,
\item[(c) ] $m_1, m_2\in\{3,-4\}$, 
$m_3\in\{2,-3\}$ and $m_i\in\{0,-1\}$ for all other $i>3$,
\item[(d) ] $m_1, m_2\in\{3,-4\}$, 
$m_3, m_4, m_5\in\{1,-2\}$ and $m_i\in\{0,-1\}$ for all other $i>5$,
\item[(e) ] $m_1\in\{3,-4\}$, 
$m_2, m_3, m_4\in\{2,-3\}$ and $m_i\in\{0,-1\}$ for all $i>4$,
\item[(f) ] $m_1\in\{3,-4\}$, 
$m_2, m_3\in\{2,-3\}$, $m_4, m_5, m_6\in\{1,-2\}$ and $m_i\in\{0,-1\}$ for all $i>6$,
\item[(g) ] $m_1\in\{3,-4\}$, 
$m_2\in\{2,-3\}$, $m_3,\ldots, m_8\in\{1,-2\}$ and $m_9\in\{0,-1\}$,
\item[(h) ] $m_1, \ldots, m_5\in\{2,-3\}$ and $m_i\in\{0,-1\}$ for all $i>5$,
\item[(i) ] $m_1, \ldots, m_4\in\{2,-3\}$, $m_5, m_6, m_7\in\{1,-2\}$ and $m_8, m_9\in\{0,-1\}$, 
\item[(j) ] $m_1, m_2, m_3\in\{2,-3\}$ and $m_i\in\{1,-2\}$ for all $i>3$. 
\ei
Similar arguments lead to a contradiction in all cases. 
\ep

We are ready to prove Theorem~\ref{thm:moduli}. In fact, we prove the following more precise version.

\begin{thm}\label{X as moduli}
Let $S$ be the blowup of $\PP^2$ at $n+4\geq 6$ Cremona-general points $Q_1, \dots, Q_{n+4}$, and let $P_1, \dots, P_{n+4}\in \PP^n$ be the Gale dual points.
Consider the ample line bundle $L_a=(a+3)h-\sum_{i=1}^{n+4} e_i$, with $a>\max\{\frac{n+4}{3}-3,0\}$, and denote by 
$M_{L_a}$ the moduli space of $L_a$-semistable torsion free sheaves with (\ref{data}) (see Notation~\ref{not:M_L}).
Then 
\bi
\item $M_{L_a}=\emptyset$ for $a>n+1$.
\item $M_{L_a}\cong \PP^{n}$ for $n-1< a \leq  n+1$. 
\item $M_{L_a}\cong \Bl_{P_1,\ldots, P_{n+4}}\PP^{n}$ for 
$\max\big\{n-3,0\big\}< a\leq n-1$. 
\item If $n>3$ then, for $\max\big\{n-5, \frac{n-3}{3}\big\}< a\leq n-3$, $M_{L_a}$ is obtained from $\Bl_{P_1,\ldots, P_{n+4}}\PP^{n}$ by flipping the strict transforms of all the lines through two of the points $P_i$'s. 
\ei
\end{thm}

\bp 
By Lemma \ref{sstable vb}, we have $M_{L_a}=\emptyset$ if $a>n+1$.
Lemma \ref{line segment} determines for which values of $a$ the polarization $L_a$ lies on stability walls $(2D+K)^\perp$ for $\max\big\{n-3,0\big\}< a \leq n+1$ 
and also describes the numerical rational classes $D$ defining these walls:
\begin{itemize}
    \item $a=n+1$ and $D=h$;
    \item $a=n-1$ and $D=h-e_i$ for $i\in\{1,\ldots,n+4\}$;
    \item $a=n-3$ and $D=h-e_i-e_j$ for distinct $i, j\in\{1,\ldots,n+4\}$ if $n\geq 3$. 
\end{itemize}
Moreover, if $n>3$ then no wall is crossed for $\max\big\{n-5, \frac{n-3}{3}\big\}< a < n-3$.

The description of the moduli spaces follows from Theorem~\ref{prop:special_loci} and Proposition~\ref{reduction map}. 
The facts that the centers of the blowups are precisely the Gale dual points of $\{Q_1, \dots, Q_{n+4}\}\subset \PP^2$, and that the flipping loci are the strict transforms of the lines through two of these points, follow from Proposition~\ref{prop:Mukai_Gale}.

\ep

\begin{rmk}\label{walls_of_C0}
When $n\geq 3$, one can check that the stability chamber $\cC_0$ containing the polarization 
$$
L_{n-2}=(n+1)h-\sum e_i
$$ 
is given by 
$$
    \cC_0 \ = \  \bigcap_{i=1}^{n+4}\Big(2(h-e_i)+K>0\Big)\cap\bigcap_{i\neq j}\Big(2(h-e_i-e_j)+K <0\Big).
$$
This can be shown directly using arguments similar to those used above. 

However, the above description of $\cC_0$ \color{black}  also follows as a corollary of the results proved in the next sections. Namely, in Section~\ref{section:determinant} we construct a linear isomorphism $\rho:\N^1(S)\to \N^1(\Bl_k\PP^n)$ which maps $\ov{\cC_0}$ onto $\Nef(\Bl_{n+4}\PP^n)$ (see Theorem~\ref{main}).
The walls of $\Nef(\Bl_{n+4}\PP^n)$ correspond to extremal rays of the Mori cone $\NE(\Bl_{n+4}\PP^n)$ of  $\Bl_{n+4}\PP^n$. These are well known: they are generated by the classes of lines on the exceptional divisors over the blown-up points and by the classes of the strict transforms of the lines through two of the points (see, for instance, \cite[Proposition 1.4]{AM16}). In $\N^1(S)$, these correspond via $\rho^{-1}$ precisely to the walls $\big(2(h-e_i)+K\big)^\perp$ and $\big(2(h-e_i-e_j)+K\big)^\perp$, respectively.

When $n=2$, there are no $K$-negative walls intersecting the interior of $\Pi$, i.e., $\Pi$ is the closure of the chamber $\cC_0$. One can check that in this case we have
$$
    \Pi \ = \  \bigcap_{i=1}^{6}\Big(2(h-e_i)+K\geq0\Big)\cap\bigcap_{i\neq j}\Big(2(h-e_i-e_j)+K \leq0\Big)\cap
    \bigcap_{i=1}^6\Big(2e_i+K\leq0\Big).
$$

\end{rmk}

\subsection{Two special subcones $\Pi\subset \cE\subset \Nef(S)^{K\leq0}$}\label{subsection:special_cones}

In this subsection we describe two special subcones that contain the
special chamber $\cC_0$ studied in Subsection~\ref{subsection:C_0}:
\begin{itemize}
    \item Let $\cE\subset \Nef(S)^{K\leq0}$ be the closure of the union of the $K$-negative chambers $\cC\subset \Amp(S)$ for which $M_{\cC}\neq \emptyset$.
    \item Let $\Pi\subset \Nef(S)^{K\leq0}$ be the closure of the union of the $K$-negative chambers $\cC\subset \Amp(S)$ for which $M_{\cC}$ is a small modification of $M_{\cC_0}\cong \Bl_{n+4}\PP^n$.
\end{itemize}
From their definition, it is clear that $\cC_0\subset \Pi\subset \cE\subset \Nef(S)^{K\leq0}$. Our goal in this subsection is to describe these cones explicitly, in terms of their defining inequalities and, in the case of $\cE$, to determine all of its extremal rays in the half-space $\N^1(S)^{K<0}$. 
It will follow from our description that both cones are invariant by the action of the Weyl group $\cW$.

Let $\cC\subset \Amp(S)$ be a $K$-negative stability chamber, with corresponding moduli space $M_{\cC}$.
By Theorem~\ref{prop:special_loci}, $M_{\cC}\neq \emptyset$ if and only if $\cC\subset \big(2B+K\geq0\big)$ for every numerical cubic $B$ defining a $K$-negative wall. By Proposition~\ref{partial classif}, these classes $B$ are precisely the ones in $\cW\cdot h$.
Recall from Lemma~\ref{ample K-negative} that 
\[
\Nef(S)^{K\leq0}=\big(K\leq0\big)\cap\bigcap_{e\in \cW\cdot e_1}\big(e\geq0\big).
\]
Therefore, the cone $\cE$ can be described as follows. 
\begin{equation}\label{coneE}
    \cE \ = \  \big(K\leq0\big)\cap\bigcap_{e\in \mathcal \cW\cdot e_1}\big(e\geq0\big)\cap\bigcap_{B\in \cW\cdot h}\big(2B+K\geq0\big) \ \subset \ \Nef(S)^{K\leq0}.
\end{equation}

In order to describe $\Pi$, let $(2D+K)^{\perp}$ be a $K$-negative wall separating two neighbouring chambers $\cC^-\subset (2D+K)^{<0}$ and $\cC^+\subset (2D+K)^{>0}$. 
By Theorem~\ref{prop:special_loci}, $M_{\cC^-}$ and $M_{\cC^+}$ are birational to each other if and only if $-n+1 \leq D^2 \leq 0$, and
$M_{\cC^-}$ and $M_{\cC^+}$ are small modifications of each other if and only if $-n+2 \leq D^2 \leq -1$.
Moreover, $M_{\cC^+}$ is isomorphic to the blowup of $M_{\cC^-}$ at one point if and only if $D^2=0$. 
By Proposition~\ref{partial classif}, this is equivalent to $D\in \cW\cdot (h-e_i)$.
Similarly, $M_{\cC^-}$ is isomorphic to the blowup of $M_{\cC^+}$ at one point if and only if $(-K-D)\in \cW\cdot (h-e_i)$.
Using again the above description of $\Nef(S)^{K\leq0}$, we conclude that the cone $\Pi$ can be described as follows. 
\begin{equation}\label{conePi}
\Pi \ = \  \big(K\leq0\big)\cap\bigcap_{e\in \mathcal \cW\cdot e_1}\big(e\geq0\big)\cap\bigcap_{C\in \cW\cdot (h-e_1)}\big(2C+K\geq0\big) \ \subset \ \Nef(S)^{K\leq0}.
\end{equation}

Notice that any class $C \in \cW\cdot (h-e_1)$ spans an extremal ray of the cone $\Nef(S)$, hence, also of $\cE$. In particular, the cone $\cE$ contains the (closure of the) cone spanned by $\cW\cdot (h-e_1)$. 
Our next goal is to prove that $\cE$ coincides with the closure of this cone.
For this purpose, we introduce the following cones:
\[
\cE'\coloneqq \sum_{C\in \cW\cdot (h-e_1)}\RR_{\geq 0}C \quad \subset  \quad  \ov{\cE'} \quad \subset  \quad \Nef(S)^{K\leq0} \ .
\]

\begin{lemma} \label{Lemma_E}
Let the notation be as above. 
Then  the following holds:
\begin{itemize}
\item[(1) ] $\cE \, =\,  \ov{\cE'}$,
\item[(2) ] $\cE'\setminus\{0\}=\cE^{K<0}$,
\end{itemize}
\end{lemma}

\bp
As we noted above, we have an inclusion of closed cones $\ov{\cE'}\subset \cE$, and, dually, $\cE^\vee\subset \ov{\cE'}^\vee$.

Consider the following cone in $\N^1(S)$:
\[
\cD \coloneqq \big\langle -K \ , \ \cW\cdot e_1 \ , \ \{2B+K\}_{B\in \cW\cdot h} \,\big\rangle \ \subset \ \N^1(S).
\]
It follows from (\ref{coneE}) and the definition of $\cE'$ that
\[
\cD\ \subset \ \cE^\vee \ \subset \ \ov{\cE'}^\vee =  \bigcap_{C\in \mathcal \cW\cdot (h-e_1)}\big(C\geq0\big)  \ \subset \ \N^1(S).
\]
We will show that $\ov{\cE'}^\vee\subseteq \cD$, and then conclude that $\cD= \cE^\vee= \ov{\cE'}^\vee$.
We will show this by induction on $k$, with base case $k=2$. Note that $\cW$ was only defined when $k>3$. 
For the purpose of this proof, when $k\leq 3$, we interpret $\cW\cdot e_1$, $\cW\cdot (h-e_1)$ and $\cW\cdot h$ as the set of numerical lines, conics and cubics on $S=\Bl_k\PP^2$, respectively. 

\medskip

\noindent {\bf Claim.} Let $S=\Bl_k\PP^2$, with $k\geq2$, and $D\in\Pic(S)$. If $D\cdot C\geq0$ for every $C\in \cW\cdot (h-e_1)$, then 
\[
D\ \in \ \ZZ_{\geq0}(-K)+\sum_{e\in \cW\cdot e_1} \ZZ_{\geq0} e +\sum_{B\in \cW\cdot h} \ZZ_{\geq0} (2B+K).
\]

\begin{proof}[Proof of Claim.] We proceed by induction on $k$. 

When $k=2$, the only numerical lines are $e_1$, $e_2$, and $h-e_1-e_2$, the only numerical conics are $h-e_1$ and $h-e_2$, and the only numerical cubic is $h$. Conditions $D\cdot(h-e_i)\geq0$ ($i=1,2$) imply that, up to exchanging $i=1,2$, one of the following holds. 
\bi
\item[(i) ] $D=-dh+m_1e_1+m_2e_2$ with $0\leq d\leq m_i$ ($i=1,2$), 
\item[(ii) ] $D=dh-m_1e_1+m_2e_2$ with $0\leq m_1\leq d$, $m_2\geq0$,
\item[(iii) ] $D=dh-m_1e_1-m_2e_2$ with $0\leq m_i\leq d$ ($i=1,2$)
\item[(iv) ] $D=dh+m_1e_1+m_2e_2$ with $d, m_1, m_2\geq 0$.
\ei

In case (i), we write $D=d(2h+K)+(m_1-d)e_1+(m_2-d)e_2$. In case (ii), we write $D=d(h-e_1-e_2)+(d-m_1)e_1+(d+m_2)e_2$. In case (iii), we have
$D=d(h-e_1-e_2)+(d-m_1)e_1+(d-m_2)e_2$, while in case (iv), we have
$D=d(h-e_1-e_2)+(d+m_1)e_1+(d+m_2)e_2$. 
This proves the statement when $k=2$. 

Assume $k\geq3$. 
Suppose that $D\cdot e < 0$ for some $e\in \cW\cdot e_1$.
Without loss of generality, we may assume that $e=e_k$. Denote by $\pi\colon S\to S'=\Bl_{k-1}\PP^2$ the blowdown of $e_k$.
Then $D=\pi^*(D')+l e_k$ for some $l\in \ZZ_{>0}$ and $D' \in\Pic(S')$.
If a curve class $C'\in \Pic(S')$ belongs to $\cW\cdot e_1$, $\cW\cdot (h-e_1)$ or $\cW\cdot h$, then its pullback $\pi^* C'\in \Pic(S)$ belongs to $\cW\cdot e_1$, $\cW\cdot (h-e_1)$ or $\cW\cdot h$, respectively. 
So the statement follows by induction on $k$. 
So we may assume that $D\cdot e\geq0$ for every $e\in \cW\cdot e_1$.
In this case, it follows from Lemma~\ref{E positive surface case}(1) that $D\in \ZZ_{\geq0}(-K)+\sum_{e\in \cW\cdot e_1} \ZZ_{\geq0} e$, and this proves the claim.   
\end{proof}

It follows from the claim that $\ov{\cE'}^\vee\subseteq \cD$, and so $\cD= \cE^\vee= \ov{\cE'}^\vee$.
Since both $\ov{\cE'}$ and $\cE$ are closed convex cones, by passing to the duals, we obtain that $\cE=\ov{\cE'}$. This proves (1). 

Now we prove (2) by induction on $k$. When $k\leq8$, $S$ is a del Pezzo surface, and $\cE'=\ov{\cE'}=\cE$ is a closed polyhedral cone. So we have $\cE^{K<0}=\cE\setminus\{0\}$.  

Assume now that $k\geq9$ and notice that $\cE'\setminus\{0\}\subseteq \cE^{K<0}$. 
The supporting hyperplanes of $\cE$ are $K^\perp$, $e^\perp$ for $e\in \cW\cdot e_1$, and $(2B+K)^\perp$ for $B\in \cW\cdot h$. 
If $D\in K^\perp\setminus\{0\}$, then $D\not\in \cE'$, since any $C\in \cW\cdot (h-e_1)$ satisfies $C\cdot (-K)=2$. We will now prove that any divisor class $D\in \cE\setminus K^\perp$ lies in $\cE'$. If $D\in e^\perp$ for some $e\in \cW\cdot e_1$, we may assume that $e=e_k$ and so $D=\pi^*D'$, where $\pi\colon S\to S'=\Bl_{k-1}\PP^2$ is the blowdown of $e_k$ and $D'\in \Pic(S')$ is such that $D'\in \cE\setminus K^\perp$. We are done by induction on $k$. 
If $D\in (2B+K)^\perp$ for some $B\in \cW\cdot h$, we may assume that $B=h$. We claim that a nef divisor class $D$ such that $D\cdot (2h+K)=0$ must lie in the cone $\cE'$. Indeed, if $D=dh-\sum_{i=1}^k m_ie_i\in (2h+K)^\perp$ then $d=\sum_{i=1}^k m_i$ and, since $D$ is nef, we have $d\geq0$, $m_i\geq0$ for all $i$. It follows that $D=\sum_{i=1}^k m_i(h-e_i)$. 
Now suppose that $D\in\cE\setminus\Big(K^\perp\cup \bigcup_{e\in \cW\cdot e_1}e^\perp \cup \bigcup_{B\in \cW\cdot h}(2B+K)^\perp \Big)$. Fix an ample divisor $A$ on $S$ and set $a:=D\cdot A>0$. 
Consider the intersection $\cE_a$ of the cone $\cE$ with the hyperplane $H_a:=\{\ga\ \in \N^1(S)\ | \ga\cdot A=a\}$. Then $\cE_a\subset H_a\cong\RR^k$ is a compact convex set with boundary 
\[
\partial\cE_a=(\cE_a\cap K^\perp)\cup \bigcup_{e\in \cW\cdot e_1}(\cE_a\cap e^\perp) \cup \bigcup_{B\in \cW\cdot h}\big(\cE_a\cap(2B+K)^\perp \big).
\]
A line in $H_a\cong\RR^k$ that passes through $D$ and is paralell to $K^\perp$ will intersect $\cE_a$ in a bounded segment with endpoints $v_1$ and $v_2$ lying on $\partial\cE_a\setminus K^\perp$. We have proved above that $v_1, v_2\in\cE'$. Since $D=av_1+bv_2$ for some $a,b\in\RR_{>0}$, it follows that $D\in\cE'$. This proves (2).

\ep

\begin{rmk}\label{generators of cone E for 9 points}
Let us consider the special case when $k=9$. Then $-K$ is nef and, by Lemma \ref{E positive surface case}(4), 
\[
\Eff(S)\, =\, \RR_{\geq 0}(-K) +\sum_{e\in \cW\cdot e_1}\RR_{\geq 0}e \quad \text{ and  } \quad \Eff(S)\cap K^\perp=\RR_{\geq0}(-K)  \, .
\]
From the definition of the cone $\cE$, we see that $-K\in \partial\cE$. 
It follows from Lemma~\ref{Lemma_E}(2) that 
\[
\cE \, =\, \RR_{\geq 0}(-K) +\sum_{C\in \cW\cdot (h-e_1)}\RR_{\geq 0}C.
\]
\end{rmk}

\section{The determinant map}\label{section:determinant}
In this section, we construct a canonical linear map 
$$
\rho: \Pic(\Bl_{n+4}\PP^2)\ra \Pic(\Bl_{n+4}\PP^{n}),
$$ 
which we call \emph{the determinant map}, that will allow us to translate the stability chamber decomposition of $\Amp(\Bl_{n+4}\PP^2)^{K\leq 0}$ to the Mori chamber decomposition of $\Eff(\Bl_{n+4}\PP^{n})^{K\leq 0}$.
Our exposition follows \cite[Section 8.1]{HL} and \cite{CCF}.

As in the previous sections, we fix $Q=\big(Q_1, \dots, Q_{n+4}\big)$ and $P=\big(P_1, \dots, P_{n+4}\big)$ Gale dual configurations of Cremona-general points in $\PP^2$ and $\PP^n$, respectively, with $n+4\geq 6$.
We denote by $S=\Bl_{n+4}\PP^2$ the blowup of $\PP^2$ at the points $Q_1, \dots, Q_{n+4}$, and by $X=\Bl_{n+4}\PP^{n}$ the blowup of $\PP^n$ at the Gale dual points $P_1, \dots, P_{n+4}$. 
Recall from Theorem~\ref{X as moduli} that $X$ is isomorphic to the moduli space $M_{L_a}$ introduced in Notation~\ref{not:M_L} for $n-3<a<n-1$. 

\begin{notn}\label{det notations}
Let $K(S)$ denote the Grothendieck group of coherent sheaves on $S=\Bl_{n+4}\PP^2$. As $S$ is a rational surface, there is a group homomorphism 
$$v: (K(S),+)\ra \ZZ\oplus\Pic(S)\oplus\ZZ, \quad v(a)=(\ch_0(a), \text{c}_1(a), \chi(a)).$$
Note that if $x$ is a point in $S$, then $v([\cO_x])=(0,\cO_S, 1)$. 

Let $f$ denote the class in $K(S)$ of a sheaf $F$ with (\ref{data}). 
Denote by $f^\perp$ the subgroup of $K(S)$ of those elements $a\in K(S)$ such that
$$\chi(f\otimes a)=2ch_2(a)-2c_1(a)\cdot K+\ch_0(a)K^2=0.$$

Consider the map $u:\Pic(S)\ra f^\perp$ given by
$$u(L)=-2[\cO_S]+2[L^{-1}]-(L^2+2L\cdot K)[\cO_x]$$
It is straightforward to check that the image of $u$ is contained in $f^\perp$, so $u$ is well defined. 
Note also that $u$ is a group homomorphism, and that 
$$v(u(L))=(0,-2L,-L\cdot K).$$

For any flat family $\cE$ of coherent sheaves on $S$ parametrized by a scheme $T$, one can define a group homomorphism 
$$\lambda_{\cE}: K(S)\ra \Pic(T), \quad \lambda_{\cE}(F)=\det\big({p_T}_!\big(p_S^*(F)\otimes\cE\big)\big),$$
where $p_S: S\times T\ra S$ and $p_T: S\times T\ra T$ are the two projections \cite[Definition 8.1.1]{HL}. If on $S\times T$ we consider two families $\cE$, $\cE'$ of torsion free sheaves on $S$ 
with (\ref{data}) such that $\cE'=\cE\otimes p_T^*\cN$ for some locally free $\cO_T$-sheaf $\cN$, then the restrictions of $\lambda_{\cE}$ and $\lambda_{\cE'}$ to $f^\perp$ are equal \cite[Lemma 8.1.2]{HL}. 

Let $L$ be a polarization on $S$. Recall that quasi-universal families on $M^s_L\times S$ always exist (Remark \ref{universal}). A quasi-universal family is uniquely defined up to tensoring with a vector bundle on $M^s_L$ \cite[Definition 4.6.1]{HL}, hence, there exists a well-defined group homomorphism
$$\lambda^L: f^\perp\ra\Pic(M^s_L),$$
which is independent of the choice of the quasi-universal family (see also \cite[Theorem 8.1.5]{HL}).  
\end{notn}

The following definition was first introduced for del Pezzo surfaces in \cite{CCF}. 
\begin{defn}\label{det}
Let $L$ be a polarization on $S$. The  \emph{determinant map of $L$} is the group homomorphism
$$\rho_L:\Pic(S)\ra \Pic(M^s_L)$$
obtained as the composition of the homomorphisms $u:\Pic(S)\ra f^\perp$ and $\lambda^L:f^\perp\ra\Pic(M^s_L)$. 
The map $\rho_L$ is independent of $L$ when $L$ varies in a chamber $\cC$, hence, we will denote it by
$$\rho_\cC:\Pic(S)\ra \Pic(M_\cC).$$
We define the  \emph{determinant map}  
$$\rho_X: \Pic(S)\ra \Pic(X)$$
to be the determinant map $\rho_{\cC_0}$ for the chamber $\cC_0$ discussed in Subsection~\ref{subsection:C_0}, for which $M_{\cC_0}\cong X=\Bl_{n+4}\PP^{n}$. 
\end{defn}

\begin{lemma}\label{det 2}
Let $L$ be a polarization on $S$. The image $\rho_L(L)$ of $L$ by the determinant map $$\rho_L:\Pic(S)\ra \Pic(M^s_L)$$ has the following properties:
\bi
\item $\rho_L(L)$ is the restriction of a line bundle $\cL_1\in\Pic(M_L)$ to $M^s_L$, 
\item A multiple $\cL_1^N$ ($N\gg 0$) induces the morphism $M_L\ra M_L^{\mu ss}$. In particular, $\cL_1$ is the pull-back of an ample line bundle on $M_L^{\mu ss}$,
\item If $L$ belongs to some chamber $\cC$ then $\rho_\cC(L)$ is an ample line bundle on $M_\cC$. 
\ei

\end{lemma}

\bp
We may assume that $L$ is very ample since $\rho_{mL}=\rho_L$.
Let $\Xi\subset S$ denote a general element in the linear system $|L|$ and let $\xi=[\cO_\Xi]\in K(S)$. Note that $[\xi]=[\cO_S]-[L^{-1}]\in K(S)$. 
As in \cite[Definition 8.1.4]{HL}, let $K_{f,L}=f^\perp\cap\{1, \xi, \xi^2\}^{\perp\perp}$. By Theorem \cite[Theorem 8.1.5]{HL}, there is a group homomorphism $K_{f,L}\ra\Pic(M_L)$  whose composition with the restriction map $\Pic(M_L)\ra \Pic(M^s_L)$ coincides with the restriction of the map $\la^L: f^\perp\ra\Pic(M^s_L)$ to  $K_{f,L}\subset f^\perp$. 
It is straightforward to check that $u(L)\in K_{f,L}$. Note that $u(L)$ coincides with the element $u_1(f)=-2\xi+\chi(f\cdot \xi)[\cO_x]$ as defined in \cite[Definition 8.1.9]{HL} and $\rho_L(L)$ coincides with the line bundle $\cL_1$ from the same definition. It follows from \cite[Theorem 8.2.8]{HL} that $\cL_1$ is ample. 
\ep

When $\cC$ and $\cC'$ are $K$-negative stability chambers whose corresponding moduli spaces $M_{\cC}$ and $ M_{\cC'}$ are isomorphic in codimension $1$, there is a natural isomorphism $\Pic(M_{\cC})\cong \Pic(M_{\cC'})$. The following result shows that this isomorphism is compatible with their associated determinant maps.

\begin{lemma}\label{compatible}
Assume that $L$, $L'$ are two polarizations with $K\cdot L<0$ and $K\cdot L'<0$, and assume that there is a small modification $\psi:  M^s_{L'}\dra M^s_{L}$. Then there is a commutative diagram 
\begin{equation*}
    \xymatrix{&\Pic(S) \ar[dl]_{\rho_L} \ar[dr]^{\rho_{L'}}&\\
\ar[rr]_{\psi^*} \Pic(M^s_{L})& &\Pic(M^s_{L'}).}
\end{equation*}
\end{lemma}

\bp
Let $U\subseteq M^s_L$ and $U'\subseteq M^s_{L'}$ be open subsets such that $\codim(M^s_L\setminus U), \codim(M^s_{L'}\setminus U')\geq2$, and $\psi: U'\ra U$ is an isomorphism.
Let $\cE$ be a quasi-universal family on $U\times S$ \cite[Definition 4.6.1]{HL}. Then $\cE'=\psi^*\cE$ is a quasi-universal family on $U'$
and the maps 
$$u\circ\lambda_\cE: \Pic(S)\ra\Pic(U),\quad u\circ\lambda_{\cE'}: \Pic(S)\ra\Pic(U')$$ 
are related by composition with $\psi^*$. 
Since both $M^s_L$  and $M^s_{L'}$ are smooth (Lemma \ref{smoothness}), we have that 
$\Pic(M^s_L)\cong\Pic(U)$ and $\Pic(M^s_{L'})\cong\Pic(U')$, and the statement follows. 
\ep

In Subsection~\ref{subsection:variation_of_ML}, we described the variation of the moduli spaces $M_L$ as the polarization $L$ crosses a $K$-negative wall in $\N^1(S)$.
We review this description and then analyse their corresponding determinant maps.

Let $(2D+K)^{\perp}$ be a $K$-negative wall separating two neighboring chambers 
$\cC^-\subset (2D+K)^{<0}$ and $\cC^+\subset (2D+K)^{>0}$. 
Recall from Lemma~\ref{good_walls} that $-n\leq D^2 \leq 1$. 
Let $L_0\in (2D+K)^{\perp}$ be a polarization lying in the closures of the chambers $\cC^-$ and $\cC^+$, and not
lying in any other wall.
The following statements follow from Theorem~\ref{prop:special_loci} and Proposition~\ref{reduction map}:
\bi
\item If $D^2=-n$, then $M_{\cC_{-}}=\PP_D\cong \PP^{n}$, $M_{\cC_{+}}= \emptyset$  and there is a natural isomorphism
$M_{\cC_{-}}\cong M_{L_0}$.
\item If $D^2=1$, then $M_{\cC_{+}}=\PP_{-K-D}\cong \PP^{n}$, $M_{\cC_{-}}= \emptyset$  and there is a natural isomorphism
$M_{\cC_{+}}\cong M_{L_0}$.
\ei
Suppose now that $-n< D^2 < 1$, so both $M_{\cC_{-}}$ and $M_{\cC_{+}}$ are non-empty. 
Then there is a commutative diagram 
$$
    \xymatrix@R=.5cm@C=.5cm{
    \PP^{-D^2}\cong \PP_D \subset M_{\cC^-} \ \ \ \ar[rd]_{\phi_-} \ar@{-->}[rr]^{\phi_+\circ \phi_-} && 
    \ \ \ M_{\cC^+}\supset \PP_{-K-D}\cong \PP^{n-1+D^2}. \ar[ld]^{\phi_+}\\
    & M^{\mu ss}_{L_0}
    }
$$
\bi
\item If $D^2=-n+1$, then $\phi_{+}\circ \phi_{-}: M_{\cC_{-}}\ra M_{\cC_{+}}$ is the blowup of the point $\PP_{-K-D}\in M_{\cC_{+}}$ with exceptional divisor 
$\PP_D\cong\PP^{n-1}$.
\item If $D^2=0$, then $\phi_{-}\circ \phi_{+}: M_{\cC_{+}}\ra M_{\cC_{-}}$ is the blowup of the point $\PP_D\in M_{\cC_{-}}$ with exceptional divisor $\PP_{-K-D}\cong \PP^{n-1}$.
\item If $-n+2\leq D^2\leq -1$, then $\phi_{+}\circ \phi_{-}: M_{\cC_{-}}\dra M_{\cC_{+}}$ is an isomorphism in codimension $1$.
\ei 
We next describe the behavior of the determinant maps with respect to this diagram.

\begin{proposition}\label{det 3}
Let $(2D+K)^{\perp}$ be a $K$-negative wall separating two neighboring chambers 
$\cC^-\subset (2D+K)^{<0}$ and $\cC^+\subset (2D+K)^{>0}$. Consider the determinant maps 
$$\rho_{\cC_{-}}:\Pic(S)\ra\Pic(M_{\cC_{-}}),\quad \rho_{\cC_{+}}:\Pic(S)\ra\Pic(M_{\cC_{+}})$$
whenever $M_{\cC_{-}}$ or $M_{\cC_{+}}$ are non-empty. 
We denote by $l_D$ the class of a line in $\PP_D$  whenever 
$\dim(\PP_D)>0$, and by 
$l_{-K-D}$ the class of a line in $\PP_{-K-D}$ whenever $\dim(\PP_{-K-D})>0$. 
Then:
\bi
\item Whenever $\dim(\PP_D)=-D^2>0$, we have 
$$\rho_{\cC_{-}}((2D+K)^\perp)=l_D^\perp.$$
\item Whenever $\dim(\PP_{-K-D})=n-1+D^2>0$, we have 
$$\rho_{\cC_{+}}((2D+K)^\perp)=l_{-K-D}^\perp.$$
\ei

\end{proposition}

\bp
As before, let $L_0\in (2D+K)^{\perp}$ be a polarization lying in the closures of the chambers $\cC^-$ and $\cC^+$, and not
lying in any other wall. 

Assume that $D^2<\frac{1-n}{2}$.  
By Proposition \ref{reduction map}, there is  an isomorphism $M_{\cC_{-}}\cong M_{L_0}=M^s_{L_0}$ and 
the morphism $\phi_{-}: M_{\cC_{-}}\ra M_{L_0}^{\mu ss}$ factors through $M_{L_0}$. 
By Lemma \ref{compatible} we have $\rho_{\cC_{-}}=\rho_{L_0}$ and $\rho_{\cC_{-}}(L_0)=(\phi_{-})^*A$ for some ample line bundle $A$ on $M_{L_0}^{\mu ss}$. 
In this case $D^2\leq -1$, so $\dim(\PP_D)>0$ and $\PP_D\subset M_{\cC_{-}}$ is contracted to a point in $M_{L_0}^{\mu ss}$. 
It follows that $\rho_{\cC_{-}}(L_0)\cdot l_D=0$ and, hence, $\rho_{\cC_{-}}((2D+K)^\perp)=l_D^\perp$. 
This finishes the proof when $D^2$ is $-n+1$ or
$-n$. 
When $-n+2\leq D^2<\frac{1-n}{2}$, we have $\dim(\PP_{-K-D})=n-1+D^2>0$, and we need to prove that 
$\rho_{\cC_{+}}(L_0)\cdot l_{-K-D}=0$ in $M_{\cC_{+}}$. Since $\phi_{+}\circ \phi_{-}: M_{\cC_{-}}\dra M_{\cC_{+}}$ is an isomorphism in codimension $1$, it follows from Lemma \ref{compatible} that $\rho_{\cC_{+}}(L_0)=(\phi_{+})^*A$, and the statement follows. 
The case  $D^2>\frac{1-n}{2}$ is analogous, exchanging the roles of $\cC_{-}$ and $\cC_{+}$. 

Assume now that $D^2=\frac{1-n}{2}$. The map $\phi_{-}: M_{\cC_{-}}\ra M_{L_0}^{\mu ss}$ factors through a map $\psi_{-}: M_{\cC_{-}}\ra M_{L_0}$
inducing an isomorphism $M_{\cC_{-}}\setminus \PP_D\cong M^s_{L_0}$ (Proposition \ref{reduction map}). Since  $\codim(\PP_D)\geq 2$, 
Lemma \ref{compatible} implies that there is a commutative diagram
\begin{equation*}
    \xymatrix{&\Pic(S) \ar[dl]_{\rho_{L_0}} \ar[dr]^{\rho_{\cC_{-}}}&\\
\ar[rr]^{\cong}_{(\psi_{-})^*} \Pic(M^s_{L_0})& &\Pic(M_{\cC_{-}}). }
\end{equation*}
By Lemma \ref{det 2}, $\rho_{L_0}(L_0)\in \Pic(M^s_{L_0})$ is the restriction to $M^s_{L_0}$ of a line bundle $\cL$ on $M_{L_0}$. It follows that the restrictions of the line bundles
$\rho_{\cC_{-}}(L_0)$ and $(\psi_{-})^*\big(\cL_{|M^s_{L_0}}\big)$ to $M_{\cC_{-}}\setminus \PP_D$ are isomorphic. 
Since $M_{\cC_{-}}$ is smooth, the two line bundles are isomorphic, and hence $\rho_{\cC_{-}}((2D+K)^\perp)=l_D^\perp$. 
Similarly,   $\rho_{\cC_{+}}((2D+K)^\perp)=l_{-K-D}^\perp$. 
\ep

\begin{rmk}\label{on wall-semiample}
The proof of Proposition \ref{det 3} shows the following. 
Let $L$ be an ample line bundle, either contained in a chamber $\cC$, or lying on a (unique) $K$-negative wall $(2D+K)^\perp$.
In the latter case, let $\cC$ be any chamber bounded by the wall $(2D+K)^\perp$ such that $M_{\cC}\neq \emptyset$.
Then the line bundle $\rho_\cC(L)$ is semiample on $M_\cC$. When $L\in \Amp(S)^{K<0}$ lies on several $K$-negative walls, an analogue of Theorem \ref{prop:special_loci} holds  and, for any chamber $\cC$ bounded by a $K$-negative wall containing $L$, and such that $M_{\cC}\neq \emptyset$, the line bundle $\rho_\cC(L)$ is again semiample on $M_\cC$ (we leave the details to the reader).
\end{rmk}

In what follows, whenever $D$ is a divisor in $X=\Bl_{n+4}\PP^n$, $D^\perp\subset\Pic(X)$ denotes the orthogonal complement with respect to the Coble pairing, not to be confused with $C^\perp$, for $C$ a curve on $X$, which is the orthogonal complement of $C$ with respect to the intersection pairing.
We consider the usual bases $\Pic(X)=\ZZ\{H,E_1,\ldots,E_k\}$  and $\Pic(S)=\ZZ\{h,e_1,\ldots,e_k\}$, and identify the Weyl groups $\cW=\cW_X=\cW_S$ as in Notation~\ref{notation:isometry}.

\begin{thm}\label{rho map}
Let $S=\Bl_{n+4}\PP^2$ and $X=\Bl_{n+4}\PP^n$ be as above and assume that $n\geq2$.
Consider the map $\rho_X: \Pic(S)\ra\Pic(X)$ of Definition \ref{det}.  
There exists $m\in\ZZ_{>0}$ such that $\rho_X=m\cdot \rho$, where
$$
\rho:\Pic(S)\ra \Pic(X)
$$
is the map defined by setting
$$
\rho(h)=-H+\sum_{j=1}^{n+4} E_j \quad \text{ and } \quad  \rho(e_i)=-H+\sum_{j=1}^{n+4} E_j-2E_i \ \ \forall i.
$$
The map $\rho$ is injective but not surjective, and it induces an isomorphism  $\N^1(S)\ra \N^1(X)$.
\end{thm}

\bp
Letting $\N_1(X)=\Pic(X)^\vee$, $\N_1(S)=\Pic(S)^\vee$, we determine the dual map
$\rho_X^\vee: \N_1(X)\ra \N_1(S)$. 
We denote by $l\in \N_1(X)$ the class of the strict transform of a general line in $\PP^{n}$, and by $f_i\in \N_1(X)$ the class of a line in $E_i\cong \PP^{n-1}$. 
Then $l, f_1,\ldots, f_n$ form a $\ZZ$-basis for $\N_1(X)$ with $l\cdot H=1$, $l\cdot E_i=0$, $f_i\cdot E_j=-\de_{ij}$. 
Consider the chamber $\cC_0$ that contains $L_a=(a+3)h-\sum_{i=1}^{n+4} e_i$, for $(n-3)<a<(n-1)$. By Theorem~\ref{X as moduli}, $M_{\cC_0}\cong X$. 

For each $i,j\in\{1,\ldots , n+4\}$, $i\neq j$, set
$$D_i:=h-e_i \ \ \text{ and } \ \ D_{ij}:=h-e_i-e_j.$$ 
The wall  $(2D_i+K)^\perp$ bounds the chamber $\cC_0\subset (2D_i+K)^{>0}$, 
and the set $\PP_{-K-D_i}\subset M_{\cC_0}\cong X$ is the exceptional divisor $E_i$ above the point $P_i$. 
By Proposition \ref{det 3}, we have $\rho_X((2D_i+K)^\perp)=f_i^\perp$. 
Similarly, the wall  $(2D_{ij}+K)^\perp$ bounds the chamber $\cC_0\subset (2D_{ij}+K)^{<0}$, and  the set $\PP_{D_{ij}}\subset M_{\cC_0}\cong X$ is the proper transform of the line through $P_i$ and $P_j$ by Proposition \ref{prop:Mukai_Gale}(3). By Proposition \ref{det 3}, we have $\rho_X((2D_{ij}+K)^\perp)=(l-f_i-f_j)^\perp$. 
Hence, there are $\lambda_i, \mu_{ij}\in\ZZ\setminus\{0\}$ such that for all $i, j$ as above, we have: 
$$\rho_X^\vee(f_i)=\lambda_i(2D_i+K),\quad \rho_X^\vee(l-f_i-f_j)=\mu_{ij}(2D_{ij}+K).$$

It is straightforward to see that there exists an integer $m$ such that 
$$\rho_X^\vee(l)=m\cdot \big(-h+\sum_{j=1}^{n+4} e_j\big),\quad \rho_X^\vee(f_i)=m\cdot \big(-h+\sum_{j=1}^{n+4} e_j-2e_i\big).$$
It follows that $\rho_X=m\cdot\rho$.
One checks easily that $\rho$ is injective.
Since $\rho_X(L_a)$ must be ample if $n-3<a<n-1$ and 
$$\rho(L_a)=(n+1-a)H-(n-1-a)\sum_{j=1}^{n+4} E_j,$$ 
it follows that $m>0$. 
\ep

\begin{rmk}\label{rho}\label{X chamber and nearby}

One can check using Proposition \ref{W and rho} that the map $\rho:\Pic(S)\ra \Pic(X)$ introduced in Theorem~\ref{rho map}
satisfies the following properties:
\bi
\item $\rho(K_S)=K_X$ and $\rho$ is $\cW$-equivariant, i.e., $\rho(w\cdot L)=w\cdot\rho(L)$, for every $w\in\cW$ and $L\in\Pic(S)$.
\item With respect to the Coble pairing $(-,-)$ on $\Pic(X)$, we have 
$$(\rho(L),K_X)=(n-1)(L\cdot K_S), \quad  \ \ \forall L\in\Pic(S),$$
$$\rho(K_S^\perp)=2\cdot K_X^{\perp},\quad (\rho(L),\rho(L'))=4(L\cdot L') \quad \ \ \forall L, L'\in K_S^\perp.$$
\ei
In particular, the restriction of the map $\rho:\Pic(S)\ra\Pic(X)$ to $K_S^\perp$ is $2\phi_0$, where $\phi_0:K_S^\perp\ra K_X^\perp$
is the isometry from Notation  \ref{notation:isometry}. Note that $\frac{1}{2}\rho=\phi_0$ is an isometry on $K_S^\perp$ but $\frac{1}{2}\rho$ is not an isometry on $\Pic(S)$ in general: If $L=aK_S+L_1$, $L'=bK_S+L_2$ with $L_1, L_2\in K_S^\perp$, then 
$$(\rho(L),\rho(L'))=4(L\cdot L')+ab((K_X,K_X) -4K^2_S)=4(L\cdot L')-ab(n-5)^2.$$

We record here the following relations for further use:
$$\rho\big(e_i^\perp\big)=(l-f_i)^\perp,$$
$$\rho\big((2h+K)^\perp\big)=l^\perp,$$
$$\rho(h-e_i)=2E_i.$$ 
\end{rmk}

\section{The pseudoeffective cone of $\Bl_{n+4}\PP^{n}$ and the Mori chamber decomposition}\label{section:proof_main_thm}

In this section, we prove Theorem~\ref{thm:main}, which describes $\Eff(\Bl_{n+4}\PP^{n})^{K\leq0}$ as well as the Mori chamber decomposition of $\Mov(\Bl_{n+4}\PP^{n})^{K\leq0}$. Throughout this section, we denote by $X=\Bl_{n+4}\PP^{n}$ the blowup of $\PP^n$ at $n+4$ Cremona-general points, and by $S=\Bl_{n+4}\PP^2$ the blowup of $\PP^2$ at the $n+4$ Gale dual points.

Recall that in Section~\ref{section:blowups2}, we identified a chamber $\cC_0\subset \Nef(S)^{K<0}$ whose corresponding moduli space is $M_{\cC_0}\cong X$. We also described special cones 
\begin{equation}\label{cones_of_S}
    \cC_0\subseteq \Pi \subseteq \cE\subseteq \N^1(S)^{K\leq0}.
\end{equation}
The cone $\cE\subset \Nef(S)^{K\leq0}$ was defined as the closure of the union of the $K$-negative chambers $\cC\subset \Nef(S)^{K<0}$ for which $M_{\cC}\neq \emptyset$, and $\Pi\subset \cE$ as the closure of the union of those $\cC\subset \Nef(S)^{K<0}$ for which $M_{\cC}$ is a small modification of $X$. 
For any chamber $\cC\subset \Pi$, there is a natural identification $\N^1(M_{\cC})=\N^1(X)$, under which 
$$
\Mov(M_\cC)=\Mov(X) \quad \text{ and } \quad \Eff(M_\cC)=\Eff(X).
$$
We will make use of this identification without further mention.

In order to prove Theorem~\ref{thm:main}, we will make use of the linear map 
$$
\rho\colon \N^1(S)\to \N^1(X)
$$
obtained by tensoring with $\RR$ the map $\rho: \Pic(S)\ra \Pic(X)$ introduced in Theorem~\ref{rho map} as a primitive submultiple of the determinant map. We will show that $\rho$ maps the special cones (\ref{cones_of_S}) bijectively onto the cones 
$$
\Nef(X)\subseteq \Mov(X)^{K\leq0}\subseteq  \Eff(X)^{K\leq0}\subseteq \N^1(X)^{K\leq0},
$$
respectively. Moreover, we will show that $\rho$ maps the stability chamber decomposition of $\Pi$, which was described in Sections~\ref{section:moduli} and \ref{section:blowups2}, to the Mori chamber decomposition of $\Mov(\Bl_{n+4}\PP^{n})^{K\leq0}$.

\begin{theorem}\label{main} For $n\geq2$,
let $X=\Bl_{n+4}\PP^{n}$ be the blowup of $\PP^n$ at $n+4$ Cremona-general points, $S=\Bl_{n+4}\PP^2$ the blowup of $\PP^2$ at the $n+4$ Gale dual points, and $\rho: \N^1(S)\to \N^1(X)$ 
the linear map introduced in Theorem~\ref{rho map}. Then the following holds. 
\begin{enumerate}
    \item For any stability chamber $\cC\subseteq \Pi$, 
    $$\rho(\ov{\cC})=\Nef(M_{\cC})\subseteq \N^1(M_{\cC})=\N^1(X).$$
    \item $\rho(\Pi)=\Mov(X)^{K_X\leq0}$.
    \item $\rho(\cE)= \overline{\displaystyle\sum_{E\in \cW\cdot E_1}\RR_{\geq 0}E} = \Eff(X)^{K_X\leq0}$.
\end{enumerate}
\end{theorem}

\bp
By Remark~\ref{X chamber and nearby}, we have $\rho(K_S^\perp)=K_X^\perp$ and $\rho\big(\N^1(S)^{K_S\leq 0}\big)=\N^1(X)^{K_X\leq 0}$.

Let $\cC\subset \Pi$ be a stability chamber. By Lemma \ref{det 2}, for any line bundle $L\in\cC$, $\rho_\cC(L)\in \Amp(M_\cC)$, so $\rho(\ov{\cC})\subseteq \Nef(M_\cC)\subseteq\Mov(M_\cC)=\Mov(X)$. It follows that $\rho(\Pi)\subseteq\Mov(X)^{K_X\leq0}$. 
In order to prove the reverse inclusion $\Nef(M_\cC)\subseteq \rho(\ov{\cC})$, it suffices to show that $\rho$ maps supporting hyperplanes of $\ov{\cC}$ onto supporting hyperplanes of $\Nef(M_\cC)$. We then conclude that $\rho(\ov{\cC})= \Nef(M_\cC)$. 
Similarly, in order to prove that $\rho(\Pi)=\Mov(X)^{K_X\leq0}$, it suffices to show that $\rho$ maps supporting hyperplanes of $\Pi$ onto supporting hyperplanes of $\Mov(X)^{K_X\leq0}$.

By Theorem~\ref{prop:special_loci} and the discussion in Subsection \ref{subsection:special_cones}, there are three types of supporting hyperplanes of $\ov{\cC}$ (here we are using the assumption that $\cC\subset \Pi$):
\begin{enumerate}
    \item[(a)] Hyperplanes of the form $(2D+K)^\perp$, where $D$ is a numerical rational class with $-n+1\leq D^2\leq0$.
    \item[(b)] Hyperplanes of the form $e^\perp$, where $e\in \cW\cdot e_1$.
    \item[(c)] The hyperplane $K^\perp$. 
\end{enumerate}
Similarly, by the description of $\Pi$ in (\ref{conePi}), there are three types of supporting hyperplanes of $\Pi$: those of type (a) above, with $D^2=0$ or $-n+1$, and those of types (b) and (c) above. We have already noted that $\rho(K^\perp)=K_X^\perp$. 

Consider a supporting hyperplane $(2D+K)^\perp$ of $\Nef(M_\cC)$ of type (a). 
After replacing $D$ with $-K-D$ if necessary, we may assume that $\cC\subset(2D+K)^{>0}$. 
In this case, we must have $-n+2\leq D^2\leq0$, and the subset $\PP_{-K-D}\subseteq M_{\cC}$ from Theorem~\ref{prop:special_loci} satisfies
$\dim(\PP_{-K-D})=n-1+D^2>0$.
By Proposition~\ref{det 3}, 
$$\rho_{\cC}\big((2D+K)^\perp\big)=l_{-K-D}^\perp\subset \N^1(M_{\cC}),$$ 
where $l_{-K-D}\in \N_1(M_{\cC})$ denotes the class of a line in $\PP_{-K-D}$. 
Moreover, by Proposition~\ref{reduction map}, there is an elementary contraction $\phi: M_{\cC}\ra M^{\mu ss}_{L_0}$ mapping $\PP_{-K-D}$ to a point. (Here $L_0\in (2D+K)^{\perp}$ is a polarization lying in the closure of $\cC$, and not
lying in any other wall.) Therefore, $l_{-K-D}^\perp$ as a supporting hyperplane of $\Nef(M_{\cC})$. 
If $D^2=0$, then $(2D+K)^\perp$ is also a supporting hyperplane of $\Pi$.  
In this case, $\phi: M_{\cC}\ra M^{\mu ss}_{L_0}$ is the blowup of a point with exceptional divisor $\PP_{-K-D}\cong \PP^{n-1}$, and thus $l_{-K-D}^\perp$ is a supporting hyperplane of $\Mov(X)^{K_X\leq0}$.

Consider now a supporting hyperplane $e^\perp$ of $\ov{\cC}$ of type (b), and note that $e^\perp$ is also a supporting hyperplane of $\Pi$ and $\cE$. Since $\Nef(M_\cC)\subseteq\Mov(X)^{K_X\leq0}\subseteq\Eff(X)^{K_X\leq0}$, it suffices to prove that $\rho(e^\perp)$ is a supporting hyperplane of $\Eff(X)^{K_X\leq0}$. 
Let $w\in\cW$ be such that $e=w\cdot e_1$. By Remark~\ref{X chamber and nearby}, $\rho$ is $\cW$-equivariant, and so $\rho(e^\perp)=w\cdot \rho(e_1^\perp)$. By Remark~\ref{rmk:cremona_action}, there is another blowup $X'$ of $\PP^n$ at a configuration of $n+4$ Cremona-general points and a small modification $\varphi:X'\dashrightarrow X$, such that $w\in\cW$ is the composition of the induced isomorphism $\varphi^*:\Pic(X)\to \Pic(X')$ with the identification of $\Pic(X')$ and $\Pic(X)$ obtained by identifying their natural bases $\{H', E'_1, \dots, E'_{n+4}\}$ and $\{H, E_1, \dots, E_{n+4}\}$. Therefore, it is enough to show that $\rho(e_1^\perp)$ is a supporting hyperplane of $\Eff(X)^{K_X\leq0}$. But this follows from the properties of $\rho$ observed in Remark~\ref{X chamber and nearby}, where we noted that $\rho(e_1^\perp)=(l-f_1)^\perp$. The class $l-f_1\in \N_1(X)$ represents the strict transform of a general line in $\PP^n$ through the blown up point $P_1$, and hence $(l-f_1)^\perp$ is a supporting hyperplane of $\Eff(X)^{K_X\leq0}$. This completes the proof of (1) and (2). 

Now we prove (3). Recall from Lemma~\ref{Lemma_E} that 
$$
\cE \ = \ \overline{\sum_{C\in \cW\cdot (h-e_1)}\RR_{\geq 0}C}.
$$
By Remark~\ref{X chamber and nearby}, $\rho(h-e_i)=2E_i$. Since $\rho$ is $\cW$-equivariant, we conclude that 
$$
\rho(\cE) \ = \ \overline{\sum_{E\in \cW\cdot E_1}\RR_{\geq 0}E}\ \subseteq \ \Eff(X)^{K_X\leq0}.
$$

In order to prove the equality $\rho(\cE)=\Eff(X)^{K_X\leq0}$, we proceed as before, showing that $\rho$ maps any supporting hyperplane of $\cE$ to a supporting hyperplane of $\Eff(X)^{K_X\leq0}$.  
By (\ref{coneE}), there are three types of supporting hyperplanes of $\cE$:
\begin{enumerate}
    \item[(a)] Hyperplanes of the form $(2B+K)^\perp$, where $B\in \cW\cdot h$.
    \item[(b)] Hyperplanes of the form $e^\perp$, where $e\in \cW\cdot e_1$.
    \item[(c)] The hyperplane $K^\perp$. 
\end{enumerate}
We have already noted that $\rho(K^\perp)=K_X^\perp$, and that $\rho(e^\perp)$ is a supporting hyperplane of $\Eff(X)^{K_X\leq0}$ for every $e\in \cW\cdot e_1$.
So it remains to consider hyperplanes of the form $(2B+K)^\perp$, where $B\in \cW\cdot h$. Arguing as we did earlier when we considered hyperplanes of the form $e^\perp$, 
we reduce to considering a unique supporting hyperplane $(2h+K)^\perp$.
By Remark~\ref{X chamber and nearby}, $\rho\big((2h+K)^\perp\big)=l^\perp\subset \N^1(X)$, where $l$ denotes the class of the strict transform of a general line in $\PP^{n}$. This is a supporting hyperplane of $\Eff(X)^{K_X\leq0}$, and this finishes the proof of (3).  
\ep

\begin{proof}[{Proof of Theorem~\ref{thm:main}}]
Part (1) and the first statement of part (2) follow immediately from Theorem~\ref{main}. 
It remains to prove that, for any stability chamber $\cC$ such that $\ov{\cC}\setminus\{0\}\subseteq\Pi^{K<0}$,
$\ov{\cC}$ is finitely generated.

For any  $a\in\RR_{>0}$, the intersection of $\Nef(S)$ with the hyperplane $H_a:=\{\gamma\in\N^1(S)\ |\ h\cdot\gamma=a\}$ is a compact set.
Let $\cC$ be a chamber such that $\ov{\cC}\setminus\{0\}\subseteq\Pi^{K<0}$, set $\ov{\cC}_a:=\ov{\cC}\cap H_a$, and note that $\ov{\cC}_a$ is also compact.
In order to prove that $\ov{\cC}$ is a polyhedral cone, it suffices to prove that,
for any $\alpha\in \ov{\cC}_a$, there is an open neighborhood $\cU$ of $\alpha$ such that $\cU$ intersects only finitely many hyperplanes of the following types:
\bi
\item[(a) ] $(2D+K)^\perp$, where $D$ is a numerical rational class with $-n\leq D^2\leq1$.
\item[(b) ] $e^\perp$, where $e\in\cW\cdot e_1$. 
\ei
For hyperplanes of type (b), the statement follows from the Cone theorem, since the rays $\RR_{\geq0}{e}$ are locally discrete in the half space $\N^1(S)^{K<0}$. So we only have to consider hyperplanes of type (a), i.e., stability walls. 
Let $L\in \Nef(S)^{K<0}$.
If $L^2>0$, then the result follows from Proposition~\ref{locally finite}.
If $L^2=0$, then $L\in \RR_{>0}C$ for some $C\in \cW\cdot (h-e_1)$ by Corollary~\ref{FDomain}.
Note however that $\cW\cdot (h-e_1)\cap\Pi=\emptyset$, and hence $L\not\in \ov{\cC}$.
\end{proof}

In the special case when $k=n+4=9$, by Remark~\ref{generators of cone E for 9 points}, we have the following corollary of Theorem~\ref{main}.

\begin{corollary}
Let $X=\Bl_9\PP^5$. Then 
$$
\Eff(X) \ = \ \Eff(X)^{K\leq 0} \ = \  \RR_{\geq 0}(-K_X) + \sum_{E\in \cW\cdot E_1}\RR_{\geq 0}E.
$$
\end{corollary}

\section*{References}

\bibliographystyle{alpha}

\begin{biblist}



\bib{AM16}{article}{  
 AUTHOR = {Araujo, Carolina},
AUTHOR = {Massarenti, Alex},
     TITLE = {Explicit log {F}ano structures on blow-ups of projective
              spaces},
   JOURNAL = {Proc. Lond. Math. Soc. (3)},
    VOLUME = {113},
      YEAR = {2016},
    NUMBER = {4},
     PAGES = {445--473},
}

\bib{AC17}{article}{ 
AUTHOR = {Araujo, Carolina}, 
AUTHOR = {Casagrande, Cinzia},
     TITLE = {On the {F}ano variety of linear spaces contained in two
              odd-dimensional quadrics},
   JOURNAL = {Geom. Topol.},
    VOLUME = {21},
      YEAR = {2017},
    NUMBER = {5},
     PAGES = {3009--3045},
}

\bib{AFKM21}{article}{
    AUTHOR = {Araujo, Carolina}, 
    AUTHOR = {Fassarella, Thiago},
    AUTHOR = {Kaur, Inder},
    AUTHOR = {Massarenti, Alex},
     TITLE = {On automorphisms of moduli spaces of parabolic vector bundles},
   JOURNAL = {Int. Math. Res. Not. IMRN},
      YEAR = {2021},
    NUMBER = {3},
     PAGES = {2261--2283},
}
	
\bib{Batyrev-Popov}{article}{  
    AUTHOR = {Batyrev, Victor V.},
    AUTHOR = {Popov, Oleg N.},
  TITLE = {The {C}ox ring of a del {P}ezzo surface},
 BOOKTITLE = {Arithmetic of higher-dimensional algebraic varieties ({P}alo
              {A}lto, {CA}, 2002)},
    SERIES = {Progr. Math.},
    VOLUME = {226},
     PAGES = {85--103},
 PUBLISHER = {Birkh\"auser Boston, Boston, MA},
      YEAR = {2004},
}

\bib{Bauer}{article}{  
    author = {Bauer, S.},
    title = {Parabolic bundles, elliptic surfaces and $\text{SU}(2)$-representation spaces of genus zero {F}uchsian groups},
    journal = {Math. Ann.},
    year = {1991},
    volume = {290},
    pages={509--526},
}

\bib{BDP16}{article}{  
AUTHOR = {Brambilla, M. C.},
AUTHOR = {Dumitrescu, O.}, 
AUTHOR = {Postinghel, E.},
     TITLE = {On the effective cone of {$\Bbb P^n$} blown-up at {$n+3$}
              points},
   JOURNAL = {Exp. Math.},
    VOLUME = {25},
      YEAR = {2016},
    NUMBER = {4},
     PAGES = {452--465},
}

\bib{BDPS}{article}{  
AUTHOR = {Brambilla, M. C.},
AUTHOR = {Dumitrescu, O.}, 
AUTHOR = {Postinghel, E.},
AUTHOR = {Santana S\'anchez, L. J.},
     TITLE = {Birational geometry of blowups via {W}eyl chamber decompositions and actions on curves},
   JOURNAL = {arXiv preprint arXiv:2410.18008},
    VOLUME = {},
      YEAR = {2024},
    NUMBER = {},
     PAGES = {},
}

\bib{CCF_v1}{article}{  
AUTHOR = {Casagrande, Cinzia},
AUTHOR = {Codogni, Giulio}, 
AUTHOR = {Fanelli, Andrea},
     TITLE = {The blow-up of {$\mathbb P^4$} at 8 points and its {F}ano
              model, via vector bundles on a del {P}ezzo surface},
   JOURNAL = {arXiv:1707.09152v1},
    VOLUME = {},
      YEAR = {2017},
    NUMBER = {},
     PAGES = {},
}

\bib{CCF}{article}{  
AUTHOR = {Casagrande, Cinzia},
AUTHOR = {Codogni, Giulio}, 
AUTHOR = {Fanelli, Andrea},
     TITLE = {The blow-up of {$\mathbb P^4$} at 8 points and its {F}ano
              model, via vector bundles on a del {P}ezzo surface},
   JOURNAL = {Rev. Mat. Complut.},
    VOLUME = {32},
      YEAR = {2019},
    NUMBER = {2},
     PAGES = {475--529},
}



\bib{CT}{article}{
AUTHOR = {Castravet, Ana-Maria},
AUTHOR = {Tevelev, Jenia},
TITLE = {Hilbert's 14th problem and {C}ox rings},
   JOURNAL = {Compos. Math.},
    VOLUME = {142},
      YEAR = {2006},
    NUMBER = {6},
     PAGES = {1479--1498},
}

\bib{Coble1}{article}{
    AUTHOR = {Coble, A.},
     TITLE = {Point sets and allied {C}remona groups. {I}},
   JOURNAL = {Trans. Amer. Math. Soc.},
    VOLUME = {16},
      YEAR = {1915},
    NUMBER = {2},
     PAGES = {155--198},
}

\bib{Coble2}{article}{
    AUTHOR = {Coble, A.},
     TITLE = {Point sets and allied {C}remona groups. {II}},
   JOURNAL = {Trans. Amer. Math. Soc.},
    VOLUME = {17},
      YEAR = {1916},
    NUMBER = {3},
     PAGES = {345--385},
}


\bib{Coble3}{article}{
    AUTHOR = {Coble, A.},
     TITLE = {Point sets and allied {C}remona groups. {III}},
   JOURNAL = {Trans. Amer. Math. Soc.},
    VOLUME = {18},
      YEAR = {1917},
    NUMBER = {3},
     PAGES = {331--372},
}






\bib{deF}{article}{
AUTHOR = {de Fernex, Tommasso},
TITLE = {On the {M}ori cone of blow-ups of the plane},
JOURNAL = {arXiv preprint  arXiv:1001.5243},
    VOLUME = {},
      YEAR = {2010},
    NUMBER = {},
     PAGES = {},
}

\bib{Debarre}{book}{
   author={Debarre, Olivier},
   title={Higher-dimensional algebraic geometry},
   series={Universitext},
   publisher={Springer-Verlag, New York},
   date={2001},
   pages={xiv+233},
}
    
\bib{Dolgachev}{incollection}{
    AUTHOR = {Dolgachev, Igor V.},
     TITLE = {Weyl groups and {C}remona transformations},
 BOOKTITLE = {Singularities, {P}art 1 ({A}rcata, {C}alif., 1981)},
    SERIES = {Proc. Sympos. Pure Math.},
    VOLUME = {40},
     PAGES = {283--294},
 PUBLISHER = {Amer. Math. Soc., Providence, RI},
      YEAR = {1983},
}


\bib{DO}{article}{
    AUTHOR = {Dolgachev, Igor V.},
    AUTHOR = {Ortland, David},
     TITLE = {Point sets in projective spaces and theta functions},
   JOURNAL = {Ast\'erisque},
    NUMBER = {165},
      YEAR = {1988},
     PAGES = {210},
}

\bib{DUrso}{article}{  
AUTHOR = {D'Urso, Luíze},
     TITLE = {On the Nef cones of blowups of the projective plane},
   JOURNAL = {arXiv preprint arXiv:2412.15460},
    VOLUME = {},
      YEAR = {2024},
    NUMBER = {},
     PAGES = {},
}

\bib{Ellingsrud_Gottsche}{article}{  
    AUTHOR = {Ellingsrud, Geir},
    AUTHOR = {G\"ottsche, Lothar},
     TITLE = {Variation of moduli spaces and {D}onaldson invariants under change of polarization},
   JOURNAL = {J. Reine Angew. Math.},
   VOLUME = {467},
      YEAR = {1995},
     PAGES = {1--49},
}

\bib{EP}{article}{
AUTHOR = {Eisenbud, David},
AUTHOR = {Popescu, Sorin},
     TITLE = {The projective geometry of the {G}ale transform},
   JOURNAL = {J. Algebra},
    VOLUME = {230},
      YEAR = {2000},
    NUMBER = {1},
     PAGES = {127--173},
}

\bib{G}{article}{  
AUTHOR = {G\"ottsche, Lothar},
     TITLE = {Change of polarization and {H}odge numbers of moduli spaces of
              torsion free sheaves on surfaces},
   JOURNAL = {Math. Z.},
    VOLUME = {223},
      YEAR = {1996},
    NUMBER = {2},
     PAGES = {247--260},
}

\bib{FQ}{article}{  
    AUTHOR = {Friedman, Robert},
    AUTHOR = {Qin, Zhenbo},
     TITLE = {Flips of moduli spaces and transition formulas for {D}onaldson
              polynomial invariants of rational surfaces},
   JOURNAL = {Comm. Anal. Geom.},
    VOLUME = {3},
      YEAR = {1995},
    NUMBER = {1-2},
     PAGES = {11--83},
      ISSN = {1019-8385,1944-9992},
}

\bib{HK}{article}{     
AUTHOR = {Hu, Yi},
AUTHOR = {Keel, Sean},
     TITLE = {Mori dream spaces and {GIT}},
   JOURNAL = {Michigan Math. J.},
    VOLUME = {48},
      YEAR = {2000},
     PAGES = {331--348},
}

\bib{HL}{book}{  
AUTHOR = {Huybrechts, Daniel},
AUTHOR = {Lehn, Manfred},
     TITLE = {The geometry of moduli spaces of sheaves},
    SERIES = {Cambridge Mathematical Library},
   EDITION = {Second},
 PUBLISHER = {Cambridge University Press, Cambridge},
      YEAR = {2010},
}

\bib{KM}{book}{
    AUTHOR = {Koll\'{a}r, J\'{a}nos},
    AUTHOR = { Mori, Shigefumi},
     TITLE = {Birational geometry of algebraic varieties},
    SERIES = {Cambridge Tracts in Mathematics},
    VOLUME = {134},
 PUBLISHER = {Cambridge University Press, Cambridge},
      YEAR = {1998},
     PAGES = {viii+254},
      ISBN = {0-521-63277-3},
   MRCLASS = {14E30},
       DOI = {10.1017/CBO9780511662560},
       URL = {https://doi.org/10.1017/CBO9780511662560},
}

\bib{LO16}{article}{     
AUTHOR = {Lesieutre, John},
AUTHOR = {Ottem, John Christian},   
     TITLE = {Curves disjoint from a nef divisor},
   JOURNAL = {Michigan Math. J.},
    VOLUME = {65},
      YEAR = {2016},
    NUMBER = {2},
     PAGES = {321--332},
}

\bib{Mukai01}{article}{  
AUTHOR = {Mukai, Shigeru},
TITLE = {Counterexample to {Hi}lbert's fourteenth problem for the 3-dimensional additive group},
JOURNAL = {RIMS preprint no. 1343},
year={2001},
  URL = {https://www.kurims.kyoto-u.ac.jp/preprint/file/RIMS1343.pdf},
} 

\bib{Mukai}{article}{  
AUTHOR = {Mukai, Shigeru},
TITLE = {Finite generation of the Nagata invariant rings in the A-D-E cases},
JOURNAL = {RIMS preprint},
year={2006},
  URL = {https://www.kurims.kyoto-u.ac.jp/preprint/file/RIMS1502.pdf},
} 




\bib{MW}{article}{
   AUTHOR = {Matsuki, Kenji},
   AUTHOR = {Wentworth, Richard},
     TITLE = {Mumford-{T}haddeus principle on the moduli space of vector
              bundles on an algebraic surface},
   JOURNAL = {Internat. J. Math.},
    VOLUME = {8},
      YEAR = {1997},
    NUMBER = {1},
     PAGES = {97--148},
 }


\bib{maruyama}{article}{
  AUTHOR = {Maruyama, Masaki},
  TITLE = {Stable vector bundles on an algebraic surface},
  JOURNAL = {Nagoya Mathematical Journal},
  VOLUME = {58},
  PAGES = {25--68},
  YEAR = {1975},
  PUBLISHER = {Cambridge University Press}
}

\bib{Nagata}{article}{
    AUTHOR = {Nagata, Masayoshi},
     TITLE = {On the {$14$}-th problem of {H}ilbert},
   JOURNAL = {Amer. J. Math.},
    VOLUME = {81},
      YEAR = {1959},
     PAGES = {766--772},
}

\bib{Ottem_master}{article}{  
AUTHOR = {Ottem, John},
TITLE = {Cox rings of projective varieties},
JOURNAL = {Master Thesis},
year={2009},
  URL = {https://www.duo.uio.no/bitstream/handle/10852/10769/Master.pdf},
} 

\bib{SX23}{article}{  
AUTHOR = {Stenger, Isabel},
AUTHOR = {Xie, Zhixin}, 
     TITLE = {Cones of divisors on {$\PP^3$} blown up at eight very general points},
   JOURNAL = {arXiv preprint arXiv:2303.12005},
    VOLUME = {},
      YEAR = {2023},
    NUMBER = {},
     PAGES = {},
}




\end{biblist}

\end{document}